\documentclass[reqno]{amsart}
\usepackage{amsfonts,amsmath,amssymb,dsfont}
\usepackage{color}
\usepackage[utf8]{inputenc}  
\usepackage[T1]{fontenc}  
\usepackage{graphicx} 
\numberwithin{equation}{section}

\usepackage[colorinlistoftodos]{todonotes}

\usepackage{xcolor}

\newcommand{\R}{\mathbb{R}}
\newcommand{\e}{\varepsilon}
\newcommand{\ve}{\varepsilon}
\newcommand{\Id}{\mathrm{Id}}
\newcommand{\dist}{\mathrm{dist}}

\newcommand{\N}{\mathbb{N}}
\newcommand{\<}{\left<}
\renewcommand{\>}{\right>}

\renewcommand{\(}{\left(}
\renewcommand{\)}{\right)}

\newcommand{\dmu}{{\delta_{\gamma}}}

\newcommand{\pmx}{\varphi_{\gamma,x}}

\newcommand{\pmxi}{\varphi_{\gamma,x_i}}

\newcommand{\pmxj}{\varphi_{\gamma,x_j}}
\newcommand{\pms}{\varphi_{\gamma,\sigma}}

\newcommand{\rmu}{r_{\gamma}}

\newtheorem{thm}{Theorem}[section]
\newtheorem*{thm*}{Theorem}
\newtheorem{lem}{Lemma}[section]

\newtheorem{cor}{Corollary}[section]
\newtheorem*{cor*}{Corollary}
\newtheorem{prop}{Proposition}[section]

\newtheorem{rem}{Remark}[section]

\newtheorem{Step}{Step}[section]

\begin{document}
\title[Critical points of the Moser-Trudinger functional]{Critical points of the Moser-Trudinger functional on closed surfaces}

\begin{abstract}
Given a closed Riemann surface $(\Sigma,{g_0})$ {and any positive weight $f\in C^\infty(\Sigma)$}, we use a minmax scheme together with compactness, quantization results and with sharp energy estimates to prove the existence of positive critical points of the functional
$${I_{p,\beta}}(u)=\frac{2-p}{2}\left(\frac{p\|u\|_{H^1}^2}{2\beta} \right)^{\frac{p}{2-p}}-\ln \int_\Sigma \(e^{u_+^p}-1\) {f}\, dv_{{g_0}}\,,$$
for every $p\in (1,2)$ and $\beta>0$, {or} for $p=1$ and $\beta\in (0,\infty)\setminus 4\pi\mathbb{N}$. Letting $p\uparrow 2$ we obtain positive critical points of the Moser-Trudinger functional
$$F(u):=\int_\Sigma \(e^{u^2}-1\){f}\,dv_{{g_0}}$$
constrained to $\mathcal{E}_\beta:=\left\{v\text{ s.t. }\|v\|_{H^1}^2=\beta\right\}$ for any $\beta>0$.  
\end{abstract}

\author[F. de Marchis]{Francesca De Marchis}
\address[Francesca De Marchis]{Dipartimento di Matematica Guido Castelnuovo, Universit\`a La Sapienza, Piazzale Aldo Moro 5, 00185 Roma, Italy}
\email{demarchis@mat.uniroma1.it}

\author[A. Malchiodi]{Andrea Malchiodi}
\address[Andrea Malchiodi]{Scuola Normale Superiore, Piazza dei Cavalieri 7, 56126 Pisa, Italy}
\email{andrea.malchiodi@sns.it}

\author[L. Martinazzi]{Luca Martinazzi}
\address[Luca Martinazzi]{Dipartimento di Matematica Guido Castelnuovo, Universit\`a La Sapienza, Piazzale Aldo Moro 5, 00185 Roma, Italy}
\email{luca.martinazzi@uniroma1.it}

\author[P.-D. Thizy]{Pierre-Damien Thizy}
\address[Pierre-Damien Thizy]{Institut Camille Jordan, Universit\'e Claude Bernard Lyon 1, B\^atiment Braconnier, 21 avenue Claude Bernard, 69622 Villeurbanne Cedex, France}
\email{pierre-damien.thizy@univ-lyon1.fr}

\maketitle

\section*{Introduction}
We consider a smooth, closed Riemann surface $(\Sigma,{g_0})$ ($2$-dimensional, connected and without boundary) {and a smooth positive function $f$}, and we endow the usual Sobolev space $H^1=H^1(\Sigma)$ with the {standard} norm $\|\cdot\|_{H^1}$ given by
\begin{equation}\label{H1Norm}
\|u\|_{H^1}^2=\int_{\Sigma} \left(|\nabla u|^2_{g_0}+u^2\right) dv_{g_0}\,.
\end{equation}
Building up on previous works, see e.g.  \cite{CarlesonChang, Flucher, Fontana, Judovic, MoserIneq, Poho65, StruweCrit, TrudingerOrlicz}, Yuxiang Li \cite{MTIneqSurface} proved that the following Moser-Trudinger inequality holds
\begin{equation}\tag{$MT$}\label{MTSurf}
\sup_{u\in H^1\,,~\|u\|_{H^1}^2=\beta } \int_\Sigma e^{u^2} {f} dv_{g_0}<+\infty\quad \Leftrightarrow\quad \beta \le 4\pi\,,
\end{equation}
(see also Remark \ref{RkZeroAv}) and that there is an extremal function for \eqref{MTSurf} even in the critical case $\beta=4\pi$ (see also Remark \ref{RkExtremals}). Such an extremal is (up to a sign change) a positive critical point of 
\begin{equation}\label{MTFunctional}
F(u):=\int_\Sigma \(e^{u^2} -1\){f}\,dv_{g_0}\,, 
\end{equation}
constrained to 
\begin{equation}\label{Constraint}
u\in \mathcal{E}_\beta:=\left\{v \in H^1 \text{ s.t. }\|v\|_{H^1}^2=\beta\right\}\,
\end{equation}
when $\beta\in (0,4\pi]$. A positive function $u$ is a critical point of $F$ constrained to $\mathcal{E}_\beta$ if and only if it satisfies the Euler-Lagrange equation
\begin{equation}\label{ELMainEqf}
\Delta_{g_0} u+u= 2\lambda {f}u e^{u^2}\,,\quad u>0\text{ in }\Sigma\,,
\end{equation}
where our convention for the Laplacian is with the sign that makes it nonnegative and where $\lambda>0$ is given by
\begin{equation}\label{FormulaLambdaPEq2}
{2 \lambda \int_\Sigma u^2 e^{u^2} f dv_{g_0}=\beta}=\|u\|^2_{H^1}\,.
\end{equation}
For $\beta<4\pi$, finding critical points of $F$ constrained to $\mathcal{E}_\beta$ reduces to a standard maximization argument. Finding such critical points for larger $\beta$'s is a  more challenging problem, since upper bounds on the functional fail, and this will be the main achievement of this paper. Some results in this direction, for planar domains and in {\em slightly supercritical regimes} {$0< \beta-4\pi \ll 1$} were obtained in  \cite{StruweCrit} and \cite{LRS}.

In order to do {handle the case of general $\beta's$ greater than $4\pi$}, we would like to use a variational method, more precisely a minmax method, to produce a converging Palais-Smale sequence. The two main analytic difficulties are that the functional $F$ does not satisfy the Palais-Smale condition and that its criticality is of borderline type, which prevents us from using the methods of \cite{LRS,StruweCrit} for $\beta$ large. To overcome these problems we will introduce a family of subcritical functional ${I_{p,\beta}}$, $p\in [1,2)$, that, in some sense, interpolate between a Liouville-type problem and our critical Moser-Trudinger problem, apply the minmax method to obtain critical points of ${I_{p,\beta}}$, and then prove new compactness and quantization results for such critical points.

More precisely, given $p\in [1,2)$ and $\beta>0$, we let ${I_{p,\beta}}$ be given in $H^1$ by
\begin{equation}\label{Energyf}
{I_{p,\beta}}(u)=\frac{2-p}{2}\left(\frac{p\|u\|_{H^1}^2}{2\beta} \right)^{\frac{p}{2-p}}-\ln \int_\Sigma \(e^{u_+^p}-1\) {f} dv_{g_0}\,,
\end{equation}
where $u_+=\max\{u,0\}$ and we set ${I_{p,\beta}}(u)=+\infty$ if $u\le 0$. By Trudinger's result \cite{TrudingerOrlicz}, for $p\in (1,2)$, ${I_{p,\beta}}$ is finite and of class $C^1$ on the subset of $H^1$ of  functions with non-trivial positive part, and its critical points are the solutions of
\begin{equation}\label{MainEquationf}
\Delta_{g_0} u+u=p \lambda   {f} u^{p-1} e^{u^p}  \,,\quad u>0\text{ in }\Sigma\,,
\end{equation}
where the positivity follows from the maximum principle, see Lemma \ref{l:pos}, and $\lambda>0$ is given by the relation 
\begin{equation}\label{FormulaLambdaf}
\frac{\lambda p^2}{2} \left(\frac{p \|u\|^2_{H^1}}{2\beta} \right)^{\frac{2(p-1)}{2-p}} \int_\Sigma \(e^{u^p}-1\) {f} dv_{g_0}=\beta\,.
\end{equation}
While {$I_{1,\beta}$} is not differentiable at functions $u$ vanishing on sets of positive measure, it is differentiable at any $u>0$ a.e., and $u>0$ is a critical point if and only if it solves \eqref{MainEquationf}{-\eqref{FormulaLambdaf}} with $p=1$. Smoothness follows by standard elliptic theory and \cite{TrudingerOrlicz}, see Lemma \ref{l:pos}. {Now, multiplying \eqref{MainEquationf} by $u$ and integrating by parts in $\Sigma$, \eqref{FormulaLambdaf} may be rewritten as 
\begin{equation}\label{AlternativeFormulaBetaf}
\frac{\lambda p^2}{2} \left( \int_\Sigma  \(e^{u^{p}}-1\) f dv_{g_0}\right)^{\frac{2-p}{p}}\left(  \int_\Sigma u^{p} e^{u^{p}} f dv_{g_0}\right)^{\frac{2(p-1)}{p}}=\beta\,.
\end{equation}}
By \eqref{MTSurf} and Young's inequality, ${I_{p,\beta}}$ is bounded from below for all $\beta\le 4\pi$, and for $\beta<4\pi$ finding critical points of ${I_{p,\beta}}$ reduces to a standard minimization argument, as it happens for the constrained extremization of $F$: similarly, finding such critical points for larger $\beta$'s is  much more difficult. As we shall discuss, compactness and quantization  (see Corollary \ref{CorThmBlowUpAnalysis}) give that, as $p$ approaches the borderline case $p_0=2$, the critical points of ${I_{p,\beta}}$ converge to critical points of the functional $F$ in \eqref{MTFunctional} constrained to $\mathcal{E}_\beta$,
at least when $\beta>0$ is given out of $4\pi \mathbb{N}^\star$, where $\mathbb{N}^\star$ denotes the set of the positive integers. 

Our main results read as follows:

\begin{thm}\label{MainThm2}
Let $(\Sigma,{g_0})$ be a {smooth} closed surface {and $f$ be a smooth positive function}. Let $p\in (1,2)$ and $\beta>0$ be given. Then the set $\mathcal{C}_{p,\beta}$ of the positive critical points of ${I_{p,\beta}}$ is not empty and compact. The same is true for $p=1$ and every $\beta \in (0,\infty)\setminus 4\pi\mathbb{N}^\star$.
\end{thm} 

Letting $p\uparrow 2$ suitably, we will obtain the following result, which according to us is the  most relevant achievement of this paper.

\begin{thm}\label{MainThm}
Let $(\Sigma,{g_0})$ be a {smooth closed surface and $f$ be a smooth positive function. Let $\beta>0$ be given.} Then the set $\mathcal{C}_{2,\beta}$ of the positive critical points of the functional $F$ constrained to $\mathcal{E}_\beta$ is not empty and compact in $C^2$.
\end{thm}

A notable fact in Theorems \ref{MainThm2} and \ref{MainThm} is that, except for $p=1$, \emph{the full range $\beta>0$ is covered}  and in particular also the case $\beta\in 4\pi \mathbb{N}^\star$. If fact we will also prove that the sets
$$\bigcup_{ {\beta\in [4\pi(k-1)+\delta,4\pi k]}\atop {p\in [1+\delta,2]}} {\mathcal{C}_{p,\beta}}, \quad \bigcup_{{\beta\in [4\pi(k-1)+\delta,4\pi k-\delta]}\atop{p\in [1,2]}} {\mathcal{C}_{p,\beta}}$$
are compact for any $\delta>0$, i.e. blow-up can occur only for $\beta\downarrow 4\pi \mathbb{N}^\star$ or for $p\to 1$ and $\beta\to 4\pi \mathbb{N}^\star$, as we shall see.

Let us explain the strategy of the proofs. We shall start with the existence part of Theorem \ref{MainThm2}. Here with a minmax scheme based on so called \emph{baricenters}, as originally used in \cite{DjaMal}, we show that given $p\in (1,2)$ and $\beta \in (4 \pi, + \infty) \setminus 4 \pi \N^\star$, the very low sublevels of {$I_{p,\beta}$} are topologically non-trivial, see Proposition \ref{prop:homoteq}. This would allow to construct a Palais-Smale sequence at some minmax level, but it is only with a \emph{monotonicity trick } introduced by Struwe, see \cite{StruweTrick}, that we are able to construct Palais-Smale sequences that are bounded for \emph{almost every} $\beta>0$ and for $p\in (1,2)$. Then, again using the \emph{subcriticality} of $e^{u^p}$ with respect to \eqref{MTSurf}, a $H^1$-bounded subsequence {strongly} converges to a positive critical point of {$I_{p,\beta}$}, see Proposition \ref{PrConvergence} (see also \cite[Thm. 5.1]{CostaTintarev} for counterexamples to the {strong} convergence of bounded Palais-Smale sequences when $p=2$).

The next step is to extend this result from the existence for a.e. $\beta$ to the existence for \emph{every} $\beta\in (0,\infty)\setminus 4\pi \mathbb{N}^\star$. This is done via the crucial compactness Theorem \ref{ThmBlowUpAnalysis},  showing that a sequence $(u_\ve)_\varepsilon$ of positive critical points of $J_{p_\ve,\beta_\ve}$ with $p_\ve \in [1,2)$ and $\beta_\ve\to \beta\in [0,\infty)$ can fail to be precompact only if $\beta \in 4\pi \mathbb{N}$. If fact, as $p_\ve\uparrow 2$, this also allows to show that the positive critical points of $J_{p_\ve,\beta_\ve}$ converge to positive critical points of $F|_{\mathcal{E}_\beta}$ if $\beta\not\in 4\pi \mathbb{N}$, (see Corollary \ref{CorThmBlowUpAnalysis}), hence proving Theorem \ref{MainThm}, except for $\beta\in 4\pi\mathbb{N}^\star$. This quantization property ($\beta\in 4\pi \mathbb{N}^\star$ in case of blow up) can be seen as a no-neck energy result, but not only. Indeed, in the specific case where $p=2$, extending the quantization of \cite{DruetDuke} to the surface setting, Yang \cite{YangQuantization} already proved a no-neck energy result for such sequences, but without excluding that some nonzero weak limit $u_0\not \equiv 0$ appears. We know now that ruling this situation out, or in other words getting the sharp quantization \eqref{Quantization} instead of \eqref{BadQuantization}, is a very sensitive property, which depends also on the lower-order terms appearing in the RHS of \eqref{ELMainEq} (see for instance \cite{ThiMan2} for counterexamples with a perturbed version of the nonlinearity $e^{u^2}$) and which requires to be  more careful in the way we approach the border case $p=2$. In this sense, our Theorem \ref{ThmBlowUpAnalysis} cannot be seen as a perturbation of previous results, but it is a novelty in itself. We also mention that the proof of Theorem \ref{ThmBlowUpAnalysis} never uses the Pohozaev identity, which is however quite classical in proving such quantization results. {Instead, we first compare in small disks our blow-up solutions with some radially symmetric functions solving the same PDE, sometimes called "bubbles", and we directly show that the difference must satisfy some balance condition (see \eqref{GradNull2}). From this balance condition, we get that the union of these separate disks is large in the sense that the complementary region cannot contribute in the quantization \eqref{Quantization}. In this last part of the proof, we also show that our {specific} family of nonlinearities forces the Lagrange multipliers to converge appropriately to $0$ (see Step \ref{StSubcriticalCase}) as blow-up occurs. One delicate consequence is that each disk only brings the minimal energy $4\pi$ in \eqref{Quantization} (see also Remark \ref{RemTower}).}

Finally, covering the case $\beta\in 4\pi\mathbb{N}^\star$ relies on delicate energy expansions of the blowing-up sequences carried out in Theorem \ref{ThmCritLevels} below. When $\beta=4\pi$ and $p=2$, it was already observed in a slightly different setting (see \cite{MalchMartJEMS,MartMan}) that such expansions do not clearly depend on the geometric quantities of the problem and that the energy always converges to $4\pi$ from above. In the present paper, we observe that this is still true at any level $\beta\in 4\pi \mathbb{N}^\star$ and for all $p\in (1,2]$, so that if we let $\beta_\ve \uparrow \beta\in 4\pi\mathbb{N}^\star$ no blow-up occurs, while it could occur for $\beta_\ve\downarrow \beta\in 4\pi\mathbb{N}^\star$. In striking contrast, the analogous expansions in  \cite{ChenLinSharpEst}, dealing with an equation qualitatively similar to the case $p=1$ (see Remark \ref{RkZeroAv} below), are different in nature: for instance, the Gauss curvature of the surface appears and compactness is not always true at critical levels $\beta \in 4\pi \mathbb{N}^\star$ (see the discussion below \cite[Corollary 1.2]{ChenLinSharpEst}). 

\medskip

{We conclude this introduction with some remarks.}

\begin{rem}\label{RemCompactPathPEq1To2}
When $\Sigma$ is a non-simply connected bounded domain in $\R^2$, in \cite{DruMalMarThi} the authors compute the Leray-Schauder degree of the Euler-Lagrange equation of the functional $F|_{\mathcal{E}_\beta}$, showing that it is non-zero if $\Sigma$ is not simply connected. Even if we were able to adapt the strategy to the case of a closed manifold $\Sigma$, when the genus of $\Sigma$ is $0$ (i.e. if $\Sigma$ is topologically a sphere), the Leray-Schauder degree of the Euler-Lagrange equation is expected to be $1$ for $\beta\in (0,4\pi]$, $-1$ for $\beta\in (4\pi,8\pi]$ and $0$ for $\beta > 8\pi$. Hence this topological method fails to completely answer the question of the existence of critical points of $F|_{\mathcal{E}_\beta}$ on a closed surface. 

In any case, \textbf{the Leray-Schauder degree does not depend on $p\in [1,2]$} by compactness (except for $p=1$ and $\beta\in 4\pi \mathbb{N}^\star$), and coincides with that of the \emph{mean field equation} (with the full $H^1$-norm, slightly different from \cite{ChenLin-Liouville} or \cite{DegreeLiouvMalch}), namely \eqref{MainEquation} $p=1$. For the case $p\in (1,2]$ and $\beta =4\pi k$ the L-S degree is equal to the degree for $\beta \in (4\pi (k-1),4\pi k)$ by Theorem \ref{ThmCritLevels}.
\end{rem}

\begin{rem}\label{RkZeroAv}
It is worth mentioning that, on a surface, there is a Moser-Trudinger inequality with a zero average constraint, namely
\begin{equation}\label{MTIneqZeroAv}\tag{$MT_{\mathcal{Z}}$}
\sup_{u\in \mathcal{Z}_\beta} \int_\Sigma e^{u^2} dv_{{g_0}}<+\infty\quad \Leftrightarrow \quad \beta \le 4\pi\,,
\end{equation}
where $\mathcal{Z}_\beta=\left\{u\in H^1\text{ s.t. }\int_\Sigma |\nabla u|_{g_0}^2 dv_{g_0}= \beta\text{ and }\int_\Sigma u~ dv_{g_0}=0 \right\}$. This inequality was already proven in the original paper of Moser \cite{MoserIneq}, if $(\Sigma,g)$ is the standard $2$-sphere, and in the general case by Fontana \cite{Fontana}, dealing also with the higher dimensional case. The functional $I_\beta$, qualitatively 
related to {$I_{p,\beta}$} in \eqref{Energyf} for $p=1$,
\begin{equation}\label{EnergyZeroAv}
I_\beta(u)=\frac{1}{4\beta } \int_\Sigma |\nabla u|^2_{{g_0}} dv_{{g_0}}+\int_\Sigma u~ dv_{{g_0}}-\ln \int_\Sigma e^u dv_{{g_0}}
\end{equation}
attracted a huge attention in the literature (see \cite{YYLi,ChenLin-Liouville,Djadli} and references therein) and its critical points give rise to the very much studied mean-field equation. As a remark, for all $\beta\le 4\pi$, as \eqref{MTSurf} implies that {$I_{1,\beta}$} is bounded below, we get from \eqref{MTIneqZeroAv} that $I_\beta$ is bounded below. 
\end{rem}

{
\begin{rem}
In the papers \cite{BarLin}, \cite{CCL}, \cite{LinLucia}, \cite{Suz92} 
some uniqueness results  for Liouville equations in planar domains 
or on closed surfaces were proved, while in \cite{Dem08}, \cite{Dem10} 
some multiplicity results as well. It would be worthwhile to consider such 
issues for the critical points of the Moser-Trudinger functional as well. 
\end{rem}
}

\begin{rem}
Different kinds of bubbling solutions for the Moser-Trudinger inequalities on domains and surfaces were built in \cite{DelPNewSol,DengMusso,FigueroaMusso}.
\end{rem}

{
\section*{Preliminaries}
It is convenient to get rid of the smooth weight function $f$ and to reformulate the problem by introducing the norm $\|\cdot\|_h$ given by
\begin{equation}\label{defnormh}
\|u\|_h^2=\int_{\Sigma} \left(|\nabla u|_g^2+h u^2\right) dv_g\,,
\end{equation}
where $h:=1/f\in C^\infty(\Sigma)$ and where the new metric $g$ is conformal to $g_0$ and given by $g=f g_0$. Keeping then the notation in \eqref{H1Norm}, we have $\|u\|_h=\|u\|_{H^1}$ for all $u\in H^1(\Sigma)$. Besides, since $\Delta_g=\Delta_{fg_0}=f^{-1}\Delta_{g_0}$ by the \emph{conformal covariance of the Laplacian}, we obtain that $u$ solves \eqref{MainEquationf} if and only if it solves \begin{equation}\label{MainEquation}
\Delta_{g} u+hu=\lambda p u^{p-1}e^{u^p}\,,\quad u>0\text{ in }\Sigma\,.
\end{equation} 
Then, for all $p\in[1,2)$, the aforementioned critical points $u$ of $I_{p,\beta}$ solving \eqref{MainEquationf}-\eqref{FormulaLambdaf} are exactly those of the functional $J_{p,\beta}$ given by
\begin{equation}\label{Energy}
J_{p,\beta}(u)=\frac{2-p}{2}\left(\frac{p\|u\|_h^2}{2\beta} \right)^{\frac{p}{2-p}}-\ln \int_\Sigma \(e^{u_+^p}-1\) dv_g\,,
\end{equation}
solving \eqref{MainEquation} with $\lambda>0$ given by
\begin{equation}\label{FormulaLambda}
\frac{\lambda p^2}{2} \left(\frac{p \|u\|^2_h}{2\beta} \right)^{\frac{2(p-1)}{2-p}} \int_\Sigma \(e^{u^p}-1\) dv_g=\beta\,.
\end{equation}
Again, multiplying \eqref{MainEquation} by $u$ and integrating by parts, \eqref{FormulaLambda} may be rephrased as 
\begin{equation}\label{AlternativeFormulaBeta}
\beta=\frac{\lambda p^2}{2} \left( \int_\Sigma  \(e^{u^{p}}-1\) dv_g\right)^{\frac{2-p}{p}}\left(  \int_\Sigma u^{p} e^{u^{p}} dv_g\right)^{\frac{2(p-1)}{p}}\,.
\end{equation}
Now, even for $p=2$, we have that $u\in H^1$ solves our problem \eqref{ELMainEqf}-\eqref{FormulaLambdaPEq2}, if and only if we have \eqref{MainEquation} for $p=2$, namely
\begin{equation}\label{ELMainEq}
\Delta_{g} u+hu=2\lambda  u e^{u^2}\,,\quad u>0\text{ in }\Sigma\,,
\end{equation}
with $\lambda>0$ given by \eqref{AlternativeFormulaBeta}. 
\begin{rem} Working with \eqref{MainEquation} instead of \eqref{MainEquationf} will considerably simplify the choice of constants and the writing of some estimates in the blow-up analysis. Yet, if one consents to burden the presentation, a straightforward adaptation of our proofs can handle the case where two independent weights appear, namely for the equation $\Delta_g u +h u=p\lambda f u^{p-1} e^{u^p}$.
\end{rem}
}

\section{Variational part}\label{sec:variational}
\noindent The main goal of the section is to prove the following theorem, with $J_{p,\beta}$  
as in \eqref{Energy}. 

\begin{thm}\label{ThmVariationalPart}
Let $(\Sigma,g)$ be a closed surface, {a positive function $h\in C^{\infty}(\Sigma)$} and let $p\in (1,2)$ be given. Then, for almost every $\beta>0$, $J_{p,\beta}$ possesses a smooth and positive critical point $u$, where $J_{p,\beta}$ is as in \eqref{Energy}.
\end{thm}

As discussed in introduction, $u$ given by Theorem \ref{ThmVariationalPart} is smooth, positive and solves \eqref{MainEquation}-\eqref{FormulaLambda} for some $\lambda>0$, as we shall now prove.

\begin{lem}\label{l:pos} Every non-trivial critical point of $J_{p,\beta}$, $p\in (1,2)$, is a smooth and positive solution to \eqref{MainEquation}. Moreover, for every $p\in [1,2]$ every solution to \eqref{MainEquation} is smooth.
\end{lem}

\proof Assume $p\in (1,2)$. One easily verifies that the Euler-Lagrange equation of $J_{p,\beta}$ is
\begin{equation}\label{EqLpos1}
\Delta u +{h}u =\lambda u_+^{p-1} e^{u_+^p},
\end{equation}
where $\lambda>0$. Since $e^{u_+^{p}}\in L^q(\Sigma)$ for every $q\in [1,\infty)$ thanks to \cite{TrudingerOrlicz}, elliptic estimates imply that $u\in C^2(\Sigma)$.

We first claim that $u\ge 0$. Indeed, assume that $\Sigma_-:=\{x\in \Sigma:u(x)<0\}\neq\emptyset$. Then $\Delta u =-{h}u>0$ in $\Sigma_-$, violating the weak maximum principle at a point of minimum.

Now consider $\Sigma_+:=\{x\in \Sigma:u(x)>0\}$. We claim that $\Sigma_+=\Sigma$, i.e. $u>0$ everywhere. Otherwise $\partial \Sigma_+\ne \emptyset$. Let then $x_0\in \partial \Sigma_+$ be a point satisfying the interior sphere condition, and let $D\subset\Sigma_+$ be a disk with $x_0\in \partial D$ and such that
$$\Delta u= \lambda u^{p-1}e^{u^p}-{h}u >0\quad \text{in }D.$$
It is possible to find such $D$ because $u(x_0)=0$, $\lambda>0$, and $p<2$. Then, by the Hopf lemma, see e.g. \cite[Lemma 3.4]{Gilbarg},
$$
\frac{\partial u}{\partial \nu}(x_0)<0,
$$
where $\nu$ is the outer normal to $\partial \Sigma_+$ at $x_0$. This violates the non-negativity of $u$, leading to a contradition. Hence $u>0$. Going back to \eqref{EqLpos1}, we can now bootstrap regularity, hence $u\in C^\infty(\Sigma)$. 

Also for $p=1,2$ the regularity of solutions to \eqref{MainEquation} follows from elliptic estimates and \cite{TrudingerOrlicz}, which implies that the right-hand side of \eqref{MainEquation} belongs to $L^q(\Sigma)$ for $q\in [1,\infty)$.
\endproof

\medskip

In the rest of the section we consider $p\in (1,2)$ {and the positive function $h\in C^\infty(\Sigma)$} fixed. The first tools we shall need in the proof of Theorem \ref{ThmVariationalPart} are improved Moser-Trudinger inequalities. Let us first observe that from Young's inequality $ab\le \frac{a^q}{q}+\frac{b^r}{r}$ applied with $q=\frac{2}{p}$ and $r=q'=\frac{2}{2-p}$ we obtain{, for $u\not\equiv0$,}
\begin{align*}
|u|^p&= \(\frac{|u|}{\|u\|_{{h}}}\sqrt{\frac{8\pi}{p}}\)^p \(\|u\|_{{h}}\sqrt{\frac{p}{8\pi}}\)^p\\
& \le 4\pi\frac{u^2}{\|u\|_{{h}}^2}+\|u\|_{{h}}^\frac{2p}{2-p}\frac{2-p}{2}\(\frac{p}{8\pi}\)^\frac{p}{2-p}, 
\end{align*}
hence with \eqref{MTSurf} we get
\begin{equation}\label{MTp0}
\ln\int_\Sigma \left(e^{|u|^p}-1\right)dv_g\le {\ln\int_\Sigma e^{|u|^p}\,dv_g \le} \frac{2-p}{2}  \(\frac{p\|u\|_{{h}}^2}{8\pi}\)^\frac{p}{2-p}+C.
\end{equation}
It follows that
$$J_{p,\beta}(u)\ge \frac{2-p}{2}\(p\|u\|_{{h}}^2\)^\frac{p}{2-p}\(\(\frac{1}{2\beta}\)^\frac{p}{2-p}- \(\frac{1}{8\pi}\)^\frac{p}{2-p}\)-C,$$
so that $J_{p,\beta}$ is coercive for $\beta<4\pi$.

On the other hand, if the density $e^{|u|^p}-1$ is \emph{spread} into $k+1\ge 1$ disjoint regions we have the following improved Moser-Trudinger inequality which gives a uniform {lower} bound on $J_{p,\beta}(u)$ for each $\beta<4\pi (k+1)$, see 
\cite{CL91} for a related argument.

\begin{lem}\label{l:imprc}
For any fixed $k\in\mathbb{N}$, let $\Omega_1, \dots, \Omega_{k+1}$
be subsets of $\Sigma$ satisfying
$\dist(\Omega_i,\Omega_j) \geq \delta_0$ for $i \neq j$ and some $\delta_0>0$. Let also $\gamma_0 \in \left( 0,
\frac{1}{k+1} \right)$, $\delta_1\in(0,8\pi(k+1))$. Then there
exists a constant $C = C(k,\delta_0, \delta_1, \gamma_0,\Sigma)$ such that
\begin{equation}\label{eq:MTimpr}
 \ln\int_\Sigma  \left(e^{|u|^p}-1\right)dv_g\le \frac{2-p}{2}\(\frac{p \|u\|^2_{{h}}}{8\pi(k+1)-\delta_1}\)^\frac{p}{2-p} +C
 \end{equation}
for all the functions $u \in H^1(\Sigma){\setminus\{0\}}$ satisfying
\begin{equation}\label{eq:ddmm}
    \frac{\int_{\Omega_i}  \left(e^{|u|^p}-1\right)dv_g}{\int_\Sigma  \left(e^{|u|^p}-1\right)dv_g} \geq \gamma_0,
  \quad  \forall \; i \in \{1, \dots, k+1\}.
\end{equation}
\end{lem}

\proof Fix $u$ satisfying \eqref{eq:ddmm}.
We can find
$k + 1$ functions $g_1, \dots, g_{k+1}$ such that
\begin{equation}\label{eq:gi}
    \left\{%
\begin{array}{ll}
    g_i(x) \in [0,1] & \hbox{ for every } x \in \Sigma; \\
    g_i(x) = 1, & \hbox{ for every } x \in \Omega_i; \\
    g_i(x) = 0, & \hbox{ if } \dist(x,\Omega_i) \geq \frac{\delta_0}{2}; \\
    \|g_i\|_{C^1} \leq C_{\delta_0,\Sigma}. &  \\
\end{array}%
\right.
\end{equation}
For $\ve > 0$ small (to be fixed depending on $k$ and $\delta_1$) using the inequality $2ab\le \ve a^2+\ve^{-1}b^2$ we can find a constant $C_{\ve,\delta_0}$ (the dependence of the constants on $\Sigma$ {and $h$} will be omitted) such that, for any $i \in \{1, \dots, k + 1\}$ and $v\in H^1(\Sigma)$ there
holds
\begin{equation}\label{eq:pgiv}
\begin{split}
    \|g_i v\|_{{h}}^2& \leq \int_{\Sigma} g_i^2  |\nabla v|^2 dv_g
    + \ve \int_{\Sigma} |\nabla v|^2 dv_g + C_{\ve,\delta_0} \int_\Sigma v^2 dv_g.
\end{split}
\end{equation}
Now let $\lambda_{\e,\delta_0}$ be an eigenvalue of {$\Delta_g+{h}$} such
that $\frac{C_{\e,\delta_0}}{\lambda_{\e,\delta_0}} < \ve$, where $C_{\e,\delta_0}$ is as in \eqref{eq:pgiv}, and write
$$
 u = P_{V_{\e,\delta_0}} u+ P_{V_{\e,\delta_0}^\perp}u =:u_1+u_2, 
$$
where $V_{\e,\delta_0}\subset H^1(\Sigma)$ is the direct sum of the eigenspaces of {$\Delta_g+{h}$}
with eigenvalues less than or equal to $\lambda_{\e,\delta_0}$, and
$P_{V_{\e,\delta_0}}, P_{V_{\e,\delta_0}^\perp}$ denote the projections onto
$V_{\e,\delta_0}$ and $V_{\e,\delta_0}^\perp$ respectively.

We now choose $i$ such that
$$\int_{\Sigma} g_i^2  |\nabla u_2|^2 dv_g\le \int_{\Sigma} g_j^2  |\nabla u_2|^2 dv_g \quad \text{for every }j \in \{1, \dots, k + 1\}.$$
Since the
functions $g_1, \dots, g_{k+1}$ have disjoint supports, \eqref{eq:pgiv} applied with $v=u_2$ gives
$$
  \|g_i u_2\|_{{h}}^2 \leq \frac{1}{k + 1} \int_{\Sigma} |\nabla u_2|^2 dv_g
    + \e \int_{\Sigma} |\nabla u_2|^2 dv_g + C_{\ve,\delta_0} \int_\Sigma u_2^2 dv_g.
$$
This, together with the inequalities
$$
 C_{\e,\delta_0} \int_\Sigma u_2^2 dx \leq
  \frac{C_{\e,\delta_0}}{\lambda_{\e,\delta_0}} \|u_2\|_{{h}}^2 \le 
  \ve \|u_2\|_{{h}}^2,
$$
implies
\begin{equation}\label{eq:gigi1}
  \|g_i u_2\|_{{h}}^2 \leq \left( \frac{1}{k + 1} + 2 \e \right)
  \|u_2\|_{{h}}^2\leq \left( \frac{1}{k + 1} + 2 \ve \right)\|u\|_{{h}}^2.
\end{equation}
In particular from the Moser-Trudinger inequality \eqref{MTp0} and \eqref{eq:gigi1}, we have for $\ve$ small enough, which we now fix depending on $\delta_1$ and $k$,
\begin{equation}\label{eq:gigi2}
\begin{split}
 \ln \int_\Sigma e^{(1+\e) |g_i u_2|^p} dv_g &\leq \frac{2-p}{2}  \(\frac{p(1+\ve)^\frac2p \|g_i u_2\|_{H^1}^2}{8\pi}\)^\frac{p}{2-p}+C\\
 & \leq \frac{2-p}{2}  \(\frac{p\|u_2\|_{H^1}^2}{8\pi(k+1)-\delta_1}\)^\frac{p}{2-p}+C.
 \end{split}
\end{equation}
Notice also that since $V_{\e,\delta_0}$ is finite dimensional, we have
$$\|v\|_{L^\infty}\le \tilde C_{\ve,\delta_0}\|v\|_{L^2}\le \hat C_{\ve,\delta_0} \|v\|_{{h}}, \quad \text{for }v\in V_{\e,\delta_0},$$
hence
$$
 \|u_1\|_{L^\infty(\Omega)} \leq \hat{C}_{\ve,\delta_0}\|u_1\|_{{h}}.
$$
Now, using the inequality
$$(a+b)^p\le C_{\ve,p} a^p+ (1+\ve)b^p,$$
we get
$$  \int_\Sigma e^{|g_i u|^p} dv_g \leq e^{C_{\ve,p}\|u_1\|_\infty^p} \int_\Sigma e^{(1+\ve) |g_i u_2|^p}dv_g,$$
hence, from \eqref{eq:ddmm} and \eqref{eq:gigi2} we deduce 
\begin{equation}\label{eq:inteu2}
\begin{split}
  \ln \int_\Sigma \(e^{|u|^p}-1\)dv_g  &\leq \ln \frac{1}{\gamma_0}+\ln \int_{\Omega_i} \(e^{|u|^p}-1\)dv_g \\
  &\leq \ln\frac{1}{\gamma_0} +\ln \int_\Sigma
    e^{|g_i u|^p}dv_g \\
    &\le \ln\frac{1}{\gamma_0} + C_{\ve, p} \|u_1\|_{L^\infty}^p+ \ln \int_\Sigma
    e^{(1 + \ve) |g_i u_2|^p}dv_g \\
    & \le \frac{2-p}{2}  \(\frac{p\|u_2\|_{{h}}^2}{8\pi(k+1)-\delta_1}\)^\frac{p}{2-p}+{\tilde C_{\ve,p}}\|u_1\|_{{h}}^p+C',
\end{split}
\end{equation}
with $C' = C'(k,\delta_0, \delta_1, \gamma_0,\Sigma)$. A further application of  Young's inequality to the term ${\tilde C_{\ve,p}}\|u_1\|_{{h}}^p$ and the inequality $a^q + b^q \leq (a+b)^q$ for $q > 1$ then gives
$$\ln \int_\Sigma \(e^{|u|^p}-1\)dv_g \le \frac{2-p}{2}  \(\frac{p(\|u_2\|_{H^1}^2+\|u_1\|_{{h}}^2)}{8\pi(k+1)-\delta_1}\)^\frac{p}{2-p} + C$$
with $C = C(k,\delta_0, \delta_1, \gamma_0,\Sigma)$, and since $\|u_2\|_{{h}}^2+\|u_1\|_{{h}}^2=\|u\|_{{h}}^2$ we conclude.
\endproof

The next lemma, proven in \cite[Lemma 2.3]{DjaMal}, is a criterion which implies the situation described by   condition  \eqref{eq:ddmm}.

\begin{lem}\label{l:criterion}
Let $k$ be a given positive integer, and consider $\e,r>0$. Suppose that for a non-negative function $f\in L^1(\Sigma)$ with $\|f\|_{L^1}=1$ there holds
\begin{equation}\label{eq:ddmm3}
\int_{\bigcup_{i=1}^k B_r(x_i)} f dx <1-\e\qquad\quad \textrm{for every $k$-tuple $x_1,\ldots,x_k\in \Sigma$.}
\end{equation}
Then there exist $\bar{\e}>0$ and $\bar{r}>0$, depending only on $\e$, $r$, $k$ and $\Omega$ (but not on $f$), and $k+1$ points $\bar{x}_{1,f},\ldots,\bar{x}_{k+1,f}\in \Sigma$ such that
$$
\int_{B_{\bar{r}(\bar{x}_{j,f})}} f dx\geq\bar{\e},\quad \text{for } j=1,\dots,k+1,$$
and $B_{2\bar{r}}(\bar{x}_{i,f})\cap B_{2\bar{r}}(\bar{x}_{j,f})=\emptyset$ for $i\neq j$.
\end{lem}

Lemma \ref{l:imprc} and Lemma \ref{l:criterion} then imply the following other result.

\begin{lem}\label{l:ls}
If $\beta\in(4\pi k,4\pi(k+1))$ with $k\geq1$, the following property holds. For any $\ve>0$ and any $r>0$ there exists a large positive constant $L=L(\ve,r,p,\beta)$ such that, for every $u\in H^1(\Sigma)$ with $J_{p,\beta}(u)\leq-L$ there exist $k$ points $x_{1},\ldots,x_{k}\in\Sigma$ such that
\begin{equation}\label{eq:ddmm2}
	\frac{\int_{\Sigma\setminus\cup_{i=1}^k B_r(x_{i})} \(e^{|u|^p}-1\) dv_g}{\int_\Sigma  \(e^{|u|^p}-1\) dv_g}<\ve.
\end{equation}
\end{lem}

\proof Fix $\ve$, $r$, $p$, and $\beta$ as in the statement of the lemma and let $u\in H^1(\Sigma)$ be such that $J_{p,\beta}(u)\le -L$ for some constant $L\ge 0$, and assume by that \eqref{eq:ddmm2} fails for every $k$-tuple of points $x_1,\dots,x_k$. Then setting
$$f:=\frac{e^{|u|^p}-1}{\|e^{|u|^p}-1 \|_{L^1}}$$
we have that \eqref{eq:ddmm3} holds.
Therefore, by Lemma \ref{l:criterion} we can find $\bar \ve=\bar \ve(\ve,r,k,\Sigma)$, $\bar r=\bar r(\ve,r,k,\Sigma)$ and points $\bar x_{1},\dots,\bar x_{k+1}\in \Sigma$ such that the assumptions of Lemma \ref{l:imprc} hold with $\Omega_i= B_{\bar r}(\bar x_{i})$, $\gamma_0=\bar \ve$ and $\delta_0=2\bar r$.
Fix also $\delta_1=8\pi(k+1)-2\beta$. Then by Lemma \ref{l:imprc} there exists a constant $\bar C$ depending on $k$, $\delta_0$, $\delta_1$, $\gamma_0$, $p$ and $\Sigma$, hence depending on $\ve$, $r$, $p$, $\beta$, $k$ and $\Sigma$ such that 
$$\ln \int_\Sigma \(e^{|u|^p}-1\)dv_g   \le \frac{2-p}{2}  \(\frac{p\|u\|_{H^1}^2}{2\beta}\)^\frac{p}{2-p}+\bar C,
$$
hence $J_{p,\beta}(u)\ge -\bar C$, and up to choosing $L>\bar C$ we obtain a contradiction, unless \eqref{eq:ddmm2} holds for a suitable $k$-tuple $x_{1},\dots, x_{k}\in \Sigma$.
\endproof

Given $k\in \N$ we introduce the set of formal \emph{barycenters of} $\Sigma$ \emph{of order} $k$, namely
$$\Sigma_k=\left\{ \sigma=\sum_{i=1}^k t_i\delta_{x_i}:x_i\in \Sigma, \,t_i\ge 0,\, \sum_{i=1}^k t_i=1\right\},$$
 where $\delta_{x_i}$ is the Dirac mass at $x_i$, see \cite{DjaMal}, \cite{MalchiodiTopological}.\\ 
 We will see $\Sigma_k$ as a subset of $\mathcal M(\Sigma)$, the set of all probability Radon measures on $\Sigma$, endowed with the distance defined using duality versus Lipschitz functions: 
\begin{equation}\label{KR}
 \dist(\mu,\nu):=\sup_{\|{\psi}\|_{Lip(\Sigma)}\leq1}\left|\int_{\Sigma} {\psi}\,d\mu - \int_{\Sigma} {\psi}\,d\nu \right|,\qquad\quad\mu, \nu\in \mathcal M(\Sigma), 
\end{equation}
which receives the name of \emph{Kantorovich-Rubinstein distance}.

\begin{lem}\label{l:RK} For any $\ve>0$ there exist $\delta>0$ and $r_\ve>0$ such that, for any $r\in(0,r_\ve]$, if $f\in L^1(\Sigma)$ is a non-negative function such that
\begin{equation}\label{eq:ddmm4}
	\frac{\int_{\Sigma\setminus\cup_{i=1}^k B_r(x_{i})} f dv_g}{\int_\Sigma  f dv_g}<\delta
\end{equation}
for some $x_1,\dots,x_k\in \Sigma$, then
$$\dist\(\frac{f dv_g}{\int_\Sigma  f dv_g},\sigma\)<\ve,$$
where
$$\sigma=\sum_{i=1}^k t_i \delta_{x_i},\qquad t_i=\frac{\int_{B_r(x_i)}f dv_g}{\int_{\cup_{j=1}^k B_r(x_j)} fdv_g}.$$
\end{lem}

\begin{proof}
Consider a function ${\psi}$ on $\Sigma$ with $\|{\psi}\|_{Lip(\Sigma)}\leq 1$, which we can assume to have zero average, and let us estimate  
for $\int_{\Sigma} f dv_g = 1$ (otherwise, we can rescale $f$ by a constant)
$$
   \left|  \int_{\Sigma} f \, {\psi} \, dv_g - \int_{\Sigma} {\psi} \, d \sigma \right| \leq \sum_{i = 1}^k \left| 
   \int_{B_r(x_i)} f\, {\psi} \, dv_g - \int_{B_r(x_i)} {\psi} \, d \sigma \right| +  \left|  \int_{\Sigma \setminus \cup_{i=1}^k B_r(x_i)} f\, {\psi} \, dv_g  \right|.  
$$
Since $f\geq0$, with  $\int_{\Sigma} f dv_g = 1$, and since ${\psi}$ is uniformly bounded by the diameter of $\Sigma$ (due to the fact that it is $1$-Lipschitz and 
has zero average), by \eqref{eq:ddmm4} we clearly have that 
$$
   \left|  \int_{\Sigma \setminus \cup_{i=1}^k B_r(x_i)} f\, {\psi} \, dv_g  \right|  \leq \delta \, diam_g(\Sigma). 
$$
On the other hand, for the same reason we have that $\int_{\cup_{j=1}^k B_r(x_j)} fdv_g = 1 + O(\delta)$, which 
implies that $t_i = (1 + O(\delta)) \int_{B_r(x_i)}f dv_g$ and in turn that 
\begin{equation*}
\begin{split}
  \int_{B_r(x_i)} {\psi} d \sigma & = t_i \Psi(x_i)\\
  &=  {\psi}(x_i)  \int_{B_r(x_i)}f dv_g + O(\delta).  
\end{split}
\end{equation*}
Again from the fact that ${\psi}$ is $1$-Lipschitz, we get that 
$$
\int_{B_r(x_i)} f\, {\psi} \, dv_g = {\psi}(x_i) \int_{B_r(x_i)}f dv_g + O(r). 
$$
Since ${\psi}$ was arbitrary, the conclusion follows from the last four formulas. 
\end{proof}

An immediate consequence of Lemma \ref{l:ls} and Lemma \ref{l:RK} is that the low sublevels of $J_{p,\beta}$ can be mapped close to $\Sigma_k$, in the sense of the following lemma.

\begin{lem}\label{l:ls2}  Given $\beta\in(4\pi k,4\pi(k+1))$ with $k\geq1$, $\ve>0$ there exists $L=L(\ve,p,\beta)$ such that for every $u\in H^1(\Sigma)$ with $J_{p,\beta}\le -L$ we have
$$\dist\(\frac{\(e^{|u|^p}-1\)dv_g}{\int_\Sigma \(e^{|u|^p}-1\)dv_g}, \Sigma_k\)<\ve.$$
\end{lem}

 Let us first recall a well known result about $\Sigma_k$, endowed with the topology induced by $\dist(\cdot, \cdot).$
 
 \begin{lem}[\cite{MalchiodiTopological}]\label{l:noncontr}
 	For any $k\geq 1$ the set $\Sigma_k$ is non-contractible.
 \end{lem}
 
Our goal is to show that, if $\beta\in(4\pi k,4\pi(k+1))$, $\Sigma_k$ can be mapped into very negative sublevels of $J_{p,\beta}$ and that this map is non trivial in the sense that it carries some homology. Then, as a consequence of the previous Lemma we will get the non contractibility of low sublevels of $J_{p,\beta}$. 

Let us first define the \emph{standard bubble} $\varphi_\gamma:\R^2\to\R$ for $\gamma>0$,
$$\varphi_\gamma(x):=\(\frac 2p\)^{\tfrac1p}\gamma\(1-\frac1{\gamma^p}\ln\(1+\frac{|x|^2}{r_\gamma^2}\)\)_+, $$
where  $r_\gamma$ is chosen so that 
\begin{equation}\label{rgammadef}
r_\gamma=o\(e^{-\gamma^p}\),\quad \ln\(r_\gamma e^{\gamma^p}\)=o(\gamma^p),
\end{equation}
for instance, $r_\gamma=\gamma^{-1}e^{-\gamma^p}$.

Now, given $x\in\Sigma$ we define the function $\pmx:\Sigma\to\R$ as
$$
\varphi_{\gamma,x}=\(\frac 2p\)^{\tfrac1p}\gamma\(1-\frac1{\gamma^p}\ln\(1+\frac{d^2(y,x)}{r_\gamma^2}\)\)_+. 
$$
Notice that $\pmx(y)>0$ if and only if $y\in B_{\dmu}(x)$, where
\begin{equation}\label{deltamu}
\delta_\gamma^2:=\rmu^2(e^{\gamma^p}-1)\to 0 \quad \text{as }\gamma\to\infty.
\end{equation}
For a barycenter $\sigma=\sum_{i=1}^{k}t_i\delta_{x_i}\in\Sigma_k$ we now want to construct test functions $\varphi_{\gamma,\sigma}$ continuous with respect to $\sigma$ (from $\mathcal{M}(\Sigma)$ into $H^1(\Sigma)$) concentrating mass near the points $x_i$, in the sense that
\begin{equation}\label{convtest}
\frac{ (e^{\pms^p}-1)dv_g}{\int_\Sigma\(e^{\pms^p}-1\)dv_g}\to\sigma,\quad \text{as }\gamma\to\infty.
\end{equation}

In order to do so, to each $t\in [0,1]$ and $\gamma>0$ we associate $\tau=\tau(t,\gamma)$ such that
\begin{equation}\label{deftau}
\frac{\int_{\R^2}\(e^{(\varphi_\gamma-\tau)^p_+}-1\)dx}{\int_{\R^2}\(e^{\varphi_\gamma^p}-1\)dx}=t.
\end{equation}
Notice that $\tau$ is decreasing with respect to $t$  and that $\tau(0,\gamma)=\(\frac 2p\)^{\tfrac1p}\gamma$, $\tau(1,\gamma)=0$ for every $\gamma>0$. We will need the following elementary estimate.

\begin{lem}\label{lemma:tau} For any fixed $\overline t\in(0,1]$ we have $\tau(t,\gamma)\to 0$ as $\gamma\to\infty$ uniformly for $t\in [\overline t,1]$.
\end{lem}

\begin{proof}
Given $\gamma, \tau>0$, consider $L\ge \frac{2}{\gamma}$ to be fixed later. We easily see that
\begin{equation}\label{eqtau1}
\begin{split}
\int_{\{\varphi_\gamma< L\}}\(e^{(\varphi_\gamma-\tau)_+^p}-1\)dx&\le \int_{\{\varphi_\gamma< L\}}\(e^{\varphi_\gamma^p}-1\)dx\\
&\le \(e^{L^p} -1\){\pi}\delta_\gamma^2=o_\gamma(1).
\end{split}
\end{equation}
Moreover, for $\gamma$ such that
$$\frac{2}{p}\(1-\frac{\ln 2}{\gamma^p}\)^p\ge 1+\ve>1,$$
also using that $\varphi_\gamma\ge L$ on $B_{r_\gamma}(0)$ for $\gamma$ large, we get
\begin{equation}\label{eqtau2}
\begin{split}
\int_{\{\varphi_\gamma\ge  L\}}\(e^{\varphi_\gamma^p}-1\)dx&\ge \int_{B_{r_\gamma}(0)}\(e^{\frac{2}{p}\gamma^p(1-\frac{\ln 2}{\gamma^p})^p}-1\)dx\\
&\ge\int_{B_{r_\gamma}(0)}\(e^{(1+\ve)\gamma^p}-1\)dx\\
&\ge \pi r_\gamma^2 \(e^{(1+\ve)\gamma^p}-1\)\to\infty.
\end{split}
\end{equation}
By the Taylor expansion
$$(1-x)^p=1-px+\frac{p(p-1)}{2(1-\xi)^{2-p}}x^2\le 1-px+C_px^2,\quad 0\le \xi\le x\le \frac{1}{2}, $$
we get for $\varphi_\gamma\ge L\ge 2\tau$
$$(\varphi_\gamma-\tau)^p_+\le \varphi_\gamma^p -p\tau\varphi_\gamma^{p-1}+C_p\tau^2\varphi_\gamma^{p-2}\le \varphi_\gamma^p-\frac{p}{2}\tau\varphi_\gamma^{p-1}$$
up to choosing $L\ge L_0(p)$ sufficiently large. 
We then infer
\begin{equation}\label{eqtau3}
\begin{split}
\int_{\{\varphi_\gamma\ge  L\}}\(e^{(\varphi_\gamma-\tau)_+^p}-1\)dx&\le  e^{-\frac{p}{2}\tau L^{p-1}} \int_{\{\varphi_\gamma\ge  L\}}e^{\varphi_\gamma^p}dx\\
&=o_L\(\int_{\R^2}\(e^{\varphi_\gamma^p}-1\)dx\),\quad \text{as }L\to\infty.
\end{split}
\end{equation}
Putting \eqref{eqtau1}-\eqref{eqtau3} together it follows that
$$t(\tau,\gamma):=\frac{\int_{\R^2}\(e^{(\varphi_\gamma-\tau)^p_+}-1\)dx}{\int_{\R^2}\(e^{\varphi_\gamma^p}-1\)dx}=o(1)\quad \text{as }\gamma\to\infty
$$
for any $\tau>0$. This implies that $\tau(\bar t,\gamma)=o(1)$ as $\gamma\to\infty$ for any $\bar t\in (0,1]$ since otherwise there would be sequences $\gamma_\ve\to 0$ and $\tau_\ve\in (0,\gamma_\ve]$ such that $\tau_\ve(\bar t,\gamma_\ve)\ge\tau_*>0$, and by monotonicity
$$0<\bar t=t(\tau_\ve,\gamma_\ve)\le  t(\tau_*,\gamma_\ve)=o(1)\quad \text{as }\ve\to 0,$$
a contradiction. 
Using the monotonicity of $\tau$ with respect to $t$ the conclusion follows at once.
\end{proof}

Now call $\tau_i=\tau(t_i,\gamma)$, $1\le i\le k$ and define $\pms$ by the formula 
$$e^{\pms^p}-1=\sum_{i=1}^k \left(e^{(\pmxi-\tau_i)^p_+}-1\right),$$
or, explicitly
\begin{equation}\label{pms}
\pms =\ln^\frac{1}{p}\left(1+\sum_{i=1}^k \left(e^{(\pmxi-\tau_i)_+^p} -1\right)\right).
\end{equation}

\noindent \emph{Notation.} Until the end of the section $o(1)$ (resp. $O(1)$) will denote a quantity tending to $0$ (resp. a bounded quantity) as $\gamma\to\infty$, uniformly with respect to $x\in \Sigma$ and $\sigma\in\Sigma_k$.

\begin{lem}\label{lemma:gradpmx} For every $x\in \Sigma$,  we have
$$
\int_\Sigma |\nabla \pmx|^2\,dv_g=\left(\frac2p\right)^{\frac2p}4\pi\gamma^{2-p}(1+o(1)),\quad \text{as }\gamma\to\infty.
$$
\end{lem}

\begin{proof}
By a straightforward computation, for any $y\in B_{\delta_\gamma}(x)$, we get 
$$
\nabla\pmx(y)=-\left(\frac2p\right)^{\frac1p}\gamma^{1-p}\:\frac{\rmu^{-2}\nabla_y(d^2(y,x))}{1+\rmu^{-2}d^2(y,x)}, 
$$
while $\nabla \pmx(y)=0$ in $\Sigma\setminus B_{\delta_\gamma}(x)$.\\
Using geodesic coordinates centered at $x$, with an abuse of notation, we identify the points in $\Sigma$ with their
pre-image under the exponential map. Using these coordinates, and recalling that $\delta_\gamma\to 0$ we have that $$d(y,x)=|y-x|(1+o(1)),\quad |\nabla_y(d^2(y,x))|=2|y-x|(1+o(1)),\quad y\in B_{\delta_\gamma}(x)$$
hence
\begin{equation}\label{stimagrad}
|\nabla \pmx(y)|=\left(\frac2p\right)^{\frac1p}\gamma^{1-p}(1+o(1))\frac{2|y-x|}{r_\gamma^2+|y-x|^2},\quad y\in B_{\delta_\gamma}(x).
\end{equation}
Thanks to the change of variable $s=r_\gamma^2+\rho^2$, we are able to conclude that 
\begin{equation*}
\begin{split}
\int_\Sigma |\nabla \pmx|^2\,dv_g&=\int_{B_{\delta_\gamma}(x)} |\nabla \pmx|^2\,dv_g\\
&=\left(\frac2p\right)^{\frac2p}\gamma^{2-2p}(1+o(1))\int_{B^{\R^2}_{\delta_\gamma}(x)}\frac{4|y-x|^2}{(r_\gamma^2+|y-x|^2)^2}dy\\
&=\left(\frac2p\right)^{\frac2p}4\pi\gamma^{2-2p}(1+o(1))\int_0^{\dmu}\frac{2\rho^3}{(r_\gamma^2+\rho^2)^2}d\rho\\
&=\left(\frac2p\right)^{\frac2p}4\pi\gamma^{2-2p}(1+o(1))\int^{r_\gamma^2 e^{\gamma^p}}_{r_\gamma^2}\(\frac{1}{s}-\frac{r_\gamma^2}{s^2}\)ds\\
&=\left(\frac2p\right)^{\frac2p}4\pi\gamma^{2-p}(1+o(1)), 
\end{split}
\end{equation*}
yielding the result. 
\end{proof}

\begin{lem}\label{lemma:gradptms} In the above notation we have, uniformly for $\sigma \in \Sigma_k$
$$\int_\Sigma |\nabla \pms|^2\,dv_g \le \left(\frac2p\right)^{\frac2p}4\pi k \gamma^{2-p}(1+o(1)). $$
\end{lem}

\begin{proof} We compute
$$\nabla \pms =\frac{\sum_{i=1}^k (\pmxi-\tau_i)_+^{p-1}e^{(\pmxi- \tau_i)_+^p}\nabla \pmxi}{\ln^\frac{p-1}{p}\(1+\sum_{j=1}^k \(e^{(\pmxj-\tau_j)_+^p}-1\) \)\left(1+\sum_{j=1}^k \(e^{(\pmxj-\tau_j)_+^p}-1\)\right)}.$$
Notice that
\begin{equation*}
\begin{split}
0&\le \frac{(\pmxi-\tau_i)_+^{p-1}e^{(\pmxi- \tau_i)_+^p}}{\ln^\frac{p-1}{p}\(1+\sum_{j=1}^k \(e^{(\pmxj-\tau_j)_+^p}-1\) \)\left(1+\sum_{j=1}^k \(e^{(\pmxj-\tau_j)_+^p}-1\)\right)}\\
&\le a_i \chi_{\{\pmxi>\tau_i\}},
\end{split}
\end{equation*}
where
$$a_i:=\frac{e^{(\pmxi- \tau_i)_+^p}}{1+\sum_{j=1}^k \(e^{(\pmxj-\tau_j)_+^p}-1\)},$$
hence
$$|\nabla \pms(x)| \le \sum_{i=1}^k a_i(x)|\nabla \pmxi(x)|\chi_{\{\pmxi>\tau_i\}}(x).$$
Split now $\Sigma$ as a disjoint (up to sets of measure zero) union $\Omega_1\cup\dots\cup\Omega_k$, such that
$${|\nabla \pmxj(x)|=\max_{1\le i\le k}|\nabla \pmxi(x)|\quad \text{for }x\in \Omega_j,}$$
and further split $\Sigma$ as  $\Sigma=\Sigma_+\cup\Sigma_-$,
where 
$$\Sigma_+:=\left\{x\in\Sigma: \sum_{j=1}^k e^{(\varphi_{{\gamma,x_j}}(x)-\tau_j)^p_+} \ge \gamma\right \},\quad  \Sigma_-:=\Sigma\setminus \Sigma_+.$$
Notice that
$$\sum_{i=1}^{k}a_i(x)\le 1+o_\gamma(1)\quad \text{for }x\in \Sigma_+.$$

Then, with the help of Lemma \ref{lemma:gradpmx} we obtain 
\begin{equation}\label{eq:lemma1.8}
\begin{split}
\int_{\Sigma_+} |\nabla \pms|^2 dx&\le \sum_{j=1}^k \int_{\Sigma_+\cap \Omega_j}\(\sum_{i=1}^k a_i|\nabla \pmxj|\)^2dx\\
&\le (1+o(1))\sum_{j=1}^k\int_{\Sigma_+}|\nabla\pmxj|^2dx\\
&\le (1+o(1))\left(\frac2p\right)^{\frac2p}4\pi k \gamma^{2-p}.
\end{split}
\end{equation}
We now want to prove that the integral over $\Sigma_-$ is negligible. Indeed we have
$$\sum_{i=1}^{k}a_i(x)\le k \quad \text{for }x\in \Sigma_-,$$
since $ \frac{s}{s-k+1}\le k$ for $s\ge k$, and similarly to \eqref{eq:lemma1.8} we get
$$\int_{\Sigma_-} |\nabla \pms|^2 dx \le k^2 \sum_{j=1}^k\int_{\Sigma_-\cap \Omega_j}|\nabla\pmxj|^2dx.$$ 
In order to estimate the right-hand side, observe that
$$1\le e^{(\varphi_j(x)-\tau_j)_+^p}\le \gamma \quad \text{for }x\in \Sigma_-.$$
This implies that
$$\Sigma_-\cap \Omega_j \subset B_{R_1}(x_j)\setminus B_{r_1}(x_j)\quad \text{for every }j,$$
where $R_1$ and $r_1$ are given by the relations
$$1 \le e^{\(C_p\gamma\(1-\frac1{\gamma^p}\ln\(1+\frac{d^2(x,x_j)}{r_\gamma^2}\)\)-\tau_j\)^p } \le \gamma,\quad C_p:=\(\frac 2p\)^{\tfrac1p}. $$
This yields
$$\gamma^p-C_p^{-1}\tau_j\gamma^{p-1}\ge \ln\(1+\frac{d^2(x,x_j)}{r_\gamma^2}\) \ge \gamma^p-C_p^{-1}\tau_j \gamma^{p-1}-\gamma^{p-1}\ln^\frac1p \gamma,$$
and
$$R_1^2= \(e^{\gamma^p-C_p^{-1}\tau_j\gamma^{p-1}}-1\)r_\gamma^2, \quad r_1^2= \(e^{\gamma^p-C_p^{-1}\tau_j\gamma^{p-1}-\gamma^{p-1}\ln^\frac1p\gamma}-1\)r_\gamma^2.$$
We now integrate as in Lemma \ref{lemma:gradpmx}, and with the same change of variables $s=r_\gamma^2+\rho^2$ we obtain
\begin{equation*}
\begin{split}
\int_{B_{R_1}(x_j)\setminus B_{r_1}(x_j)} |\nabla \pmxj|^2 dv_g
&=O\(\gamma^{2-2p}\)\int_{r_\gamma^2+r_1^2}^{r_\gamma^2+R_1^2}\frac{s-r_\gamma^2}{s^2}ds\\
&\le O\(\gamma^{2-2p}\)\int^{r_\gamma^2 e^{\gamma^p-C_p^{-1}\tau_j\gamma^{p-1}}}_{r_\gamma^2 e^{\gamma^p-C_p^{-1}\tau_j\gamma^{p-1}-\gamma^{p-1}\ln^\frac1p\gamma}}\frac{ds}{s}\\
&=O\(\gamma^{1-p}\ln^\frac1p\gamma\)\\
&=o(\gamma^{2-p}).
\end{split}
\end{equation*}
Together with \eqref{eq:lemma1.8}, we conclude.
\end{proof}

\begin{lem}\label{lemma:L2} We have the following estimates, uniformly for $\sigma \in \Sigma_k$
$$
\int_\Sigma {h} \pms^2\,dv_g=o(\gamma^{2-p}).
$$
\end{lem}

\begin{proof}
Let us first evaluate, for $x\in \Sigma$, $\int_{\Sigma}\pmx^2\,dv_g$. Being
$$\int_{B_{\rmu}(x)}\pmx^2\,dv_g=o(1),\quad \pmx=0 \text{ in }\Sigma\setminus B_{\delta_\gamma}(x),$$
it is enough to estimate $\int_{B_{\delta_\gamma}(x)\setminus B_{\rmu}(x)}\pmx^2\,dv_g$.\\
Using normal coordinates  at $x$ and the change of variables $s=1+\frac{r^2}{\rmu^2}$, we get
\begin{equation}\label{opiccolo}
\begin{split}
\int_{\Sigma\setminus B_{\rmu}(x)}\pmx^2 dv_g&=O(\gamma^2)\int_{\rmu}^{\delta_\gamma} r\left(1-\frac{2}{\gamma^p}\ln(1+\frac{r^2}{\rmu^2})+\frac{1}{\gamma^{2p}}\ln^2(1+\frac{r^2}{\rmu^2})\right)dr \\
&=O(\gamma^2)\int_{2}^{e^{\gamma^p}} \rmu^2\left(1-\frac{2}{\gamma^p}\ln(s)+\frac{1}{\gamma^{2p}}\ln^2(s)\right)ds\\
&=O(\gamma^2\rmu^2) \left[s-\frac{2}{\gamma^p}(-s+s\ln s)+\frac{1}{\gamma^{2p}}(2s-2s\ln s+s\ln^2 s)\right]_{2}^{e^{\gamma^p}}\\
&=O(\gamma^{4-4p})\\
&=o(\gamma^{2-p}).
\end{split}
\end{equation}
Splitting  $\Sigma$ as a disjoint (up to sets of measure zero) union $\tilde\Omega_1\cup\dots\cup\tilde\Omega_k$, so that
$$\pmxi(x)=\max_{1\le j\le k} \pmxj(x)\quad \text{for }x\in \tilde\Omega_j,$$
we have
\begin{eqnarray*}
\pms^2(x)&\leq& \ln^{\frac2p}\left(\sum_{i=1}^k e^{\pmxi^p(x)}\right)\leq \sum_{j=1}^k \chi_{\tilde\Omega_j}(x)\ln^{\frac2p}\left(e^{\pmxj^p(x)}\right)\\
&\leq& \sum_{j=1}^k \left(\ln k+\pmxj^p(x)\right)^{\frac2p}\\
&\leq& O(1)+O(1)\sum_{j=1}^k \pmxj^2(x),
\end{eqnarray*}
where in the last inequality we  used the convexity of the map $t\mapsto t^{\frac2p}$.

As a consequence{, since $h$ is bounded,}
$$
\int_\Sigma {h} \pms^2\,dv_g=O(1)+O(1)\sum_{j=1}^k \int_\Sigma \pmxj^2(x)\,dv_g\overset{\eqref{opiccolo}}{=}o(\gamma^{2-p}), 
$$
as desired. 
\end{proof}

\begin{lem}\label{lemma:expterm}
We have, uniformly for $\sigma \in \Sigma_k$
$$
\ln\int_\Sigma \(e^{\pms^p}-1\)dv_g\geq \frac {2-p}p \gamma^p (1+o(1)),\quad \text{as }\gamma\to\infty.
$$
\end{lem}
\begin{proof}
Given $\sigma=\sum_{i=1}^k t_i\delta_{x_i}\in\Sigma_k$, fix $i$ such that $t_i\ge\frac{1}{k}$. Then, according to Lemma \ref{lemma:tau} we have $\tau_i=o(1)$ as $\gamma\to\infty$, hence
$$\pms^p\ge (\pmxi-\tau_i)_+^p \ge \frac{2}{p}\gamma^p\(1-\frac{\ln 2}{\gamma^p}-o_\gamma(1)\)_+^p \ge \frac{2}{p}\gamma^p(1+o(1)) \quad \text{on }B_{r_\gamma}(x_i)$$
for $\gamma$ sufficiently large.
Then, also using \eqref{rgammadef}, it follows
\begin{equation*}
\begin{split}
\ln\int_\Sigma \(e^{\pms^p}-1\)dv_g&\ge \ln\int_{B_{r_\gamma}(x_i)} \(e^{\pms^p}-1\)dv_g\\
&\ge \ln\((1+o(1))\pi r_\gamma^2 e^{\frac{2}{p}\gamma^p(1+o(1))} \)\\
&=\(\frac{2}{p}-1+o(1)\)\gamma^p\\
&=\frac{2-p}{p}\gamma^p(1+o(1)),
\end{split}
\end{equation*}
as claimed. 
\end{proof}

\begin{lem}
	\label{lemma:lowsublevels}
	Given $\beta\in(4\pi k,4\pi(k+1))$, with $k\geq1$, then as $\gamma\to+\infty$ we have:

	\begin{itemize}
		\item[\emph{i.}] $J_{p,\beta}(\varphi_{\gamma,\sigma})\to-\infty$ uniformly for $\sigma\in\Sigma_k$,
		\item[\emph{ii.}] $\dist\left(\frac{\(e^{\varphi_{\gamma,\sigma}^p}-1\)dv_g}{\int_\Sigma \(e^{\varphi_{\gamma,\sigma}^p}-1\) dv_g},\sigma \right)\to 0$  uniformly for $\sigma\in\Sigma_k$,  see \eqref{KR}.	\end{itemize}
\end{lem}

\begin{proof}
\emph{i.}\quad By definition of $\pms$ and {Lemmas \ref{lemma:gradptms}, \ref{lemma:L2} and \ref{lemma:expterm}}
 we have
\begin{equation*}
\begin{split}
J_{p,\beta}(\varphi_{\gamma,\sigma})&=\frac{2-p}{2}\(\frac{p\|\pms\|^2_{{h}}}{2\beta}\)^{\frac{p}{2-p}}-\ln\int_\Sigma \(e^{\pms^p}-1\)dv_g\\
&\leq \frac{2-p}{2}\(\frac{p (\frac 2p)^{\frac 2p} 4\pi k \gamma^{2-p}(1+o(1))}{2\beta}\)^\frac{p}{2-p}-\frac {2-p}p \gamma^p (1+o(1))\\
&=\frac{2-p}{2}\left[\(\frac{4\pi k}{\beta}\)^\frac{p}{2-p}\frac 2p \gamma^p(1+o(1))\right]-\frac {2-p}p \gamma^p (1+o(1))\\
&=\frac{2-p}{p}\gamma^p\left[\(\frac{4\pi k}{\beta}\)^\frac{p}{2-p}-1\right](1+o(1))\to -\infty, 
\end{split}
\end{equation*}
uniformly for $\sigma\in\Sigma_k$.

\emph{ii.} Let us first collect some simple calculations.\\
Let $\sigma=\sum_{i=1}^k t_i\delta_{x_i}\in\Sigma_k$: then, since $\delta_\gamma\to0$ when $\gamma\to+\infty$,
\begin{eqnarray}\label{fact1}
\int_{B_{\delta_{\gamma}}(x_i)}\left(e^{(\pmxi-\tau_i)_+^p} -1\right)\,dv_g&\!\!\!\!\!=&\!\!\!\!\!(1+o(1))\int_{B^{\R^2}_{\delta_{\gamma}}(0)}\left(e^{(\varphi_{\gamma}-\tau_i)_+^p} -1\right)\,dx\nonumber\\
\!\!\!\!\!&\!\!\!\!\!\overset{\eqref{deftau}}{=}&\!\!\!\!\!(1+o(1))\,t_i\int_{\R^2}\left(e^{\varphi_{\gamma}^p} -1\right)\,dx,
\end{eqnarray}
as a consequence
\begin{eqnarray}\label{fact2}
\int_{\cup_{j=1}^k B_{\delta_\gamma}(x_j)}(e^{\varphi_{\gamma,\sigma}^p}-1)\,dv_g&\overset{\eqref{pms}}{=}&\int_{\cup_{j=1}^k B_{\delta_\gamma}(x_j)}\left(\sum_{i=1}^k \left(e^{(\pmxi-\tau_i)_+^p} -1\right)\right)\,dv_g\nonumber\\
&=&\sum_{i=1}^k\int_{B_{\delta_{\gamma}}(x_i)}\left(e^{(\pmxi-\tau_i)_+^p} -1\right)dv_g\nonumber\\
&\overset{\eqref{fact1}}{=}&(1+o(1))\,\int_{\R^2}\left(e^{\varphi_{\gamma}^p} -1\right)\,dx,
\end{eqnarray}
where in the second identity we used that $\pmxi\equiv 0$ on $\Sigma\setminus B_{\delta_{\gamma}}(x_i)$.
Given $\ve>0$, we need to show that
$$
\dist\left(f_{\gamma,\sigma}dv_g,\sigma\right)<2\ve\qquad\text{for $\gamma$ sufficiently large,}
$$
uniformly for $\sigma\in\Sigma_k$, where
$$
f_{\gamma,\sigma}=\frac{\(e^{\varphi_{\gamma,\sigma}^p}-1\)}{\int_\Sigma \(e^{\varphi_{\gamma,\sigma}^p}-1\) dv_g}.
$$
Let $\delta>0$ and $r_\ve>0$ be the positive constants of the statement of Lemma \ref{l:RK}. \\
It is immediate to see that $f_{\gamma,\sigma}$ satisfies \eqref{eq:ddmm4}, being $\varphi_{\gamma,\sigma}\equiv0$ in $\Sigma\setminus\cup_{i=1}^k B_{\delta_\gamma}(x_i)$, then by Lemma \ref{l:RK} (which holds with $r=\delta_\gamma$, if $\gamma$ is sufficiently large)
\begin{equation}\label{star}
\dist\left(f_{\gamma,\sigma},\sigma_\gamma\right)<\ve\quad\text{where}\quad\sigma_\gamma:=\sum_{i=1}^{k}\frac{\int_{B_{\delta_\gamma}(x_i)}(e^{\varphi_{\gamma,\sigma}^p}-1)dv_g}{\int_{\cup_{j=1}^k B_{\delta_\gamma}(x_j)}(e^{\varphi_{\gamma,\sigma}^p}-1)dv_g}\delta_{x_i}.
\end{equation}
In virtue of \eqref{fact1} and \eqref{fact2} $\sigma_\gamma=\sum_{i=1}^{k}t_i(1+o(1))\delta_{x_i}$, and so
\begin{equation}
\label{sstar}
\dist(\sigma_\gamma,\sigma)<\ve\quad\text{for $\gamma$ sufficiently large.}
\end{equation}
The thesis follows from \eqref{star} and \eqref{sstar}.
\end{proof}

Let us set for $L>0$
$$J_{p,\beta}^{-L}:=\{u\in H^1(\Sigma):  J_{p,\beta}(u)\le -L\}.$$

\begin{prop}\label{prop:homoteq} 
Let $\beta\in(4\pi k+\delta,4\pi(k+1)-\delta)$, with $k\geq1$ and $\delta \in (0,\tfrac{1}{2})$. Then, there exist $L>0$ and $\gamma>0$ sufficiently large depending on $p$, $k$ and $\delta$, and a continuous function
$$\Psi:J_{p,\beta}^{-L}\longrightarrow \Sigma_k$$
such that i) $\Phi(\sigma):=\pms\in J_{p,\beta}^{-2L}$ for every $\sigma\in\Sigma_k$ and ii) the map $\Psi\circ \Phi:\Sigma_k\to\Sigma_k$, is homotopically equivalent to the identity on $\Sigma_k$.
\end{prop}

\begin{proof}
By \cite[Proposition 2.2]{BattagliaJevnikarMalchiodi} there exist $\ve>0$ and a continuous retraction 
$$
\hat\Psi:\{\sigma\in \mathcal M(\Sigma)\,:\,\dist(\sigma,\Sigma_k)<\ve\}\to\Sigma_k.
$$
By Lemma \ref{l:ls2} there exists $L=L(\ve,p,\beta)$ such that for every $u\in J_{p,\beta}^{-L}$
$$
\dist\left(\frac{ \(e^{|u|^p}-1\)dv_g}{\int_\Sigma  \(e^{|u|^p}-1\) dv_g},\Sigma_k\right)<\ve.
$$
Since the map $u\mapsto \frac{ \(e^{|u|^p}-1\)dv_g}{\int_\Sigma  \(e^{|u|^p}-1\) dv_g}$ is continuous from $J^{-L}_{p,\beta}\subset H^1(\Sigma)$ into $\mathcal{M}(\Sigma)$,
for such $L$ the map $\Psi:J_{p,\beta}^{-L}\to\Sigma_k$ defined as
$$\Psi(u):=\hat\Psi\left(\frac{ \(e^{|u|^p}-1\)dv_g}{\int_\Sigma  \(e^{|u|^p}-1\) dv_g}\right)$$
is well posed and  continuous with respect to the $H^1(\Sigma)$ topology.

In turn, by Lemma \ref{lemma:lowsublevels} there exist $\gamma>0$ such that
\begin{equation}\label{eq:Pr1.13}
\pms\in J_{p,\beta}^{-2L},\quad \dist\left(\frac{\(e^{\varphi_{\gamma,\sigma}^p}-1\)dv_g}{\int_\Sigma \(e^{\varphi_{\gamma,\sigma}^p }-1\)dv_g},\sigma\right)<\ve,\quad \text{for any }\sigma\in\Sigma_k.
\end{equation}
Hence $\Psi\circ \Phi(\sigma)=\Psi(\pms)$ is well defined and we only need to show that $\Psi\circ \Phi\simeq \Id_{\Sigma_k}$. Consider the homotopy $H:[0,1]\times \Sigma_k\to \mathcal M(\Sigma)$ given by
$$
H(s,\sigma)=s\sigma+(1-s)\frac{\(e^{\varphi_{\gamma,\sigma}^p-1}\)dv_g}{\int_\Sigma \(e^{\varphi_{\gamma,\sigma}^p}-1\)dv_g}.
$$
From \eqref{eq:Pr1.13} we infer that 
$$\dist(H(s,\sigma),\Sigma_k)\le \dist(H(s,\sigma),\sigma)<\ve \quad \text{for } s\in [0,1], \,\sigma\in \Sigma_k,$$
so  $\hat\Psi$ is well defined on the image of $H$ and we can then define the homotopy $\mathcal H:[0,1]\times \Sigma_k\to  \Sigma_k$
$$
\mathcal H(s,\sigma)=\hat \Psi\circ H(s,\sigma).
$$
Clearly $\mathcal H(0,\cdot)=\Psi \circ\Phi$ and $\mathcal H(1,\cdot)=\Id_{\Sigma_k}$.
\end{proof}

We are now ready to construct a minmax scheme in the spirit of \cite{DjaMal}.
Given $p$, $k$ and $\delta>0$, fix $L>0$, $\gamma>0$ and $\Phi:\Sigma_k\to H^1(\Sigma)$ as in Proposition \ref{prop:homoteq}. 

Consider the topological cone $\mathcal{C}_k$ over $\Sigma_k$ defined as
$$\mathcal{C}_k=(\Sigma_k\times [0,1]) /\sim $$
where $(\sigma_1,r_1)\sim (\sigma_2,r_2)$ if and only if $r_1=r_2=1$. We shall also identify $\Sigma_k\times\{0\}$ with $\Sigma_k$. Set
$$\mathcal{A}_{k}:= \{ \bar\Phi \in C^0(\mathcal C_k, H^1(\Sigma)) \text{ s.t. } \bar\Phi|_{\Sigma_k}=\Phi\},$$
and call 
\begin{equation}\label{defalphabeta}
\alpha_\beta :=\inf_{\bar \Phi \in \mathcal{A}_k}\max_{\xi\in \mathcal{C}_k} J_{p,\beta}(\bar\Phi(\xi))
\end{equation}
the minmax value.

\begin{lem}\label{l:minmax}
With the above choice of $L$ and $\gamma$, depending on $p$, $k$ and $\delta$, we have
\begin{equation}\label{stimaminmax}
\alpha_\beta\ge -L,\quad \sup_{\bar\Phi\in \mathcal{A}_k} \sup_{\xi\in\Sigma_k} J_{p,\beta}(\bar\Phi(\xi))\le -2L.
\end{equation}
\end{lem}

\begin{proof} The second inequality follows immediately from Proposition \ref{prop:homoteq}. Assume by contradiction that $\alpha_\beta < -L$: then we can find $\bar\Phi\in \mathcal{A}_k$ such that
$$\bar\Phi(\mathcal{C}_k)\subset J_{p,\beta}^{-L}.$$
By Proposition \ref{prop:homoteq}, the map
$$\Psi\circ \bar\Phi :\mathcal{C}_k\to \Sigma_k$$
is well-defined and continuous. Moreover, on the one hand
\begin{equation}\label{homot}
\Psi\circ \bar\Phi|_{\Sigma_k}=\Psi \circ \Phi \simeq Id_{\Sigma_k},
\end{equation}
and on the other hand $\Psi\circ\bar\Phi$ gives a homotopy between $\Psi\circ\bar\Phi(\cdot, 0)=\Psi\circ \bar\Phi|_{\Sigma_k}$ and the constant map $\Psi\circ\bar\Phi(\cdot, 1)$.
This and \eqref{homot} imply that $\Sigma_k$ is homotopic to a point, which contradicts Lemma \ref{l:noncontr}.
\end{proof}

We will now use a  well-known monotonocity trick by Struwe to construct bounded Palais-Smale sequences for $J_{p,\beta}$ at level $\alpha_\beta$, as defined in \eqref{defalphabeta}:

\begin{prop}\label{PrboundedPS} For almost every $\beta >4\pi$ the functional $J_{p,\beta}$ admits a bounded Palais-Smale sequence at level $\alpha_\beta$, i.e. a sequence $(u_\ve)$ bounded in $H^1(\Sigma)$ such that 
\begin{equation}\label{PS}
J_{p,\beta}(u_\ve)\to \alpha_\beta,\quad J'_{p,\beta}(u_\ve)\to 0\quad \text{as }k\to \infty.
\end{equation}
\end{prop}

\begin{proof}
Since for all $u \in H^1$ $\beta \mapsto J_{p,\beta}(u)$ is monotone decreasing, 
 the function $\beta \mapsto \alpha_\beta$ is \emph{non-increasing}, hence it is differentiable almost everywhere. Set
$$D_p:=\{\beta \in (4\pi,\infty)\setminus 4\pi\mathbb{N} : \alpha_{\beta}\text{ is differentiable}\}.$$
Take $\beta\in D_p$, fix $\delta\in (0,\tfrac12)$ and $k\in\mathbb{N}^\star$ such that $\beta \in (4\pi k+\delta, 4\pi(k+1)-\delta)$, and  choose a sequence $\beta_\ve\uparrow \beta$ with $\beta_\ve \in (4\pi k+\delta, 4\pi(k+1)-\delta)$.
For every $\ve>0$ let a function $\bar\Phi_\ve\in \mathcal{A}_k$ be given such that
\begin{equation}\label{stimamono1}
\max_{\xi\in \mathcal{C}_k} J_{p,\beta_\ve}(\bar\Phi_\ve(\xi))\le \alpha_{\beta_\ve}+(\beta-\beta_\ve),
\end{equation}
and let also $\xi_\ve\in \mathcal{C}_k$ be given such that
\begin{equation}\label{stimamono2}
J_{p,\beta}(\bar\Phi_\ve(\xi_\ve))\ge \alpha_{\beta}.
\end{equation}
Notice that the set of  $(\bar\Phi_\ve,\xi_\ve)$'s satisfying \eqref{stimamono1}-\eqref{stimamono2} is non-empty thanks to \eqref{defalphabeta} (used with $\beta$ and $\beta_\ve$).

Set $v_\ve:=\bar\Phi_\ve(\xi_\ve)$. Then, posing $C_p:=\frac{2-p}{2}\(\frac{p}{2}\)^\frac{p}{2-p}$, we have that 
$$J_{p, \beta_\ve}(v_\ve)-  J_{p,\beta}(v_\ve)=C_p\|v_\ve\|_{{h}}^\frac{2p}{2-p}\left(\frac{1}{\beta_\ve^\frac{p}{2-p}}-\frac{1}{\beta^\frac{p}{2-p}}\right),$$
hence, setting $q=\frac{p}{2-p}$,  $\alpha'_\beta =\frac{d\alpha_\beta}{d\beta}$, and writing
$$\beta^q-\beta_\ve^q=-q\beta^{q-1}(\beta_\ve-\beta)+o(\beta_\ve-\beta),$$
we bound
\begin{equation}\label{Eqmontrick}
\begin{split}
\|v_\ve\|_{{h}}^{2q}&=\frac{(\beta_\ve\beta)^q}{C_p} \frac{ J_{p, \beta_\ve}(v_\ve)- J_{p,\beta}(v_\ve)}{\beta^q-\beta_\ve^q}\\
&\le \frac{(\beta_\ve\beta)^q}{C_p} \frac{\alpha_{\beta_\ve}- \alpha_{\beta} +\beta-\beta_\ve}{\beta^q-\beta_\ve^q}\\
&= \frac{\beta^{2q}+o(1)}{C_p}\cdot \frac{-\alpha'_{\beta}+1+o(1)}{q\beta^{q-1}}\\
&\le \bar C_{p,\beta}.
\end{split}
\end{equation}
In particular $\|v_\ve\|_{{h}}^\frac{2p}{p-2}=O(1)$ as $\ve\to 0$ for any sequence $v_\ve=\Phi_\ve(\xi_\ve)$, where $\Phi_\ve$ and $\xi_\ve$ satisfy \eqref{stimamono1} and \eqref{stimamono2}.

We now proceed similarly to \cite{DingJostLiWang2}. For every $\delta>0$ (not the same as in Lemma \ref{l:minmax}) consider the set
$$N_{\delta,M}:=\left\{u\in H^1(\Sigma): \|u\|_{{h}}\le M,\,|J_{p,\beta}(u)-\alpha_\beta|<\delta\right\}$$
for $M\ge \bar C_{p,\beta}^\frac{p-2}{2p}+1$, where $\bar C_{p,\beta}$ is as in \eqref{Eqmontrick}. Notice that $N_{\delta,M}$ is non-empty by the previous discussion.

Assume that the claim of the proposition is false, so that there exists $\delta>0$ small such that
$$\|J'_{p,\beta}(u)\|_{{H^{-1},h}}:=\sup_{\|v\|_{{h}}\le 1}\<J'_{p,\beta}(u),v\>\ge 2\delta \quad \text{for }u\in N_{\delta,M}.$$ 
Since $J_{p,\beta}$ is of class $C^1$ (on the open set of $H^1(\Sigma)$ where it is finite), we can construct a locally Lipschitz pseudo-gradient vector field (see e.g. \cite[Lemma 3.2]{StruweBook})
$$X:H^1(\Sigma) \to H^1(\Sigma)$$
such that
$$\sup_{u\in N_{\delta,M}}\|X(u)\|_{{h}}\le 1,\quad \sup_{u\in N_{\delta,M}}\<J'_{p,\beta}(u),X(u)\>\le -\delta.$$
We have
\begin{equation}\label{formulaJ'}
\<J_{p,\beta}'(u),v\>=C_{p,\beta}\|u\|_{H^1}^\frac{4p-4}{2-p}\<u,v\>_{{h}}-\frac{\int_\Sigma p u^{p-1}_+e^{u^p_+}v dv_g}{\int_{\Sigma}\(e^{u^p_+}-1\)dv_g},
\end{equation}
{where $\<u,v\>_h:=\int_\Sigma (\nabla u\nabla v+h uv) dv_g$ and} $C_{p,\beta}=p\(\frac{p}{2\beta}\)^\frac{p}{2-p}$, hence, for any sequence $\beta_\ve\uparrow \beta$
$$\|J_{p,\beta}'(u)-J_{p,\beta_\ve}'(u)\|_{H^{-1}{, h}}\le \(C_{p,\beta}-C_{p,\beta_\ve}\)\|u\|_{{h}}^\frac{3p-2}{2-p} =o(1)\quad \text{as }\ve\to 0,$$
uniformly for $u\in N_{\delta,M}$. Then for $\ve$ small we have
$$\sup_{u\in N_{\delta,M}}\<J'_{p,\beta_\ve}(u),X(u)\>\le 0.$$
We now choose a Lipschitz cut-off function $\eta:H^1(\Sigma)\to [0,1]$ such that
$$\eta(u)=0\text{ if } u\in H^1(\Sigma)\setminus N_{\delta, M}$$
and
$$\eta(u)=1\text{ if } u\in N_{\frac\delta2,M-1},$$
and consider the flow $\phi_t: H^1(\Sigma)\to H^1(\Sigma)$ generated by the vector field $\eta X$. Assuming with no loss of generality that $-2L< \alpha_\beta-\delta$, since $\Phi(\Sigma_k)\subset J_{p,\beta}^{-2L}$, it follows that
$$\phi_t \circ\bar\Phi|_{\Sigma_k}=\bar\Phi|_{\Sigma_k}=\Phi,$$
hence
$$\phi_t\circ \bar\Phi\in \mathcal{A}_{k} \quad \text{for every }\bar\Phi\in  \mathcal{A}_{k}, \quad t\ge 0.$$
Moreover
\begin{equation}\label{dJk<delta}
\frac{dJ_{p,\beta_\ve}(\phi_t(u))}{dt}\bigg|_{t=0}\le 0,\quad \text{for }u\in H^1(\Sigma),
\end{equation}
hence if $\bar\Phi_\ve$ satisfies \eqref{stimamono1}, so does $\phi_t\circ\bar\Phi_\ve$ for $t\ge 0$. Moreover, for $\varepsilon$ small, given any $\bar\Phi_\ve\in \mathcal{A}_k$ satisfying \eqref{stimamono1}
\begin{equation}\label{EqMono2}
\alpha_\beta\le \max_{\xi\in\mathcal{C}_{k}}J_{p,\beta}(\phi_t(\bar\Phi_\ve(\xi))) =\max_{\xi\in\mathcal{C}_k: \bar\Phi_\ve(\xi)\in N_{\frac\delta2,M-1} }J_{p,\beta}(\phi_t(\bar\Phi_\ve(\xi))),
\end{equation}
since every $\xi_\ve\in \mathcal{C}_k$ attaining the maximum of $J_{p,\beta}(\phi_t(\bar\Phi_\ve(\cdot)))$ satisfies \eqref{stimamono2}, so \eqref{EqMono2} follows from \eqref{Eqmontrick} and our choice of $M$. Therefore, since
\begin{equation*}
\frac{dJ_{p,\beta}(\phi_t(u))}{dt}\bigg|_{t=0}\le -\delta,\quad \text{for }u\in N_{\frac{\delta}{2},M-1},
\end{equation*}
we infer
$$\frac{d}{dt} \sup_{\xi\in\mathcal{C}_{k}}J_{p,\beta}(\phi_t(\bar\Phi_\ve(\xi)))\le -\delta\quad \text{for }t\ge 0,$$
which contradicts \eqref{EqMono2}.
\end{proof}

\begin{prop}\label{PrConvergence} Given $p\in (1,2)$ and $\beta>0$, let $(u_\ve)\subset H^1(\Sigma)$ be a bounded Palais-Smale sequence for $J_{p,\beta}$.
Then up to a subsequence we have $u_\ve\to u_0$ strongly in $H^1(\Sigma)$, where $u_0>0$ is a positive critical point of $J_{p,\beta}$.
\end{prop}

\begin{proof}
Up to a subsequence we have $u_\ve\to  u_0$ in $L^q(\Sigma)$ for every $q<\infty,$ almost everywhere and weakly in $H^1(\Sigma)$. 
Moreover, by Young's inequality and the Moser-Trudinger inequality we infer
\begin{equation}\label{MTp}
\|e^{u_{\ve+}^p}\|_{L^q}\le C(p,q,\|u_\ve\|_{{h}})\quad \text{for every }q<\infty,
\end{equation}
hence from Vitali's theorem
\begin{equation}\label{MTp2}
\int_{\Sigma}e^{u_{\ve+}^p}dv_g\to \int_{\Sigma}e^{u_{0+}^p}dv_g  \quad \text{as }\ve\to 0.
\end{equation}
From \eqref{formulaJ'} we deduce that 
$$\<J_{p,\beta}'(u_\ve),u_\ve-u_0\>=o(1)\quad\text{as }\ve\to 0.$$
Using $u_\ve-u_0$ as test function in $J_{p,\beta}'(u_\ve)\to 0$, we obtain
\begin{align*}
o(1)&=\<J_{p,\beta}'(u_\ve)-J_{p,\beta}'(u_0),u_\ve-u_0\>\\
&=C_{p,\beta}\<\|u_\ve\|_{{h}}^\frac{4p-4}{2-p}u_\ve-\|u_0\|_{{h}}^\frac{4p-4}{2-p}u_0,u_\ve-u_0\>_{{h}}\\
&\quad - \frac{\int_\Sigma p u_{\ve+}^{p-1}e^{u_{\ve+}^p}(u_\ve-u_0) dv_g}{\int_{\Sigma}e^{u_{\ve+}^p}dv_g} +\frac{\int_\Sigma p u_{0+}^{p-1}e^{u_{0+}^p}(u_\ve-u_0) dv_g}{\int_{\Sigma}e^{u_{0+}^p}dv_g}.
\end{align*}
Taking \eqref{MTp}, \eqref{MTp2} and the Sobolev embedding into account we notice that the last two terms sum up to $o(1)$, so that
\begin{align*}
o(1)&=\<\|u_\ve\|_{{h}}^\frac{4p-4}{2-p}u_\ve-\|u_0\|_{{h}}^\frac{4p-4}{2-p}u_0,u_\ve-u_0\>_{{h}}\\
&=\|u_\ve\|_{{h}}^\frac{4p-4}{2-p}\|u_\ve-u_0\|_{{h}}^2+ \(\|u_\ve\|_{{h}}^\frac{4p-4}{2-p}-\|u_0\|_{{h}}^\frac{4p-4}{2-p}\)\<u_0,u_\ve-u_0\>_{{h}}\\
&=\|u_\ve\|_{{h}}^\frac{4p-4}{2-p}\|u_\ve-u_0\|_{{h}}^2+o(1),
\end{align*}
hence $u_\ve\to u_0$ strongly in $H^1(\Sigma)$.

In order to prove that $u_0$ is a critical point of $J_{p,\beta}$, for $v\in H^1(\Sigma)$ we write
\begin{align*}
J_{p,\beta}'(u_0)(v)&= J_{p,\beta}'(u_0)(v)-J_{p,\beta}'(u_\ve)(v)+  o(1)\\
&=C_{p,\beta}\<\|u_\ve\|_{{h}}^\frac{4p-4}{2-p}u_\ve-\|u_0\|_{{h}}^\frac{4p-4}{2-p}u_0,v\>_{{h}}\\
&\quad - \frac{\int_\Sigma p u_{\ve+}^{p-1}e^{u_{\ve+}^p}v dv_g}{\int_{\Sigma}e^{u_{\ve+}^p}dv_g} +\frac{\int_\Sigma p u_{0+}^{p-1}e^{u_{0+}^p}v dv_g}{\int_{\Sigma}e^{u_{0+}^p}dv_g}+o(1)\\
&=C_{p,\beta}\|u_0\|_{{h}}^\frac{4p-4}{2-p}\<u_\ve-u_0,v\>_{{h}}\\
&\quad+\(\|u_\ve\|_{{h}}^\frac{4p-4}{2-p}-\|u_0\|_{{h}}^\frac{4p-4}{2-p}\)\<u_\ve,v\>_{{h}}+o(1)\\
&=o(1),
\end{align*}
hence $J_{p,\beta}'(u_0)=0$.

Were $u_0\equiv 0$, with \eqref{MTp2} we would infer that $J_{p,\beta}(u_\ve)\to\infty$, which is impossible since $(u_\ve)$ is a Palais-Smale sequence. Then Lemma \ref{l:pos} implies that $u_0>0$.
\end{proof}

\begin{rem} The analogue of proposition \ref{PrConvergence} does not hold in the case $p=2$ as proven by Costa-Tintarev (Theorem 5.1 in \cite{CostaTintarev}).
\end{rem}

\noindent\emph{Proof of Theorem \ref{ThmVariationalPart} (completed).} For every $\beta\in (0,4\pi)$ the functional $J_{p,\beta}$ has a minimizer, hence a critical point, which can be obtained via direct methods, using \eqref{MTp0}, \eqref{MTp} and \eqref{MTp2}. The existence of critical points for a.e. $\beta>4\pi$ follows at once from Propositions \ref{PrboundedPS} and \ref{PrConvergence}. \hfill $\square$

\section{A first analysis in the radially symmetric case}\label{SectRad}
Let $(p_\gamma)_\gamma$ be any family of numbers in $[1,2]$, and let $(\mu_\gamma)_\gamma$ be a given family of positive real numbers. Let $\lambda_\gamma>0$ be given by
\begin{equation}\label{BigLambdaDef1}
\lambda_\gamma p_\gamma^2\gamma^{2(p_\gamma-1)}\mu_\gamma^2 e^{\gamma^{p_\gamma}}=8\,,
\end{equation}
and let $t_\gamma, \bar{t}_\gamma$ be defined in $\mathbb{R}^2$ by 
\begin{equation}\label{TGammaDef1}
t_\gamma(x)=\ln\left(1+\frac{|x|^2}{\mu_\gamma^2} \right); \qquad \bar{t}_\gamma=t_\gamma+1\,,
\end{equation}
for all $\gamma>0$ large. In the sequel, for any radially symmetric function $f$ around $0\in \mathbb{R}^2$, since no confusion is then possible, we often make an abuse of notation and write $f(r)$ instead of $f(x)$ for $|x|=r$. Let $\eta\in (0,1)$ be fixed. Let also $(\bar{r}_\gamma)_\gamma$ be any family of positive real numbers such that
\begin{align}
&\lim_{\gamma\to +\infty} \frac{\mu_\gamma}{\bar{r}_\gamma}=0\,,\label{RMuTo01}\\
&\quad t_\gamma(\bar{r}_\gamma)\le \eta \frac{p_\gamma \gamma^{p_\gamma}}{2}\,,\label{BarRNotTooLarge1}\\
& \quad \gamma^{2 p_\gamma}\bar{r}_\gamma^2=O(1) \label{LpBd1}
\end{align}
for all $\gamma\gg 1$ large. {Given a positive constant $h_0>0$, w}e study in this section the behavior as $\gamma\to +\infty$ of a family $(B_\gamma)_\gamma$ of functions solving
\begin{equation}\label{EqBubble1}
\begin{cases}
&\Delta B_\gamma+{h_0}B_\gamma= \lambda_\gamma p_\gamma B_\gamma^{p_\gamma-1}e^{B_\gamma^{p_\gamma}}\,,\\
& B_\gamma(0)=\gamma>0\,,\\
& B_\gamma\text{ is radially symmetric and positive in }B_{\bar{r}_\gamma}(0)\,,
\end{cases}
\end{equation}
where $\Delta=-\partial_{xx}-\partial_{yy}$ denotes the Euclidean Laplace operator in $\mathbb{R}^2$. For $\gamma$ fixed,  \eqref{EqBubble1} reduces to an ODE with respect to the radial variable $r=|x|$: then we may assume that $B_\gamma$, defined in $[0,s_\gamma)$, is the \emph{maximal} positive solution of \eqref{EqBubble1} and it may be checked that it does not blow-up before it vanishes, namely $s_\gamma<+\infty$ implies $\lim_{r\to s_\gamma^-}B_\gamma(r)=0$. Actually, the proof of Proposition \ref{PropRadAnalysis1} below shows that our assumptions \eqref{BarRNotTooLarge1}-\eqref{LpBd1} ensure that $B_\gamma$ is well defined and positive in $B_{\bar{r}_\gamma}(0)$ for all $\gamma\gg 1$. Let $w_\gamma$ be given by
\begin{equation}\label{WGamma1}
B_\gamma=\gamma\left(1-\frac{2 t_\gamma}{p_\gamma \gamma^{p_\gamma}}+\frac{w_\gamma}{\gamma^{p_\gamma}} \right)\,.
\end{equation}
Then we have the following result:

\begin{prop}\label{PropRadAnalysis1}
We have $B_\gamma\le \gamma$,
\begin{equation*}
 w_\gamma=O(\gamma^{-p_\gamma} t_\gamma)\,,\quad w'_\gamma=O(\gamma^{-p_\gamma} t'_\gamma)\,,
\end{equation*}
and
\begin{equation*}
 \lambda_\gamma p_\gamma B_\gamma^{p_\gamma-1}e^{B_\gamma^{p_\gamma}}=\frac{8 e^{-2t_\gamma}}{\mu_\gamma^{2}\gamma^{p_\gamma-1}p_\gamma}\left(1+O\left(\frac{e^{\tilde{\eta} t_\gamma}}{\gamma^{p_\gamma}} \right)\right)\,,
\end{equation*}
uniformly in $[0,\bar{r}_\gamma]$ and for all $\gamma\gg 1$ large, where $\tilde{\eta}$ is any fixed constant in $(\eta,1)$ and $w_\gamma$ is as in \eqref{WGamma1}. 
\end{prop}

Once Proposition \ref{PropRadAnalysis1} is proven, we obtain first  
$$ B_\gamma(r)=\gamma-\frac{2}{p_\gamma \gamma^{p_\gamma-1}}\left( \ln \frac{1}{\mu_\gamma^2}+\ln (\mu_\gamma^2+r^2) \right)+O\left(\gamma^{1-p_\gamma} \right)$$
using \eqref{BarRNotTooLarge1} to handle the remainder term, so that we get from \eqref{BigLambdaDef1}
\begin{equation}\label{LowPOVBubble1}
B_\gamma(r)=-\left(\frac{2}{p_\gamma}-1 \right)\gamma+\frac{2}{p_\gamma \gamma^{p_\gamma-1}}\ln \frac{1}{\lambda_\gamma \gamma^{2(p_\gamma-1)}(\mu_\gamma^2+r^2)}+O\left(\gamma^{1-p_\gamma}\right)
\end{equation}
uniformly in $r\in [0,\bar{r}_\gamma]$ and for all $\gamma\gg 1$ large. While the principal part of the expression in \eqref{LowPOVBubble1} becomes negative for $r>0$ large enough, writing it in its initial form \eqref{WGamma1},  condition \eqref{BarRNotTooLarge1} and the pointwise estimate of $w_\gamma$ in Proposition \ref{PropRadAnalysis1} clearly ensure its positivity in the considered range $r\in [0,\bar{r}_\gamma]$, as claimed in \eqref{EqBubble1}.
\begin{proof}[Proof of Proposition \ref{PropRadAnalysis1}] Let $r_\gamma$ be given by
\begin{equation}\label{DefRGamma}
r_\gamma=\sup\left\{r\in [0,\bar{r}_\gamma]\text{ s.t. }|w_\gamma|\le \frac{t_\gamma}{ \gamma^{\frac{p_\gamma}{2}}} \text{ in }[0,r] \right\}
\end{equation}
for all $\gamma$. We aim to show that
\begin{equation}\label{BarREqR1}
r_\gamma=\bar{r}_\gamma
\end{equation}
for all $\gamma\gg 1$. We start by expanding the RHS in the first equation of \eqref{EqBubble1} uniformly in $[0,r_\gamma]$ as $\gamma\to +\infty$, using in a crucial way the control on $w_\gamma$ that we have by \eqref{DefRGamma}. Fix $\eta_1<\eta_2<\eta_3$ such that $\eta_k\in (\eta,1)$ for all $k$. When not specified, the expansions of this proof are uniform in $[0,r_\gamma]$ as $\gamma\to +\infty$. First, since $|w_\gamma|=o(t_\gamma)$, we get from \eqref{WGamma1} that $B_\gamma/\gamma\ge (1-\eta_1)$ in $[0,r_\gamma]$ for all $\gamma\gg 1$ large. First, for all $p\in [1,2]$ and all $x\le 1$, we notice that
$$0\le (1-x)^p-\left(1-p x \right)\le \frac{p^2}{4} x^2\,. $$
Then, we have 
$$0\le \frac{B_\gamma^{p_\gamma}}{\gamma^{p_\gamma}}-\left(1-\frac{2 t_\gamma- p_\gamma w_\gamma}{\gamma^{p_\gamma}} \right)\le \frac{t_\gamma^2}{\gamma^{2 p_\gamma}}(1+o(1))\,, $$
so  we get from \eqref{BarRNotTooLarge1} that
\begin{equation*}
\exp\left(B_\gamma^{p_\gamma}\right)=e^{\gamma^{p_\gamma}} e^{-2 t_\gamma}e^{p_\gamma w_\gamma}\left(1+O\left(\frac{t_\gamma^2}{\gamma^{p_\gamma}}e^{\eta_1 t_\gamma} \right) \right)\,.
\end{equation*}
Here and several times in the sequel, we use the elementary inequality
\begin{equation*}
\left|e^x-\sum_{j=0}^{n-1} \frac{x^j}{j!} \right|\le \frac{|x|^n}{n!} e^{|x|}
\end{equation*}
for all $x\in \mathbb{R}$ and all integers $n\ge 1$. Using also \eqref{BigLambdaDef1} and \eqref{DefRGamma} again, we get that
\begin{equation}\label{NonlinExpan1}
\begin{split}
&\lambda_\gamma p_\gamma B_\gamma^{p_\gamma-1} e^{B_\gamma^{p_\gamma}}\\
&=\frac{8~ e^{-2 t_\gamma}}{\mu_\gamma^2 \gamma^{p_\gamma-1} p_\gamma} \left(1+O\left(\frac{t_\gamma}{\gamma^{p_\gamma}} \right) \right)\times\\
&\quad  \left(1+ p_\gamma w_\gamma +O\left(\frac{t_\gamma^2}{\gamma^{p_\gamma}} \exp\left(\frac{p_\gamma t_\gamma}{\gamma^{p_\gamma/2}} \right)\right) \right)\left(1+O\left(\frac{t_\gamma^2}{\gamma^{p_\gamma}}e^{\eta_1 t_\gamma} \right)\right)\,,\\
&= \frac{8~ e^{-2 t_\gamma}}{\mu_\gamma^2 \gamma^{p_\gamma-1} p_\gamma}\left(1+ p_\gamma w_\gamma +O\left(\frac{\bar{t}_\gamma^3}{\gamma^{p_\gamma}}e^{\eta_2 t_\gamma}\right) \right)\,.
\end{split}
\end{equation}
In view of \eqref{DefRGamma}, to conclude the proof of \eqref{BarREqR1}, it is sufficient to obtain
\begin{equation}\label{BetterEstimW1}
|w_\gamma|=O\left(\frac{t_\gamma}{\gamma^{p_\gamma}} \right)\,,
\end{equation}
which we prove next. By \eqref{DefRGamma}, we have that $B_\gamma\le \gamma$ in $[0,r_\gamma]$ for all $\gamma\gg 1$. Set $\tilde{w}_\gamma=w_\gamma(\cdot/\mu_\gamma)$. Then, since $T_0:=\ln(1+|\cdot|^2)$ solves
\begin{equation}\label{Liouville1}
\Delta T_0= -4 e^{-2 T_0}\text{ in }\mathbb{R}^2\,,
\end{equation}
 we get from \eqref{EqBubble1} and \eqref{NonlinExpan1} that 
\begin{equation}\label{EqTildeW1}
\Delta \tilde{w}_\gamma = 8 e^{-2 T_0} \tilde{w}_\gamma+ O\left(\mu_\gamma^2\gamma^{p_\gamma} \right)+O\left(\frac{e^{(-2+\eta_3)T_0}}{\gamma^{p_\gamma}} \right)\,, \end{equation}
uniformly in $[0,r_\gamma/\mu_\gamma]$ as $\gamma\to +\infty$, applying $\Delta$ to \eqref{WGamma1}. By integrating \eqref{EqTildeW1} in $B_r(0)$ and also by parts, writing merely $|\tilde{w}_\gamma|\le r\|\tilde{w}'_\gamma\|_\infty$, we get that
$$-2\pi r~ \tilde{w}'_\gamma(r)=O\left(r^2 \mu_\gamma^2 \gamma^{p_\gamma} \right)+O\left(\frac{r^2}{\gamma^{p_\gamma}\left(1+r^2\right)}\right)+O\left(\frac{\|\tilde{w}'_\gamma\|_{\infty} r^3}{1+r^3}\right) \,,  $$
where $\|\tilde{w}'_\gamma\|_{\infty}$ stands for $\|\tilde{w}'_\gamma\|_{L^\infty([0,r_\gamma/\mu_\gamma])}$ and where $\tilde{w}'_\gamma=\frac{d}{dr}\tilde{w}_\gamma$, so that we get
\begin{equation}\label{TildeW1Estim1}
|\tilde{w}'_\gamma(r)|=O\left(\frac{r\mu_\gamma^2}{r_\gamma^2 \gamma^{p_\gamma}}\right)+O\left(\frac{r}{1+r^2}\left(\|\tilde{w}'_\gamma\|_{\infty}+\frac{1}{\gamma^{p_\gamma}} \right) \right)\,,
\end{equation}
uniformly in $r\in [0,r_\gamma/\mu_\gamma]$ as $\gamma\to +\infty$, using \eqref{LpBd1} and $r_\gamma\le \bar{r}_\gamma$. If $\|\tilde{w}'_\gamma\|_{\infty}=O(\gamma^{-p_\gamma})$ for all $\gamma$, \eqref{BetterEstimW1} follows from \eqref{RMuTo01},  \eqref{TildeW1Estim1} and from the fundamental theorem of calculus, using again $\tilde{w}_\gamma(0)=0$. Then, assume by contradiction that the complementary case occurs, namely that 
\begin{equation}\label{Contrad1}
\lim_{\gamma\to +\infty}\gamma^{p_\gamma}\|\tilde{w}'_\gamma\|_{\infty}=+\infty\,,
\end{equation}
maybe after passing to a subsequence. Let $\rho_\gamma\in [0,r_\gamma/\mu_\gamma]$ be such that $|\tilde{w}'_\gamma(\rho_\gamma)|=\|\tilde{w}'_\gamma\|_{\infty}$. By \eqref{RMuTo01}, \eqref{TildeW1Estim1} and \eqref{Contrad1}, up to a subsequence, $\rho_\gamma\to l$ and $r_\gamma/\mu_\gamma\to L$ as $\gamma\to +\infty$, for some $l\in (0,+\infty)$, $L\in (0,+\infty]$, $l\le L$. Setting now $\check{w}_\gamma:=\tilde{w}_\gamma/\|\tilde{w}'_\gamma\|_{\infty}$, we then get from (radial) elliptic theory and from \eqref{EqTildeW1} with \eqref{RMuTo01} and \eqref{LpBd1} that, up to a subsequence, $\check{w}_\gamma\to \check{w}_\infty$ in $C^1_{loc}([0,L))$ as $\gamma\to +\infty$ , where $\check{w}_\infty$ solves
\begin{equation}\label{SystContrad1}
\begin{cases}
&\Delta \check{w}_\infty=8 e^{-2 T_0} \check{w}_\infty\text{ in }B_L(0)\,,\\
&\check{w}_\infty(0)=0\,,\\
&\check{w}_\infty\text{ is radially symmetric}\,,\\
&|\check{w}'_\infty(l)|=1\,;
\end{cases}
\end{equation}
but by ODE theory, the only function satisfying the first three conditions in 
\eqref{SystContrad1} is the null function, which gives the expected contradiction. {Observe that we get also a contradiction in the most delicate case where $l=L$.} Indeed, since we then have $L\in (0,+\infty)$, writing \eqref{EqTildeW1} in radial coordinates  gives in this case that $(\|\check{w}_\gamma\|_{C^2([0,r_\gamma/\mu_\gamma])})_\gamma$ is bounded, so that $\check{w}'_\infty\in C^1([0,l])$ is well defined at $l$,  so that the fourth line in \eqref{SystContrad1} makes sense and holds true. As explained above, this concludes the proof of \eqref{BarREqR1}. Proposition \ref{PropRadAnalysis1} clearly follows.
\end{proof}

\section{Nonradial blow-up analysis: the case of a single bubble}\label{SectOneBubble}
Let $(p_\varepsilon)_\varepsilon$ be a sequence of numbers in $[1,2]$, let $(\mu_\varepsilon)_\varepsilon$ and $(\bar{r}_\varepsilon)_\varepsilon$ be sequences of positive real numbers. Let $(u_\varepsilon)_\varepsilon$ be a sequence of functions such that $u_\varepsilon$ is smooth in the closure of $B_{\bar{r}_\varepsilon}(0)$, where $B_{\bar{r}_\varepsilon}(0)$ is the ball of center $0$ and radius $\bar{r}_\varepsilon$ in the standard Euclidean space $\mathbb{R}^2$.  We assume that
\begin{equation}\label{GradNull2}
\nabla u_\varepsilon(0)=0
\end{equation}
for all $\varepsilon$ and that 
\begin{equation}\label{DefGammaEps2}
\gamma_\varepsilon:=u_\varepsilon(0)\to +\infty
\end{equation}
as $\varepsilon \to 0$. As for \eqref{BigLambdaDef1}, let $(\lambda_\varepsilon)_\varepsilon$ be given by
\begin{equation}\label{ScalRel2}
\lambda_\varepsilon p_\varepsilon^2 \gamma_\varepsilon^{2(p_\varepsilon-1)} \mu_\varepsilon^2 e^{\gamma_\varepsilon^{p_\varepsilon}}=8
\end{equation}
and let $t_\varepsilon, \bar{t}_\varepsilon$ be given by 
$$t_\varepsilon=\ln\left(1+\frac{|\cdot|^2}{\mu_\varepsilon^2} \right); \qquad \bar{t}_\varepsilon=t_\varepsilon+1$$ for all $\varepsilon$. Let $\eta\in (0,1)$ be fixed; assume also that
\begin{align}
&  \frac{\mu_\varepsilon}{\bar{r}_\varepsilon}=o(1)\,,\label{RMuTo02}\\
& t_\varepsilon(\bar{r}_\varepsilon)\le \eta \frac{p_\varepsilon \gamma_\varepsilon^{p_\varepsilon}}{2}\,,\label{BarRTNotTooLarge2}\\
& \int_{B_{\bar{r}_\varepsilon}(0)} u_\varepsilon^4 dx\le \bar{C}\,,\label{LpBd2Prelim}
\end{align}
for all $\varepsilon\ll 1$ small and for some given $\bar{C}>1$, and that
\begin{equation}\label{ConvLoc2}
\lim_{\varepsilon\to 0} \frac{p_\varepsilon}{2} \gamma_\varepsilon^{p_\varepsilon-1}\left(\gamma_\varepsilon-u_\varepsilon(\mu_\varepsilon\cdot) \right)=\ln\left(1+|\cdot|^2 \right)\text{ in }C^1_{loc}(\mathbb{R}^2)\,,
\end{equation}
up to a subsequence. As we will see in the subsequent blow-up analysis and in Lemma \ref{CorPropWeakPwEst}, 
the last two assumptions are indeed natural ones. 

Let $(v_\varepsilon)_\varepsilon$ be a sequence of smooth functions solving 
\begin{equation}\label{EqVEps2}
\begin{cases}
&\Delta v_\varepsilon + {h(0)}~v_\varepsilon=\lambda_\varepsilon p_\varepsilon v_\varepsilon^{p_\varepsilon-1} e^{v_\varepsilon^{p_\varepsilon}}\text{ in }B_{\bar{r}_\varepsilon}(0)\,,\\
&v_\varepsilon(0)=\gamma_\varepsilon\,\\
&v_\varepsilon\text{ is radially symmetric around }0\in\mathbb{R}^2\,,
\end{cases}
\end{equation}
for all $\varepsilon$, {where $h$ is a given smooth positive function on a neighborhood of $0\in \mathbb{R}^2$}. Let $(\varphi_\varepsilon)_\varepsilon$ be a sequence of smooth functions such that
\begin{equation}\label{CondVarphiEps2}
\lim_{\varepsilon\to 0} \varphi_\varepsilon(\bar{r}_\varepsilon\cdot)=\varphi_0\text{ in }C^2\left(\overline{B_1(0)}\right)\text{ and }\varphi_\varepsilon(0)=0
\end{equation}
for all $\varepsilon$ small. We assume that $u_\varepsilon$ solves
\begin{equation}\label{MainEqEps2}
\Delta u_\varepsilon=e^{2 \varphi_\varepsilon}\left(- {h} u_\varepsilon+\lambda_\varepsilon p_\varepsilon u_\varepsilon^{p_\varepsilon-1} e^{u_\varepsilon^{p_\varepsilon}}\right)\,,\quad u_\varepsilon>0\text{ in }B_{\bar{r}_\varepsilon}(0)\,,
\end{equation}
for all $\varepsilon$. At last, we assume that the following key gradient estimate holds true: there exists $C_G>0$ such that
\begin{equation}\label{WeakGradEstU2}
|x| |\nabla u_\varepsilon(x)| u_\varepsilon(x)^{p_\varepsilon-1} \le C_G\text{ for all }x\in B_{\bar{r}_\varepsilon}(0)
\end{equation}
for all $\varepsilon$. Letting $w_\varepsilon$ be given by
\begin{equation}\label{DefW2}
u_\varepsilon=v_\varepsilon+w_\varepsilon\,,
\end{equation}
the following proposition holds true:

\begin{prop}\label{PropRadCompar2}
We have that
\begin{equation}\label{WPointwEst}
|w_\varepsilon(x)|\le \frac{C_0 |x|}{\gamma_\varepsilon^{p_\varepsilon-1} \bar{r}_\varepsilon}\text{ for all }x\in B_{\bar{r}_\varepsilon}(0)\,,
\end{equation}
and that
\begin{equation}\label{WGradEst2}
\|\nabla w_\varepsilon\|_{L^\infty(B_{\bar{r}_\varepsilon}(0))}\le \frac{C_0}{\gamma_\varepsilon^{p_\varepsilon-1}\bar{r}_\varepsilon}
\end{equation}
for all $\varepsilon\ll 1$ small, where $C_0$ is any fixed constant greater than $(C_G/(1-\eta))+4$, for $C_G$ as in \eqref{WeakGradEstU2}
and $\eta$ as in \eqref{BarRTNotTooLarge2}. Up to a subsequence, there exists a function $\psi_0$, harmonic in $B_1(0)$, such that we have
\begin{align}
&\lim_{\varepsilon\to 0}\gamma_\varepsilon^{p_\varepsilon-1} w_\varepsilon(\bar{r}_\varepsilon \cdot)=\psi_0\text{ in }C^1_{loc}(B_1(0)\backslash\{0\})\label{HConv2}\,,\\
&\quad\quad\quad\quad\quad\quad\nabla \psi_0(0)=0 \label{CondH2}\,.
\end{align}
\end{prop}

In order to make sure that the estimates of Section \ref{SectRad} can be used to control the $v_\varepsilon$'s, it will be checked in the proof below that our assumptions of this section actually imply
\begin{equation}\label{LpBd2}
\gamma_\varepsilon^{2 p_\varepsilon} \bar{r}_\varepsilon^2= O(1)\,,
\end{equation}
for all $\varepsilon$ (see \eqref{LpBd1}). Besides, if we strengthen assumption \eqref{LpBd2Prelim} and  we assume 
\begin{equation}\label{AssumpLpBd2Strong}
\int_{B_{\bar{r}_\varepsilon}(0)} e^{u_\varepsilon^{1/3}} dx=O(1)\,,
\end{equation}
for all $\varepsilon$, again guaranteed by Lemma \ref{CorPropWeakPwEst}, we will also show that \eqref{LpBd2} may be improven to
\begin{equation}\label{LpBd2Strong}
\ln \gamma_\varepsilon=o\left(\ln \frac{1}{\bar{r}_\varepsilon} \right)
\end{equation}
as $\varepsilon\to 0$.
\begin{proof}[Proof of Proposition \ref{PropRadCompar2}]
We first prove \eqref{WPointwEst}. By \eqref{EqVEps2}, we have that $v_\varepsilon'(0)=0$; by \eqref{DefGammaEps2} and  \eqref{EqVEps2}, we have that $u_\varepsilon(0)=v_\varepsilon(0)$ and we then find 
\begin{equation}\label{CIWEps}
w_\varepsilon(0)=0\text{ and }\nabla w_\varepsilon(0)=0
\end{equation}
for all $\varepsilon$, using \eqref{GradNull2} and \eqref{DefW2}. Then, in order to get \eqref{WPointwEst}, it is sufficient to prove \eqref{WGradEst2}. Let $r_\varepsilon$ be given by
\begin{equation}\label{REpsDef2}
r_\varepsilon=\sup\left\{r\in [0,\bar{r}_\varepsilon]\text{ s.t. } 
\begin{cases}
&\gamma_\varepsilon^{p_\varepsilon-1} r \|\nabla w_\varepsilon\|_{L^\infty(B_r(0))}\le C_0\,,\\
& \gamma_\varepsilon^4 r^2\le \frac{2 \bar{C}}{\pi(1-\eta)^4}
\end{cases}
 \right\}
\end{equation}
for all $\varepsilon$, with $C_0>(C_G/(1-\eta))+4$ fixed as in Proposition \ref{PropRadCompar2} and $\bar{C}$ as in \eqref{LpBd2Prelim}. Then proving \eqref{WGradEst2} reduces to show that
\begin{equation}\label{EqRBar2}
r_\varepsilon=\bar{r}_\varepsilon
\end{equation}
for all $\varepsilon\ll 1$. By \eqref{RMuTo02} and \eqref{ConvLoc2}, there exist numbers $\tilde{r}_\varepsilon$ such that $\mu_\varepsilon=o(\tilde{r}_\varepsilon)$, $\tilde{r}_\varepsilon\le \bar{r}_\varepsilon$ and such that $u_\varepsilon=\gamma_\varepsilon(1+o(1))$ uniformly in $B_{\tilde{r}_\varepsilon}(0)$:  then, we get from \eqref{LpBd2Prelim} that
$$\int_{B_{\tilde{r}_\varepsilon}(0)} u_\varepsilon^4 dx=\pi \gamma_\varepsilon^4\tilde{r}_\varepsilon^2(1+o(1))\le \bar{C} $$
and that $\gamma_\varepsilon^{2 p_\varepsilon} \tilde{r}_\varepsilon^2\le 2\bar{C}/\pi$ for all $\varepsilon\ll 1$. Then, we may use Proposition \ref{PropRadAnalysis1} in $B_{\tilde{r}_\varepsilon}(0)$, with assumption \eqref{LpBd1}, to get that $$\lim_{\varepsilon\to 0}\frac{p_\varepsilon}{2}\gamma_\varepsilon^{p_\varepsilon-1}(\gamma_\varepsilon-v_\varepsilon(\mu_\varepsilon\cdot))=\ln(1+|\cdot|^2)\text{ in }C^1_{loc}(\mathbb{R}^2)\,,$$
which implies with \eqref{ConvLoc2} that $\gamma_\varepsilon^{p_\varepsilon-1}\mu_\varepsilon \|\nabla w_\varepsilon\|_{L^\infty(B_{R\mu_\varepsilon}(0))}=o(1)$ as $\varepsilon\to 0$, for all given $R\gg 1$. Summarizing, both conditions in  \eqref{REpsDef2} give that $\mu_\varepsilon=o(r_\varepsilon)$ as $\varepsilon\to 0$ and we may now apply Proposition \ref{PropRadAnalysis1} in $B_{r_\varepsilon}(0)$: we have that
\begin{equation}\label{MaxWGE2}
\sup_{s\in [0,r_\varepsilon]} \frac{p_\varepsilon}{2}\gamma_\varepsilon^{p_\varepsilon-1} s|v_\varepsilon'(s)|\le 2+o(1)
\end{equation}
for all $\varepsilon\ll 1$. Using $w_\varepsilon(0)=0$, we get from the first condition in \eqref{REpsDef2} that $|w_\varepsilon|\le C_0\gamma_\varepsilon^{1-p_\varepsilon}$ so that $u_\varepsilon=v_\varepsilon+O\left(\gamma_\varepsilon^{1-p_\varepsilon}\right)$ in $B_{r_\varepsilon}(0)$ for all $\varepsilon\ll 1$. Independently, we get from Proposition \ref{PropRadAnalysis1} and from \eqref{BarRTNotTooLarge2} that
\begin{equation}\label{MinorV2}
v_\varepsilon\ge \gamma_\varepsilon (1-\eta+o(1))\text{ in }[0,r_\varepsilon]\,,
\end{equation}
for all $\varepsilon\ll 1$. Then, writing $|\nabla w_\varepsilon|\le |\nabla u_\varepsilon|+|\nabla v_\varepsilon|$, using first \eqref{WeakGradEstU2} and \eqref{MaxWGE2}, and then \eqref{MinorV2} together with $p_\varepsilon\in [1,2]$, we get that
\begin{equation}\label{CClPart1}
\|\nabla w_\varepsilon\|_{L^\infty(\partial B_{r_\varepsilon}(0))}\le \frac{1+o(1)}{\gamma_\varepsilon^{p_\varepsilon-1} r_\varepsilon}\left(\frac{C_G}{(1-\eta)^{p_\varepsilon-1}}+4 \right)<\frac{C_0}{\gamma_\varepsilon^{p_\varepsilon-1} r_\varepsilon} 
\end{equation} 
for all $\varepsilon\ll 1$, using our assumption on $C_0$. Independently, Proposition \ref{PropRadAnalysis1} gives that $v_\varepsilon(r)'=O\left(r^{-1}\gamma_\varepsilon^{1-p_\varepsilon}\right)$, so  we first get  that 
\begin{equation}\label{IntegrateFromBdry2}
u_\varepsilon=v_\varepsilon(r_\varepsilon)+O\left(\gamma_\varepsilon^{1-p_\varepsilon}\ln \frac{2 r_\varepsilon}{|\cdot|} \right)\,,
\end{equation}
 then, with \eqref{MinorV2}, that also 
$$u_\varepsilon(r)^4=v_\varepsilon(r_\varepsilon)^4\left[1+O\left(\left(\gamma_\varepsilon^{-p_\varepsilon}\ln \frac{2 r_\varepsilon}{r}\right)+\left(\gamma_\varepsilon^{-p_\varepsilon}\ln \frac{2 r_\varepsilon}{r}\right)^4 \right)\right] $$
uniformly in $r\in (0,r_\varepsilon]$, and at last, with \eqref{LpBd2Prelim}, that
$$\pi v_\varepsilon(r_\varepsilon)^4 r_\varepsilon^2(1+o(1))= \int_{B_{r_\varepsilon}(0)}u_\varepsilon^4 dx\le \bar{C}\,: $$
summarizing, the second inequality in \eqref{REpsDef2} is strict at $r=r_\varepsilon$ for all $\varepsilon\ll 1$, using \eqref{MinorV2} again. However by \eqref{CClPart1}, the first inequality in \eqref{REpsDef2} is strict as well at $r=r_\varepsilon$, which concludes the proof of \eqref{EqRBar2} by continuity and then, as discussed above, those of \eqref{WPointwEst} and \eqref{WGradEst2}. Since $p_\varepsilon\le 2$, we get at the same time from \eqref{REpsDef2} that \eqref{LpBd2} holds true, so that we may apply Proposition \ref{PropRadAnalysis1} from now on to estimate the $v_\varepsilon$'s in $B_{\bar{r}_\varepsilon}(0)$. We turn now to the proofs of \eqref{HConv2} and \eqref{CondH2}. First, using \eqref{MinorV2} and that $v_\varepsilon\le \gamma_\varepsilon$ by Proposition \ref{PropRadAnalysis1}, since $|w_\varepsilon|=O\left(\gamma_\varepsilon^{1-p_\varepsilon} \right)$ by \eqref{WPointwEst}, we may  first write  
$u_\varepsilon^{p_\varepsilon}=v_\varepsilon^{p_\varepsilon}+p_\varepsilon v_\varepsilon^{p_\varepsilon-1}w_\varepsilon(1+o(1)) $ and $u_\varepsilon^{p_\varepsilon-1}=v_\varepsilon^{p_\varepsilon-1}\left(1+O\left(|w_\varepsilon|/\gamma_\varepsilon \right) \right)$, then 
\begin{equation*}
\begin{split}
&u_\varepsilon^{p_\varepsilon-1} e^{u_\varepsilon^{p_\varepsilon}}\\
 &~=v_\varepsilon^{p_\varepsilon-1} e^{v_\varepsilon^{p_\varepsilon}}\left(1+p_\varepsilon v_\varepsilon^{p_\varepsilon-1} w_\varepsilon\left[1+O\left(\frac{|w_\varepsilon|}{\gamma_\varepsilon}+v_\varepsilon^{p_\varepsilon-1}|w_\varepsilon| \right)\right]+O\left(\frac{|w_\varepsilon|}{\gamma_\varepsilon} \right) \right)\\
 &~= v_\varepsilon^{p_\varepsilon-1} e^{v_\varepsilon^{p_\varepsilon}}\left(1+p_\varepsilon v_\varepsilon^{p_\varepsilon-1} w_\varepsilon \left[1+O\left( \gamma_\varepsilon^{p_\varepsilon-1} |w_\varepsilon| \right)+O\left(\gamma_\varepsilon^{-p_\varepsilon} \right)\right] \right)\,, 
 \end{split}
 \end{equation*}
and, observing also $e^{2\varphi_\varepsilon}=1+O(|\cdot|)$ by \eqref{CondVarphiEps2}, $|w_\varepsilon|=O\left(\gamma_\varepsilon^{1-p_\varepsilon}|\cdot|/\bar{r}_\varepsilon \right)$ by \eqref{WPointwEst}, and using \eqref{EqVEps2} and \eqref{MainEqEps2}, we may write at last
\begin{equation}\label{WLapl2}
\begin{split}
\Delta w_\varepsilon=~&-e^{2 \varphi_\varepsilon}w_\varepsilon+O\left(|\cdot|v_\varepsilon \right)\\
&+\lambda_\varepsilon p_\varepsilon v_\varepsilon^{p_\varepsilon-1}e^{v_\varepsilon^{p_\varepsilon}}\left(p_\varepsilon v_\varepsilon^{p_\varepsilon-1} w_\varepsilon\left[1+O\left(\frac{|\cdot|}{\bar{r}_\varepsilon}+\frac{1}{\gamma_\varepsilon^{p_\varepsilon}} \right) \right]+O(|\cdot|) \right)
\end{split}
\end{equation}
uniformly in $B_{\bar{r}_\varepsilon}(0)$ and for all $\varepsilon\ll 1$. Setting now $\tilde{w}_\varepsilon=\gamma_\varepsilon^{p_\varepsilon-1}\frac{\bar{r}_\varepsilon}{\mu_\varepsilon} w_\varepsilon(\mu_\varepsilon \cdot)$ and given any $R\gg 1$, we get from Proposition \ref{PropRadAnalysis1} and \eqref{WLapl2} that
\begin{equation*}
\begin{split}
\Delta \tilde{w}_\varepsilon ~&~=O\left(\mu_\varepsilon^2\tilde{w}_\varepsilon\right)+O\left(\mu_\varepsilon^2 \gamma_\varepsilon^{p_\varepsilon} \bar{r}_\varepsilon \right)+\left[\frac{8 e^{-2 T_0}}{p_\varepsilon \gamma_\varepsilon^{p_\varepsilon-1}}\left(1+O(\gamma_\varepsilon^{-p_\varepsilon}) \right)\right]\times\\
&\quad\quad\left(p_\varepsilon \gamma_\varepsilon^{p_\varepsilon-1} \tilde{w}_\varepsilon\left[1+O\left(\frac{\mu_\varepsilon}{\bar{r}_\varepsilon}+\gamma_\varepsilon^{-p_\varepsilon} \right) \right]+O\left(\gamma_\varepsilon^{p_\varepsilon-1}\bar{r}_\varepsilon \right) \right) 
\end{split}
\end{equation*}
uniformly in $B_{R \mu_\varepsilon}(0)$, for all $\varepsilon$. Then, by \eqref{RMuTo02}, \eqref{WGradEst2}, \eqref{LpBd2}, the first assertion in \eqref{CIWEps} and elliptic theory, we get that, up to a subsequence,
\begin{equation}\label{W0ConvLoc2}
\lim_{\varepsilon\to 0} \tilde{w}_\varepsilon=w_0\text{ in }C^1_{loc}(\mathbb{R}^2)\,,
\end{equation}
where $w_0$ satisfies 
\begin{equation}\label{WEqLin2}
\begin{cases}
&\Delta w_0=8 \exp(-2 T_0) w_0\text{ in }\mathbb{R}^2\,,\\
&|w_0|\le C_0|\cdot| \text{ in }\mathbb{R}^2\,.
\end{cases}
\end{equation}
By the second assertion in \eqref{CIWEps} and \eqref{W0ConvLoc2}, we have $\nabla w_0(0)=0$. According to the classification result stated by Chen-Lin \cite[Lemma 2.3]{ChenLinSharpEst} and also in the generality on the growth assumption that we need here by Laurain \cite[Lemma C.1]{LaurainLemma}, this last property and \eqref{WEqLin2} imply 
\begin{equation}\label{WNull2}
w_0\equiv 0\,.
\end{equation}
In order to conclude the proofs of \eqref{HConv2} and \eqref{CondH2}, we establish now the following key estimate:
\begin{equation}\label{ComparGrad2}
\lim_{\varepsilon\to 0}\gamma_\varepsilon^{p_\varepsilon-1} \bar{r}_\varepsilon\|\nabla(w_\varepsilon-(\psi_\varepsilon-\psi_\varepsilon(0)))\|_{\infty,\varepsilon}=0\,,
\end{equation}
 where $\|\cdot\|_{\infty,\varepsilon}$ denotes $\|\cdot\|_{L^\infty(B_{\bar{r}_\varepsilon}(0))}$ and where $\psi_\varepsilon$ is given by
\begin{equation}\label{PsiEq2}
\begin{cases}
&\Delta \psi_\varepsilon=0\text{ in }B_{\bar{r}_\varepsilon}(0)\,,\\
&\psi_\varepsilon=w_\varepsilon\text{ on }\partial B_{\bar{r}_\varepsilon}(0)\,,
\end{cases}
\end{equation}
for all $\varepsilon$. Let $G^{(\varepsilon)}$ be the Green's function of $\Delta$ in $B_{\bar{r}_\varepsilon}(0)$ with zero Dirichlet boundary conditions (for an explicit formula for $G^{(\varepsilon)}$, see for instance Han-Lin \cite[Proposition 1.22]{HanLin}). Then (see also for instance \cite[Appendix B]{DruThiI}), there exists $C>0$ such that
\begin{equation*}
|\nabla G^{(\varepsilon)}_y(x)|\le \frac{C}{|x-y|}\,,
\end{equation*}
for all $x,y\in B_{\bar{r}_\varepsilon}(0)$, $x\neq y$ and all $\varepsilon$. Let $(y_\varepsilon)_\varepsilon$ be any sequence such that $y_\varepsilon\in B_{\bar{r}_\varepsilon}(0)$ for all $\varepsilon$. By the Green's representation formula, we may write
\begin{equation*}
\nabla (w_\varepsilon-\psi_\varepsilon)(y_\varepsilon)=\int_{B_{\bar{r}_\varepsilon}(0)} \nabla G^{(\varepsilon)}_{y_\varepsilon}(x) (\Delta w_\varepsilon)(x) dx
\end{equation*}
for all $\varepsilon$. Then, using also \eqref{CondVarphiEps2}, \eqref{WLapl2},  Proposition \ref{PropRadAnalysis1} and the first assertion in \eqref{CIWEps}, we get that
\begin{equation}\label{EstimGreen2}
\begin{split}
&|\nabla (w_\varepsilon-\psi_\varepsilon)(y_\varepsilon)|\\
&=O\left(\int_{B_{\bar{r}_\varepsilon}(0)} \frac{(\|\nabla w_\varepsilon\|_{\infty,\varepsilon}+\gamma_\varepsilon) |x| dx}{|y_\varepsilon-x|}  \right)\\
&\quad+O\left( \int_{B_{\bar{r}_\varepsilon}(0)} \frac{ |x| e^{(-2+\tilde{\eta})t_\varepsilon(x)}\left(\|\nabla w_\varepsilon\|_{\infty,\varepsilon}+\gamma_\varepsilon^{1-p_\varepsilon} \right) dx}{\mu_\varepsilon^2 |y_\varepsilon-x|} \right)\,,
\end{split}
\end{equation}
for all $\varepsilon$, where $\tilde{\eta}$ is some given constant in $(\eta,1)$. By the change of variable $x=\bar{r}_\varepsilon y$, we  first deduce
$$\int_{B_{\bar{r}_\varepsilon}(0)} \frac{|x| dx}{|y_\varepsilon-x|}=O\left( \bar{r}_\varepsilon^2\right)\,.$$ 
If we have $|y_\varepsilon|=O(\mu_\varepsilon)$, we get that
$$\int_{B_{\bar{r}_\varepsilon}(0)} \frac{ |x| e^{(-2+\tilde{\eta})t_\varepsilon(x)} dx}{\mu_\varepsilon^2 |y_\varepsilon-x|}=O(1) $$
for all $\varepsilon$, by the change of variable $x=\mu_\varepsilon y$; otherwise, up to a subsequence, we have $\mu_\varepsilon=o(|y_\varepsilon|)$ and 
$$ \int_{B_{\bar{r}_\varepsilon}(0)} \frac{ |x| e^{(-2+\tilde{\eta})t_\varepsilon(x)} dx}{\mu_\varepsilon^2 |y_\varepsilon-x|}=\int_{B_{\bar{r}_\varepsilon/|y_\varepsilon|}(0)} \frac{1}{\tilde{\mu}_\varepsilon^2} \frac{1}{\left(1+\frac{|y|^2}{\tilde{\mu}_\varepsilon^2} \right)^{2-\tilde{\eta}}} \frac{|y| dy}{|\tilde{y}_\varepsilon-y|}=O\left(\tilde{\mu}_\varepsilon \right) $$
for all $\varepsilon\ll 1$, by the change of variable $x=|y_\varepsilon| y$, where $\tilde{y}_\varepsilon=y_\varepsilon/|y_\varepsilon|$ has norm $1$ and $\tilde{\mu}_\varepsilon=\mu_\varepsilon/|y_\varepsilon|$. Plugging these estimates in \eqref{EstimGreen2}, we get in any case
\begin{equation}\label{PartCC2}
\begin{split}
&|\nabla (w_\varepsilon-(\psi_\varepsilon-\psi_\varepsilon(0)))(y_\varepsilon)|\\~&~=O\left((\|\nabla w_\varepsilon\|_{\infty,\varepsilon}+\gamma_\varepsilon) \bar{r}_\varepsilon^2 \right) +O\left(\frac{1}{1+\frac{|y_\varepsilon|}{\mu_\varepsilon}}\left(\|\nabla w_\varepsilon\|_{\infty,\varepsilon}+\gamma_\varepsilon^{1-p_\varepsilon} \right) \right)\,,\\
~&~=\frac{1}{\gamma_\varepsilon^{p_\varepsilon-1}\bar{r}_\varepsilon}\left(O\left(\frac{1}{1+\frac{|y_\varepsilon|}{\mu_\varepsilon}}\right)+o(1) \right)
\end{split}
\end{equation}
for all $\varepsilon$. The last line in \eqref{PartCC2} uses  \eqref{WGradEst2} and \eqref{LpBd2}. {We claim now that $(\psi_\varepsilon)_\varepsilon$ from \eqref{PsiEq2} satisfies
\begin{equation}\label{HarmExtGradEst2}
\|\nabla \psi_\varepsilon\|_{\infty,\varepsilon}=O\left(\frac{1}{\gamma_\varepsilon^{p_\varepsilon-1} \bar{r}_\varepsilon} \right)\,.
\end{equation}
Writing $\nabla \psi_\varepsilon=\nabla w_\varepsilon+\nabla (\psi_\varepsilon-w_\varepsilon)$, using \eqref{PartCC2} which gives
$$\left\|\nabla (\psi_\varepsilon-w_\varepsilon) \right\|_{\infty,\varepsilon}=O\left( \frac{1}{\gamma_\varepsilon^{p_\varepsilon-1} \bar{r}_\varepsilon}\right)\,, $$
we indeed get \eqref{HarmExtGradEst2} from \eqref{WGradEst2}.} Thus, we find from \eqref{HarmExtGradEst2} and elliptic theory that
\begin{equation}\label{ConvHarmPart02}
\lim_{\varepsilon\to 0}\gamma_\varepsilon^{p_\varepsilon-1}(\psi_\varepsilon(r_\varepsilon\cdot)-\psi_\varepsilon(0))=\psi_0\text{ in }C^1_{loc}(B_1(0))\,,
\end{equation}
up to a subsequence, where $\psi_0$ is harmonic in $B_1(0)$, and we obtain at last
\begin{equation}\label{ConvHarmPart}
\lim_{\varepsilon\to 0} \gamma_\varepsilon^{p_\varepsilon-1}\frac{\bar{r}_\varepsilon}{\mu_\varepsilon}(\psi_\varepsilon(\mu_\varepsilon\cdot)-\psi_\varepsilon(0))=\langle \nabla \psi_0(0),\cdot\rangle\text{ in }C^1_{loc}(\mathbb{R}^2)\,,
\end{equation}
by \eqref{RMuTo02}, where $\langle\cdot, \cdot\rangle$ denotes the standard scalar product in $\mathbb{R}^2$.

\smallskip

Assume now by contradiction that \eqref{ComparGrad2} does not hold true, in other words that, up to a subsequence,
\begin{equation}\label{Contradiction3}
\frac{1}{\gamma_\varepsilon^{p_\varepsilon-1} \bar{r}_\varepsilon}=O\left(\|\nabla(w_\varepsilon-\psi_\varepsilon)\|_{\infty,\varepsilon} \right) 
\end{equation}
for all $\varepsilon$. First, we claim that \eqref{CondH2} holds true, for $\psi_0$ as in \eqref{ConvHarmPart02}-\eqref{ConvHarmPart}. Indeed, let $R\gg 1$ be given and let $(y_\varepsilon)_\varepsilon$ be such that $y_\varepsilon\in\partial B_{R \mu_\varepsilon}(0)$ for all $\varepsilon\ll 1$.  We get from \eqref{W0ConvLoc2}, \eqref{WNull2} and \eqref{ConvHarmPart} that
 $$\lim_{\varepsilon\to 0}\gamma_\varepsilon^{p_\varepsilon-1} \bar{r}_\varepsilon\nabla(w_\varepsilon-\psi_\varepsilon)(y_\varepsilon)=\nabla \psi_0(0)\,.$$
This estimate, combined with \eqref{PartCC2}, proves  \eqref{CondH2} since $R\gg 1$ may be chosen arbitrarily large. Secondly, we may pick $(y_\varepsilon)_\varepsilon$, such that $y_\varepsilon\in \overline{B_{\bar{r}_\varepsilon}(0)}$ and
\begin{equation}\label{AsYEps2}
\|\nabla(w_\varepsilon-\psi_\varepsilon)\|_{\infty,\varepsilon}=|\nabla(w_\varepsilon-(\psi_\varepsilon-\psi_\varepsilon(0)))(y_\varepsilon)|
\end{equation}
for all $\varepsilon$, and we get from \eqref{PartCC2} and \eqref{Contradiction3} that $|y_\varepsilon|=O\left(\mu_\varepsilon \right)$ for all $\varepsilon$. However, \eqref{WNull2} and \eqref{ConvHarmPart} with \eqref{CondH2} contradict \eqref{Contradiction3} with \eqref{AsYEps2}, which concludes the proof of \eqref{ComparGrad2}. Then \eqref{HConv2} and \eqref{CondH2} follow from both assertions in \eqref{CIWEps}, from \eqref{ComparGrad2} and from \eqref{ConvHarmPart02}, which concludes the proof of Proposition \ref{PropRadCompar2}. To end this section, we assume \eqref{AssumpLpBd2Strong} and prove \eqref{LpBd2Strong}. We have \eqref{EqRBar2} and \eqref{IntegrateFromBdry2}. Then, using \eqref{MinorV2}, $$(1+t)^{1/3}=1+O(|t|^{1/3})\text{ for all }t>-1\,,$$ $p_\varepsilon\ge 1$ and $v_\varepsilon(\bar{r}_\varepsilon)\le \gamma_\varepsilon$, we get first 
$$u_\varepsilon^{1/3}=v_\varepsilon(\bar{r}_\varepsilon)^{1/3}\left(1+O\left(\gamma_\varepsilon^{-p_\varepsilon}\ln \frac{2 \bar{r}_\varepsilon}{|\cdot|} \right) \right)^{1/3}= v_\varepsilon(\bar{r}_\varepsilon)^{1/3}+O\left(\left(\ln \frac{2 \bar{r}_\varepsilon}{|\cdot|}\right)^{1/3} \right) $$
uniformly in $B_{\bar{r}_\varepsilon}(0)\backslash \{0\}$, so that we eventually get
$$\int_{B_{\bar{r}_\varepsilon}(0)} e^{u_\varepsilon^{1/3}} dx=e^{v_\varepsilon(\bar{r}_\varepsilon)^{1/3}}\int_{B_{\bar{r}_\varepsilon}(0)} \exp\left(O\left(\left(\ln \frac{2 \bar{r}_\varepsilon}{|\cdot|}\right)^{1/3} \right)  \right) dx \gtrsim e^{\left(\frac{(1-\eta)\gamma_\varepsilon}{2}\right)^{1/3}}\bar{r}_\varepsilon^2 \,, $$
for all $\varepsilon\ll 1$, which concludes the proof of \eqref{LpBd2Strong} by \eqref{AssumpLpBd2Strong}.
\end{proof}

\section{Nonradial blow-up analysis: the case of several  bubbles}
 The following theorem is the main result of this section. It is a quantization result determining in a precise way the possible blow-up energy levels. Notice that assumption \eqref{EnergyBound3} will follow from  variational reasons. 
\begin{thm}\label{ThmBlowUpAnalysis}
 {Let $h$ be a smooth positive function on $\Sigma$.} Let $(\lambda_\varepsilon)_\varepsilon$ be any sequence of positive real numbers and $(p_\varepsilon)_\varepsilon$ be any sequence of numbers in $[1,2]$. Let $(u_\varepsilon)_\varepsilon$ be a sequence of smooth functions solving 
\begin{equation}\label{MainEqEps}
\Delta_g u_\varepsilon+{h} u_\varepsilon=\lambda_\varepsilon p_\varepsilon u_\varepsilon^{p_\varepsilon-1} e^{u_\varepsilon^{p_\varepsilon}}\,,\quad u_\varepsilon>0\text{ in }\Sigma\,,
\end{equation}
for all $\varepsilon$. Let $(\beta_\varepsilon)_\varepsilon$ be given by
\begin{equation}\label{BetaEps}
\beta_\varepsilon=\frac{\lambda_\varepsilon p_\varepsilon^2}{2} \left( \int_\Sigma  \(e^{u_\varepsilon^{p_\varepsilon}}-1\)~dv_g\right)^{\frac{2-p_\varepsilon}{p_\varepsilon}}\left(  \int_\Sigma u_\varepsilon^{p_\varepsilon} e^{u_\varepsilon^{p_\varepsilon}}~dv_g\right)^{\frac{2(p_\varepsilon-1)}{p_\varepsilon}}
\end{equation}
for all $\varepsilon$. If we assume the energy bound
\begin{equation}\label{EnergyBound3}
\lim_{\varepsilon\to 0}\beta_\varepsilon=\beta \in [0,+\infty)\,,
\end{equation}
but the pointwise blow-up of the $u_\varepsilon$'s, namely
\begin{equation}\label{BlowUp3}
\lim_{\varepsilon\to 0} \max_{\Sigma} u_\varepsilon=+\infty\,,
\end{equation}
then, there exists an integer $k\ge 1$ such that
\begin{equation}\label{Quantization}
\beta=4\pi k\,.
\end{equation}
\end{thm}

A quantization result on a surface and in the specific case $p_\varepsilon=2$ was partially obtained by Yang \cite{YangQuantization}, following basically the scheme of proof developed in \cite{DruetDuke} to get an analoguous result on a bounded domain. However, even in this specific case, Theorem \ref{ThmBlowUpAnalysis} is stronger {(see also Remark \ref{RemTower})}. Indeed the analysis in \cite{YangQuantization} does not exclude that a nonzero $H^1$-weak limit $u_0$ of the $u_\varepsilon$'s contributes and breaks \eqref{Quantization}, that would become  
\begin{equation}\label{BadQuantization}
\beta=4\pi k+\|u_0\|_{H^1}^2\,.
\end{equation}
 On a bounded domain and still in this specific case $p_\varepsilon=2$, starting from the so-called weak pointwise estimates and \emph{using the first quantization} in \cite{DruetDuke}, a more precise blow-up analysis was carried out and in particular the precise quantization \eqref{Quantization} was obtained recently in \cite{DruThiI}. Here on a surface and for general $p_\varepsilon$'s in $[1,2]$, our proof starts also from the weak pointwise estimates, but \emph{gives at once the precise quantization, without using any intermediate one,} by pushing techniques in the spirit of \cite{DruThiI}. As mentioned in introduction, perturbing the standard critical nonlinearity in the RHS of \eqref{ELMainEq}, as we do here, requires to be very careful, if one wants to keep the precise quantization \eqref{Quantization}, which is crucial for the overall strategy of the present paper to work. Indeed, it was recently proven in \cite{ThiMan2} that \eqref{Quantization} may actually break down for some perturbations of the nonlinearity in \eqref{ELMainEq} which are surprisingly weaker in some sense than the ones that we consider here.

\medskip

 As a byproduct of Theorem \ref{ThmBlowUpAnalysis}, we easily get the following corollary, allowing to get critical points of $F$ in \eqref{MTFunctional} constrained to $\mathcal{E}_\beta$ in \eqref{Constraint}, as the limit of critical points of $J_{p,\beta}$ as $p\to 2$, for any fixed $\beta\not \in 4\pi \mathbb{N}^\star$.
\begin{cor}\label{CorThmBlowUpAnalysis}
{Let $h$ be a smooth positive function on $\Sigma$ and let}  $\beta\in (0,+\infty)\backslash 4\pi \mathbb{N}^\star$ be given. Let $(p_\varepsilon)_\varepsilon$ be any sequence of numbers in $[1,2)$ such that $p_\varepsilon\to 2$ as $\varepsilon\to 0$. Let $(u_\varepsilon)_\varepsilon$ be a sequence of smooth functions such that \eqref{MainEqEps} holds true for $\lambda_\varepsilon>0$ given by \eqref{BetaEps} and for $\beta_\varepsilon:=\beta$ for all $\varepsilon$. Then, up to a subsequence, we have that $u_\varepsilon\to u$ in $C^2$, where $u>0$ is smooth and solves {\eqref{FormulaLambda} for $p=2$ and} \eqref{ELMainEq}.
\end{cor}
 For any $\lambda > 0$, $p\in [1,2]$ and $u$ satisfying \eqref{MainEquation}, observe first that we necessarily have
\begin{equation}\label{BdLambdaEps3}
{2\lambda \le \max_\Sigma h\,,}
\end{equation}
 by integrating \eqref{MainEquation} in $\Sigma$, by using $qt^{q-1} e^{t^q}\ge 2 t$, for all $t>0$ and all $q\in[1,2]$, and the assumption in \eqref{MainEquation} that $u$ is positive on $\Sigma$.

\begin{proof}[Proof of Corollary \ref{CorThmBlowUpAnalysis}]
Let $\beta$, $(p_\varepsilon)_\varepsilon$, $(u_\varepsilon)_\varepsilon$ and $(\lambda_\varepsilon)_\varepsilon$ be given as in Corollary \ref{CorThmBlowUpAnalysis}. Since $\beta \not \in 4\pi \mathbb{N}^\star$, we get from Theorem \ref{ThmBlowUpAnalysis} that \eqref{BlowUp3} cannot hold true. Then, by \eqref{BdLambdaEps3} and by standard elliptic theory as developed in \cite{Gilbarg}, up to a subsequence, $\lambda_\varepsilon\to \lambda$ and $u_\varepsilon\to u$ in $C^2$ as $\varepsilon\to 0$, for some $C^2$-function $u\ge 0$ and some $\lambda\ge 0$ satisfying the equation in \eqref{ELMainEq} and {\eqref{FormulaLambda} for $p=2$}. If $u\equiv 0$, we clearly get a contradiction with \eqref{FormulaLambda}, since $\beta>0$. Then, $u\not \equiv 0$ and $u>0$ in $\Sigma$ by the maximum principle, which concludes the proof of Corollary \ref{CorThmBlowUpAnalysis}.
\end{proof}

We now turn to the proof of Theorem \ref{ThmBlowUpAnalysis} itself. From now on, we let $(\lambda_\varepsilon)_\varepsilon$ be a sequence of positive real numbers, we let $(p_\varepsilon)_\varepsilon$ be a sequence of numbers in $[1,2]$ and we let $(u_\varepsilon)_\varepsilon$ be a sequence of smooth functions solving \eqref{MainEqEps}. Let $(\beta_\varepsilon)_\varepsilon$ be given by \eqref{BetaEps}. \emph{We also assume \eqref{EnergyBound3}}. Then, since 
\begin{equation} \label{Conjugate}
 \frac{2-p_\varepsilon}{p_\varepsilon}+\frac{2(p_\varepsilon-1)}{p_\varepsilon}=1\,,
 \end{equation} 
 H\"older's inequality gives that
\begin{equation}\label{Easy2bis}
\lambda_\varepsilon \int_\Sigma u_\varepsilon^{2(p_\varepsilon-1)} \left(e^{u_\varepsilon^{p_\varepsilon}}-1\right) dv_g=O(1)\,. 
\end{equation}
By \eqref{BdLambdaEps3}, $p_\varepsilon\in[1,2]$ and the fact that $\Sigma$ has finite volume
\begin{eqnarray}\label{Easy2tris}
\lambda_\varepsilon\int_\Sigma u_\varepsilon^{2(p_\varepsilon-1)} dv_g&=&
\lambda_\varepsilon\int_{\{u_\varepsilon\leq2\}} u_\varepsilon^{2(p_\varepsilon-1)} dv_g+\lambda_\varepsilon\int_{\{u_\varepsilon>2\}} u_\varepsilon^{2(p_\varepsilon-1)} dv_g\nonumber\\
&\leq& O(1)+\lambda_{\varepsilon}e^{-2}\int_\Sigma e^{u_\varepsilon^{p_\varepsilon}}u_\varepsilon^{2(p_\varepsilon-1)} dv_g,
\end{eqnarray}
then as a consequence
\begin{equation}\label{Easy3}
\lambda_\varepsilon \int_\Sigma u_\varepsilon^{p} e^{u_\varepsilon^{p_\varepsilon}} dv_g=O(1)
\end{equation}
for all $p\in [0,2(p_\varepsilon-1)]$ and all $\varepsilon$. We get \eqref{Easy3}, for $p=2(p_\varepsilon-1)$, combining \eqref{Easy2bis} and \eqref{Easy2tris}, and then also for $p \in [0,2(p_\varepsilon-1))$, using that $\Sigma$ has finite volume and \eqref{BdLambdaEps3}.

As a first step, observe that we may directly get the following rough, subcritical but global bounds on the $u_\varepsilon$'s.

\begin{lem}\label{CorPropWeakPwEst}
There exists $C>0$ such that
$$\int_\Sigma e^{u_\varepsilon^{{1}/{3}}} dv_g\le C $$
for all $\varepsilon$. In particular, for all given $p<+\infty$, $(u_\varepsilon)_\varepsilon$ is bounded in $L^p$.
\end{lem}

Lemma \ref{CorPropWeakPwEst} strongly relies on \eqref{EnergyBound3} and is actually the very first step to get Proposition \ref{PropWeakPwEst} below, already obtained in \cite{YangQuantization} for $p_\varepsilon=2$. This lemma is relevant to handle the term ${h} u_\varepsilon$ in the LHS of \eqref{MainEqEps}, appearing in the present surface setting. 

\begin{proof}[Proof of Lemma \ref{CorPropWeakPwEst}] 
 Integrating \eqref{MainEqEps} in $\Sigma$, we get from the consequence \eqref{Easy3} of \eqref{EnergyBound3} that $(u_\varepsilon)_\varepsilon$ is bounded in $L^1$. Set now $\check{u}_\varepsilon=\max \{ u_\varepsilon,1\}$. Multiplying \eqref{MainEqEps} by $\check{u}_\varepsilon^{-1/3}$ and integrating by parts in $\Sigma$ (see for instance \cite[Proposition 2.5]{NonlinAnaH}), we get
$$3 \int_\Sigma |\nabla (\check{u}_\varepsilon^{1/3})|^2 dv_g=+\int_\Sigma \check{u}_\varepsilon^{-1/3} {h} u_\varepsilon dv_g-\lambda_\varepsilon p_\varepsilon \int_\Sigma \check{u}_\varepsilon^{-1/3} u_\varepsilon^{p_\varepsilon-1} e^{u_\varepsilon^{p_\varepsilon}} dv_g\,.  $$
Since $\check{u}_\varepsilon\ge 1$ and $(u_\varepsilon)_\varepsilon$ is bounded in $L^1$, it is clear that $\int_\Sigma \check{u}_\varepsilon^{-1/3} u_\varepsilon dv_g=O(1)$. Concerning the last integral, writing $\Sigma=\{x\text{ s.t. }u_\varepsilon>1\}\cup \{x\text{ s.t. }u_\varepsilon\le 1\}$, we find that the integral on the latter set is of order $O(1)$ since $\Sigma$ has finite volume and by \eqref{BdLambdaEps3}, while the integral on the complement is of order $O(1)$ by \eqref{Easy3} for $p=p_\varepsilon-1$, using $\check{u}_\varepsilon\ge 1$. Similarly, since $(u_\varepsilon)_\varepsilon$ is bounded in $L^1$, $(\check{u}_\varepsilon^{1/3})_\varepsilon$ is bounded in $L^2$. Then, by the Moser-Trudinger inequality, $(\exp(\check{u}_\varepsilon^{1/3}))_\varepsilon$ is bounded in $L^1$. Obviously, the same property also holds for $(\exp(u_\varepsilon^{1/3}))_\varepsilon$, which concludes the proof.
\end{proof}

\medskip

From now on, we also assume that the $u_\varepsilon$'s blow-up, namely \emph{we assume that \eqref{BlowUp3} holds}. In order to prove Theorem \ref{ThmBlowUpAnalysis}, we need to introduce some notation and a first set of pointwise estimates on the $u_\varepsilon$'s gathered in Proposition \ref{PropWeakPwEst} below. As aforementioned, these estimates have already been proven by Yang \cite{YangQuantization} in the case where  $p_\varepsilon$ equals $2$ for all $\varepsilon$. Yet, if this last specific condition is not satisfied, note that, even in the case $p_\varepsilon\to 2^-$, we are not here in the suitable framework  to use the results from \cite{YangQuantization}, since the nonlinearity appearing in the RHS of \eqref{MainEqEps} is not of \emph{uniform} Moser-Trudinger critical growth (see \cite[Definition 1]{DruetDuke}). However, as it was already observed in the literature (see for instance \cite{DeMarchisIanniPacella}), the technique of the \emph{pointwise} exhaustion of concentration points introduced in \cite{DruetDuke} is rather robust and may be successfully adapted to a much broader class of problems. Once Lemma \ref{CorPropWeakPwEst} is obtained, the proof of Proposition \ref{PropWeakPwEst} for general $p_\varepsilon$'s  is very similar to the corresponding proof for $p_\varepsilon=2$ in \cite{YangQuantization}. 

\medskip

Concerning the notation, for all $i\in\{1,...,N\}$ and $\varepsilon\ll 1$, we may choose isothermal coordinates $(B_{\kappa_1}{(x_{i,\varepsilon})}, \phi_{i,\varepsilon}, U_{i,\varepsilon})$ around $x_{i,\varepsilon}$, such that $\phi_{i,\varepsilon}$ is a diffeomorphism from $B_{\kappa_1}(x_{i,\varepsilon})\subset \Sigma$ to $U_{i,\varepsilon} \subset \mathbb{R}^2$, where $\kappa_1>0$ is some appropriate given positive constant and $B_{\kappa_1}(x_{i,\varepsilon})$ is the ball of radius $\kappa_1$ and center $x_{i,\varepsilon}$ for the metric $g$, such that $\phi_{i,\varepsilon}(x_{i,\varepsilon})=0$, such that $B_{2\kappa}(0)\subset U_{i,\varepsilon}$, for some $\kappa>0$, and such that $(\phi_{i,\varepsilon})_\star g=e^{2 \varphi_{i,\varepsilon}} \xi$, where $\mathbb{R}^2$ is endowed with its standard metric $\xi$ (see for instance  \cite{DoCarmo,TaylorBookI}). We may also assume that $(\varphi_{i,\varepsilon})_\varepsilon$ satisfies
\begin{equation}\label{ConVarphi3}
\forall \varepsilon\,,\quad \varphi_{i,\varepsilon}(0)=0\text{ and }\lim_{\varepsilon\to 0}\varphi_{i,\varepsilon}=\varphi_i\text{ in }C^2_{loc}(B_{2 \kappa}(0))\,.
\end{equation}
 At last, we set $$u_{i,\varepsilon}=u_\varepsilon\circ \phi_{i,\varepsilon}^{-1}  {\text{ and } h_{i,\varepsilon}=h\circ \phi_{i,\varepsilon}^{-1}}$$ in $B_{2\kappa}(0)$. We denote also by $d_g(\cdot,\cdot)$ the Riemannian distance on $(\Sigma,g)$.

\begin{prop}\label{PropWeakPwEst}
 Up to a subsequence, there exist an integer $N\ge 1$ and sequences $(x_{i,\varepsilon})_{\varepsilon}$ of points in $\Sigma$ such that $\nabla u_\varepsilon(x_{i,\varepsilon})=0$, such that, setting $\gamma_{i,\varepsilon}:=u_\varepsilon(x_{i,\varepsilon})$,
\begin{equation}\label{Mui3}
\mu_{i,\varepsilon}:=\left(\frac{8}{\lambda_\varepsilon p_\varepsilon^2 \gamma_{i,\varepsilon}^{2(p_\varepsilon-1)}  e^{\gamma_{i,\varepsilon}^{p_\varepsilon}}}\right)^{\frac{1}{2}} \to 0\,,
\end{equation} 
such that
\begin{equation}\label{FarFromEachOther3}
\forall j\in\{1,...,N\}\backslash\{ i\}\,, \frac{d_g(x_{j,\varepsilon},x_{i,\varepsilon})}{\mu_{i,\varepsilon}}\to +\infty\,,
\end{equation}
and such that
\begin{equation}\label{Convloc3}
\frac{p_\varepsilon}{2}\gamma_{i,\varepsilon}^{p_\varepsilon-1}(\gamma_{i,\varepsilon}-{u}_{i,\varepsilon}(\mu_{i,\varepsilon}\cdot))\to T_0:=\ln\left(1+|\cdot|^2 \right)\text{ in }C^1_{loc}(\mathbb{R}^2)\,,
\end{equation}
as $\varepsilon\to 0$, for all $i\in\{1,...,N\}$. Moreover, there exist $C_1, C_2>0$ such that we have
\begin{equation}\label{WeakPointwEst}
\min_{i\in \{1,...,N\}} u_\varepsilon^{p_\varepsilon-1} d_g(x_{i,\varepsilon},\cdot)^2 |\Delta_g u_\varepsilon| \le C_1\text{ in }\Sigma
\end{equation}
and
\begin{equation}\label{WeakGradEst}
\min_{i\in\{1,...,N\}} u_\varepsilon^{p_\varepsilon-1} d_g(x_{i,\varepsilon},\cdot) |\nabla u_\varepsilon|_g \le C_2\text{ in }\Sigma
\end{equation}
for all $\varepsilon$. We also have that $\lim_{\varepsilon\to 0}x_{i,\varepsilon}=x_i$ for all $i$, and that there exists $u_0\in C^2(\Sigma\backslash \mathcal{S})$ such that 
\begin{equation}\label{LocConv3}
\lim_{\varepsilon\to 0}u_\varepsilon=u_0\text{ in }C^2_{loc}(\Sigma\backslash \mathcal{S})\,,
\end{equation}
where $\mathcal{S}:=\{x_1,...,x_N\}$.
\end{prop}

Observe first that $\gamma_{i,\varepsilon}\to +\infty$ as $\varepsilon\to 0$ by  \eqref{BdLambdaEps3} and \eqref{Mui3}. As an other remark, by \eqref{FarFromEachOther3} and \eqref{Convloc3}, we have that
$$ 4\pi N=N \int_{\mathbb{R}^2} 4 e^{-2T_0} dx\le \liminf_{\varepsilon\to 0} \frac{\lambda_\varepsilon p_\varepsilon^2}{2}\int_\Sigma u_\varepsilon^{2(p_\varepsilon-1)} e^{u_\varepsilon^{p_\varepsilon}} dv_g\,, $$ 
so that \eqref{EnergyBound3} and its consequence \eqref{Easy3} for $p=2(p_\varepsilon-1)$ are not only used to get \eqref{Convloc3} from the classification in  \cite{ChenLi}, but also to get that the extraction procedure of the blow-up points $(x_{i,\varepsilon})_\varepsilon$ has to stop after a finite number $N$ of steps, which eventually gives \eqref{WeakPointwEst} (see \cite[Section 3]{DruetDuke}).
{\begin{rem}\label{RemTower}
 At this stage, we have only extracted the "highest bubbles" in \eqref{Convloc3} and it is not yet clear at all whether $N$ in Proposition \ref{PropWeakPwEst} is a good candidate to be $k$ in \eqref{Quantization} (see also the discussion in \cite[Section 2]{DruThiI})). Indeed, for $p=2$, it is now known (see \cite{ThizyManciniGenMT}) that a tower of $k$-bubbles may exist for nonlinearities which are lower order perturbations of the one in \eqref{ELMainEq} and we may then have only one "highest bubble" (i.e. $N=1$) with any $k\in \mathbb{N}^\star$ in \eqref{Quantization}. 
\end{rem} }
  We get from \eqref{MainEqEps} that
\begin{equation}\label{EqUIEps3}
\Delta u_{i,\varepsilon}=e^{2 \varphi_{i,\varepsilon}}\left(- {h_{i,\varepsilon}}u_{i,\varepsilon}+\lambda_\varepsilon p_\varepsilon u_{i,\varepsilon}^{p_\varepsilon-1} e^{u_{i,\varepsilon}^{p_\varepsilon}} \right)\,,\quad u_{i,\varepsilon}>0\text{ in }B_{2 \kappa}(0)\,,
\end{equation}
for all $i$ and $\varepsilon$, where $\Delta=\Delta_\xi$ throughout the paper. For all $i\in\{1,...,N\}$, we set
\begin{equation}\label{DefRIEps3}
r_{i,\varepsilon}=
\begin{cases}
&\kappa\text{ if }N=1\,,\\
&\min\left(\frac{1}{3} \min_{j\in\{1,...,N\}\backslash\{i\}} d_g(x_{i,\varepsilon}, x_{j,\varepsilon}), \kappa\right)\text{ otherwise}\,,
\end{cases}
\end{equation}
for all $\varepsilon$, so that we get from \eqref{Mui3} and  \eqref{FarFromEachOther3} that
\begin{equation}\label{RMuITo03}
\lim_{\varepsilon\to 0}\frac{\mu_{i,\varepsilon}}{r_{i,\varepsilon}}=0\,.
\end{equation}
We set $t_{i,\varepsilon}:=\ln\left(1+\frac{|\cdot|^2}{\mu_{i,\varepsilon}^2}\right)$ in $\mathbb{R}^2$. We set also
\begin{equation}\label{DefVI3}
v_{i,\varepsilon}=B_{\gamma_{i,\varepsilon}}\,,
\end{equation}
where $B_\gamma$ is as in \eqref{EqBubble1} for $(p_\gamma)_\gamma$ and $(\mu_\gamma)_\gamma$ satisfying $p_{\gamma_{i,\varepsilon}}=p_\varepsilon$ and $\mu_{\gamma_{i,\varepsilon}}= \mu_{i,\varepsilon}$, for all $\varepsilon$ and all $i\in\{1,...,N\}$. 
 Up to renumbering, we may also assume that
\begin{equation}\label{OrderRI3}
r_{1,\varepsilon}\le r_{2,\varepsilon}\le ...\le r_{N,\varepsilon}
\end{equation}
for all $\varepsilon$.

\smallskip

 In order to link the present situation to the results of Sections \ref{SectRad} and \ref{SectOneBubble}, we need some preliminary observations. Let $l\in\{1,...,N\}$ be given. Given a parameter $\eta\in (0,1)$ that is going to take several values in the proof below, we let $r_{l,\varepsilon}^{(\eta)}$ be given by 
\begin{equation}\label{DefRiEpsEta3}
t_{l,\varepsilon}\left(r_{l,\varepsilon}^{(\eta)}\right)=\eta \frac{p_\varepsilon \gamma_{l,\varepsilon}^{p_\varepsilon}}{2}\,,
\end{equation}
  and, for $r_{l,\varepsilon}$ as in \eqref{DefRIEps3}, we set 
\begin{equation}\label{BarRDef3}
\bar{r}_{l,\varepsilon}^{(\eta)}=\min\left(r_{l,\varepsilon}, r_{l,\varepsilon}^{(\eta)}\right)
\end{equation}
 for all $\varepsilon$. By collecting the above preliminary information, we can check that Proposition \ref{PropRadCompar2} applies with $\bar{r}_\varepsilon=\bar{r}_{l,\varepsilon}^{(\eta)}$, $\varphi_\varepsilon=\varphi_{l,\varepsilon}$, $u_\varepsilon=u_{l,\varepsilon}$, $\gamma_\varepsilon=\gamma_{l,\varepsilon}$ and $v_\varepsilon=v_{l,\varepsilon}$. In particular, the definition \eqref{DefRIEps3} of $r_{l,\varepsilon}$ is used to get \eqref{WeakGradEstU2} from \eqref{WeakGradEst}, while Lemma \ref{CorPropWeakPwEst} is used to get \eqref{LpBd2Prelim} and \eqref{AssumpLpBd2Strong}. As a remark, the metrics $(\phi_{l,\varepsilon})_\star g$ and $\xi$ are equivalent in $B_{\kappa}(0)$  by \eqref{ConVarphi3}: we use this fact here and currently in the sequel. We get in particular (see \eqref{LpBd2Strong}) that $\bar{r}_{l,\varepsilon}^{(\eta)}=o(1)$ and even that
\begin{equation}\label{BarRITo03}
\ln \gamma_{l,\varepsilon}=o\left(\ln \frac{1}{\bar{r}_{l,\varepsilon}^{(\eta)}} \right)\,,
\end{equation}
so that Proposition \ref{PropRadCompar2} also applies (see the remark involving \eqref{LpBd2}), and so that we get
\begin{equation}\label{IneqViEps3}
\gamma_{l,\varepsilon}\ge v_{l,\varepsilon}=\gamma_{l,\varepsilon}\left(1-\frac{2 t_{l,\varepsilon}\left(1+O\left(\gamma_{l,\varepsilon}^{-p_\varepsilon}\right) \right)}{p_\varepsilon \gamma_{l,\varepsilon}^{p_\varepsilon}} \right)\ge (1-\eta)\gamma_{l,\varepsilon}+O\left(\gamma_{l,\varepsilon}^{1-p_\varepsilon} \right)\,,
\end{equation}
uniformly in $\left[0,\bar{r}_{l,\varepsilon}^{(\eta)}\right]$ and for all $\varepsilon\ll 1$, using Proposition \ref{PropRadAnalysis1} and \eqref{DefRiEpsEta3}. We also get from Section \ref{SectOneBubble} (see \eqref{WPointwEst}) that
\begin{equation}\label{CorSect23}
|u_{l,\varepsilon}-v_{l,\varepsilon}|=O\left( \frac{|\cdot|}{ \gamma_{l,\varepsilon}^{p_\varepsilon-1}\bar{r}_{l,\varepsilon}^{(\eta)}} \right)
\end{equation}
and (see \eqref{WGradEst2})
\begin{equation}\label{GradCompl3}
|\nabla(u_{l,\varepsilon}-v_{l,\varepsilon})|=O\left(\frac{1}{\gamma_{l,\varepsilon}^{p_\varepsilon-1}\bar{r}_{l,\varepsilon}^{(\eta)}} \right)
\end{equation}
uniformly in $B_{\bar{r}_{l,\varepsilon}^{(\eta)}}(0)$ and for all $\varepsilon\ll 1$. We get now the following result:

\begin{Step}\label{StFarEnough3}
For all $i\in\{1,...,N\}$, we have that 
\begin{equation}\label{FarEnoughEq3}
\liminf_{\varepsilon\to 0}\frac{2 t_{i,\varepsilon}(r_{i,\varepsilon})}{p_\varepsilon \gamma_{i,\varepsilon}^{p_\varepsilon}} \ge 1\,,
\end{equation}
and that there exists $C\gg 1$ such that
\begin{equation}\label{MajorGlobal3}
0<\bar{u}_{i,\varepsilon}(r)\le -\left(\frac{2}{p_\varepsilon}-1 \right) \gamma_{i,\varepsilon}+\frac{2}{p_\varepsilon \gamma_{i,\varepsilon}^{p_\varepsilon-1}}\ln \frac{C}{\lambda_\varepsilon \gamma_{i,\varepsilon}^{2(p_\varepsilon-1)} r^2} +O\left(r^{3/2}\right)
\end{equation}
uniformly in $r\in \left(0,\kappa \right]$ and for all $\varepsilon\ll 1$, where $\bar{u}_{i,\varepsilon}$ is continuous in $[0,2\kappa)$ and given by
\begin{equation}\label{DefBarUI3}
\bar{u}_{i,\varepsilon}(r)=\frac{1}{2 \pi r} \int_{\partial B_r(0)} u_{i,\varepsilon} ~d\sigma_\xi\,,
\end{equation}
for all $r\in (0,2\kappa)$, where $d\sigma_\xi$ is the volume element for the metric induced in $\partial B_r(0)$ by the standard metric $\xi$ in $\mathbb{R}^2$. 
\end{Step}

\begin{proof}[Proof of Step \ref{StFarEnough3}]
We divide the proof of Step \ref{StFarEnough3} into two parts. 
 \begin{proof}[Proof of \eqref{MajorGlobal3}] Here we show \eqref{MajorGlobal3}, assuming that \eqref{FarEnoughEq3} is already obtained for some $i$. Let $\eta_1<\eta_2$ be two given numbers in $(0,1)$. Then by \eqref{DefRiEpsEta3}, \eqref{BarRDef3} and \eqref{FarEnoughEq3}, we get
 $$\bar{r}_{i,\varepsilon}^{(\eta_1)}=r_{i,\varepsilon}^{(\eta_1)}\text{ and }\bar{r}_{i,\varepsilon}^{(\eta_2)}=r_{i,\varepsilon}^{(\eta_2)} $$
 for all $\varepsilon\ll 1$. Then \eqref{MajorGlobal3} holds true uniformly in $\left(0,r_{i,\varepsilon}^{(\eta_2)}\right]$ using  \eqref{LowPOVBubble1} and \eqref{CorSect23} for $l=i$ and parameters $\eta_1$ or $\eta_2$. We get also from \eqref{IneqViEps3} and \eqref{CorSect23} that
 \begin{equation}\label{C0Compar3}
 \bar{u}_{i,\varepsilon}\left(r_{i,\varepsilon}^{(\eta_1)} \right)=v_{i,\varepsilon}\left(r_{i,\varepsilon}^{(\eta_1)} \right)+O\left(\gamma_{i,\varepsilon}^{1-p_\varepsilon} \right)\le \gamma_{i,\varepsilon}+O\left(\gamma_{i,\varepsilon}^{1-p_\varepsilon} \right)\,,
 \end{equation}
  and from \eqref{GradCompl3} that
 \begin{equation}\label{C1Compar3}
 \|\nabla(u_{i,\varepsilon}-v_{i,\varepsilon})\|_{L^\infty\left(\partial B_{r_{i,\varepsilon}^{(\eta_1)} }(0)\right)}=O\left(\frac{1}{\gamma_{i,\varepsilon}^{p_\varepsilon-1}r_{i,\varepsilon}^{(\eta_2)}} \right)
  \end{equation}
  for all $\varepsilon\ll 1$. For $f$ a $C^2$ function around $0\in \mathbb{R}^2$ and $r\ge 0$, we let $\bar{f}(r)$ (see \eqref{DefBarUI3}) be the average of $f$ on $\partial B_r(0)$; integrating by parts, we get
\begin{equation} \label{IPP3}
- 2\pi r \bar{f}'(r)=\int_{B_{r}(0)} (\Delta f)(x) dx
\end{equation}
with the usual radial (abuse of) notation. We write with \eqref{ConVarphi3} and \eqref{EqUIEps3} that
$$
\int_{B_r(0)} (\Delta u_{i,\varepsilon}) dx\ge \int_{B_{r_{i,\varepsilon}^{(\eta_1)}}(0)} (\Delta u_{i,\varepsilon}) dx+O\left(\int_{B_r(0)\backslash B_{r_{i,\varepsilon}^{(\eta_1)}}(0)} u_{i,\varepsilon} dx\right)\,, $$ that $\int_{B_r(0)} u_{i,\varepsilon} dx=O\left(r^{3/2} \right)$ by H\"older's inequality with Lemma \ref{CorPropWeakPwEst} for $p=4$, and then, with \eqref{IPP3}, that
\begin{equation}\label{DerEstUi3}
\bar{u}_{i,\varepsilon}'(r)\le -\frac{1}{2\pi r}\left(-2\pi r_{i,\varepsilon}^{(\eta_1)} \bar{u}'_{i,\varepsilon}\left(r_{i,\varepsilon}^{(\eta_1)} \right) \right)+O\left(r^{1/2} \right)
\end{equation}
uniformly in $r\in \left[r_{i,\varepsilon}^{(\eta_1)},\kappa\right]$ and for all $\varepsilon \ll 1$.  We get from the definition \eqref{DefRiEpsEta3} of $r_{l,\varepsilon}^{(\eta_j)}$ for $l=i$ and $j\in\{1,2\}$ that
 \begin{equation}\label{RiEtaCompar3}
 \ln\frac{r_{i,\varepsilon}^{(\eta_1)}}{r_{i,\varepsilon}^{(\eta_2)}}=-\frac{p_\varepsilon \gamma_{i,\varepsilon}^{p_\varepsilon}}{4}(\eta_2-\eta_1)+o(1)
 \end{equation}
 as $\varepsilon\to 0$. We now write 
\begin{equation*}
\begin{split}
\bar{u}_{i,\varepsilon}'=v_{i,\varepsilon}'+\left(\bar{u}'_{i,\varepsilon}-v'_{i,\varepsilon} \right)=-\frac{2 t'_{i,\varepsilon}}{p_\varepsilon \gamma_{i,\varepsilon}^{p_\varepsilon-1}}\left[1+O\left(\gamma_{i,\varepsilon}^{- p_\varepsilon} \right)  +O\left( \frac{r_{i,\varepsilon}^{(\eta_1)}}{r_{i,\varepsilon}^{(\eta_2)}} \right)\right]\,,
\end{split}
\end{equation*}
at $r_{i,\varepsilon}^{(\eta_1)}$ for all $\varepsilon\ll 1$, using Proposition \ref{PropRadAnalysis1} and \eqref{C1Compar3}. This implies with \eqref{RiEtaCompar3} that
\begin{equation}\label{UiEpsPrimeEstim3}
-2\pi r_{i,\varepsilon}^{(\eta_1)} \bar{u}'_{i,\varepsilon}\left(r_{i,\varepsilon}^{(\eta_1)} \right)=\frac{8\pi}{p_\varepsilon \gamma_{i,\varepsilon}^{p_\varepsilon-1}}+O\left(\gamma_{i,\varepsilon}^{1-2p_\varepsilon} \right) \,,
\end{equation}
using also that $$r_{i,\varepsilon}^{(\eta_1)} t'_{i,\varepsilon}\left(r_{i,\varepsilon}^{(\eta_1)}\right)=2+O\left(\mu_{i,\varepsilon}^2/\left(r_{i,\varepsilon}^{(\eta_1)}\right)^2 \right)=2+O\left(\gamma_{i,\varepsilon}^{-p_\varepsilon} \right)$$ for all $\varepsilon\ll 1$, by the definition \eqref{DefRiEpsEta3} of $r_{i,\varepsilon}^{(\eta_1)}$. Then, integrating \eqref{DerEstUi3} in $[r_{i,\varepsilon}^{(\eta_1)}, s ]$ and using the fundamental theorem of calculus and \eqref{UiEpsPrimeEstim3}, we get that
\begin{equation}\label{Interccl3}
\bar{u}_{i,\varepsilon}(s)-\bar{u}_{i,\varepsilon}\left(r_{i,\varepsilon}^{(\eta_1)}\right)\le -\frac{4}{p_\varepsilon \gamma_{i,\varepsilon}^{p_\varepsilon-1}}\ln \frac{s}{r_{i,\varepsilon}^{(\eta_1)}}\left(1+O\left(\gamma_{i,\varepsilon}^{-p_\varepsilon} \right) \right)+O\left(s^{3/2} \right) 
\end{equation}
uniformly in $s\in [r_{i,\varepsilon}^{(\eta_1)}, \kappa]$, for all $\varepsilon\ll 1$, and conclude the proof of \eqref{MajorGlobal3} by evaluating $\bar{u}_{i,\varepsilon}\left(r_{i,\varepsilon}^{(\eta_1)}\right)$ with  \eqref{LowPOVBubble1} and \eqref{C0Compar3}. To get the existence of $C>0$ in \eqref{MajorGlobal3} from the remainder in \eqref{Interccl3}, we use that \eqref{Interccl3} and $\bar{u}_{i,\varepsilon}(\kappa)>0$ imply 
$$0\le \ln \frac{s}{r_{i,\varepsilon}^{(\eta_1)}}=O\left(\gamma_{i,\varepsilon}^{p_\varepsilon-1}\bar{u}_{i,\varepsilon}\left(r_{i,\varepsilon}^{(\eta_1)}\right) \right)+O\left(\gamma_{i,\varepsilon}^{p_\varepsilon-1} \right)=O(\gamma_{i,\varepsilon}^{p_\varepsilon})$$
uniformly in $s\in [r_{i,\varepsilon}^{(\eta_1)},\kappa]$ and for all $\varepsilon\ll 1$, thanks to \eqref{C0Compar3}.
\end{proof}

\begin{proof}[Proof of \eqref{FarEnoughEq3}] We now turn to the proof of \eqref{FarEnoughEq3}. We prove it by induction on $i\in \{1,...,N\}$. In particular, we assume that \eqref{FarEnoughEq3} holds true at steps $1,...,i-1$ if $i\ge 2$. By contradiction, assume in addition that \eqref{FarEnoughEq3} does not hold true at step $i$. Thus, by \eqref{DefRiEpsEta3}-\eqref{BarRDef3}, up to a subsequence, we may choose and fix $\eta\in (0,1)$ sufficiently close to $1$ such that
\begin{equation}\label{ContradCons3}
\bar{r}_{i,\varepsilon}^{(\eta)}=r_{i,\varepsilon} 
\end{equation}
for all $\varepsilon\ll 1$. Set $J_i=\left\{j\in \{1,...,N\}\text{ s.t. } d_g(x_{i,\varepsilon}, x_{j,\varepsilon})=O\left(r_{i,\varepsilon} \right) \right\}$. Obviously, we get from \eqref{DefRIEps3} that 
\begin{equation}\label{ObviousProp}
r_{l,\varepsilon}=O\left(r_{i,\varepsilon} \right)
\end{equation}
for all $\varepsilon\ll 1$ and all $l\in  J_i$. We also find from \eqref{BarRITo03} for $l=i$ and from \eqref{ContradCons3} that $r_{i,\varepsilon}\to 0$, so   we get from \eqref{ConVarphi3} that
\begin{equation}\label{LocEucl3}
g_{l,\varepsilon}:=\left[\left(\phi_{l,\varepsilon} \right)_\star g \right](r_{i,\varepsilon}\cdot)\to \xi\text{ in }C^2_{loc}(\mathbb{R}^2)\,,
\end{equation}
 as $\varepsilon\to 0$ for all $l\in  J_i$. Up to a subsequence, we may assume that
$$\lim_{\varepsilon\to 0} \frac{\phi_{i,\varepsilon}(x_{l,\varepsilon})}{r_{i,\varepsilon}}=\check{x}_l\in \mathbb{R}^2 $$
for all $l\in J_i$, and we have that $\mathcal{S}_i:=\{\check{x}_l,l\in J_i\}$ contains at least two distinct points, by \eqref{DefRIEps3}, since $r_{i,\varepsilon}\to 0$ as $\varepsilon\to 0$. We may now choose and fix $\tau\in (0,1)$ small enough such that 
$$3\tau <\min_{\left\{(x,y)\in \mathcal{S}_i^2\,, x\neq y\right\}}|x-y|$$ and such that $\mathcal{S}_i\subset B_{1/(3 \tau)}(0)$. We can check that there exists $C>0$ such that any point in $$\Omega_{i,\varepsilon}:=B_{r_{i,\varepsilon}/\tau}(0)\backslash \cup_{j\in J_i}B_{\tau r_{i,\varepsilon}}(\phi_{i,\varepsilon}(x_{j,\varepsilon}))$$ may be joined to $\partial B_{\tau r_{i,\varepsilon}}(0)$ by a $C^1$ path in $\Omega_{i,\varepsilon}$ of $\xi$-length at most $C r_{i,\varepsilon}$, for all $\varepsilon\ll 1$. Therefore, by \eqref{ContradCons3} with \eqref{IneqViEps3} and \eqref{CorSect23} for $l=i$, we may estimate first $u_{i,\varepsilon}$ on $\partial B_{\tau r_{i,\varepsilon}}(0)$ and then get from \eqref{WeakGradEst} and \eqref{LocEucl3} that
\begin{equation}\label{LocBd3}
u_{i,\varepsilon}=\bar{u}_{i,\varepsilon}(\tau r_{i,\varepsilon})+O\left( \gamma_{i,\varepsilon}^{1-p_\varepsilon}\right)\ge (1-\eta ) \gamma_{i,\varepsilon}+O(1)
\end{equation}
uniformly in $\Omega_{i,\varepsilon}$ and for all $\varepsilon\ll 1$, with $\eta\in (0,1)$ still as initially fixed in \eqref{ContradCons3}. Independently, we get from \eqref{LowPOVBubble1}, \eqref{CorSect23} for $l=i$ and \eqref{ContradCons3} that
\begin{equation}\label{EstBdryBri3}
\bar{u}_{i,\varepsilon}(\tau r_{i,\varepsilon})=-\left(\frac{2}{p_\varepsilon}-1 \right)\gamma_{i,\varepsilon}+\frac{2}{p_\varepsilon \gamma_{i,\varepsilon}^{p_\varepsilon-1}}\Bigg(\ln \frac{1}{\lambda_\varepsilon \gamma_{i,\varepsilon}^{2(p_\varepsilon-1)} r_{i,\varepsilon}^2}+O(1)\Bigg)
\end{equation}
for all $\varepsilon\ll 1$.

\smallskip

\noindent \textbullet~ We prove now that, for all $j\in J_i$
\begin{equation}\label{NoCluster3}
j<i \implies \lim_{\varepsilon\to 0}\frac{\gamma_{i,\varepsilon}}{\gamma_{j,\varepsilon}}=0\,,
\end{equation}
up to a subsequence. Then, let $j\in J_i$ such that $j<i$. By \eqref{OrderRI3}, we have that $r_{j,\varepsilon}\le r_{i,\varepsilon}$ and, by our induction assumption, we know from \eqref{FarEnoughEq3} at step $j$ and from \eqref{DefRiEpsEta3}-\eqref{BarRDef3} that, given any $\eta_2\in (0,1)$, 
\begin{equation}\label{EncadrRjEta1}
\bar{r}_{j,\varepsilon}^{(\eta_2)}=r_{j,\varepsilon}^{(\eta_2)}
\end{equation}
 for all $\varepsilon\ll 1$. Then, by \eqref{IneqViEps3}, by  \eqref{CorSect23} for $l=j$ with parameter $\eta=\eta_2$, and by the definition \eqref{DefRiEpsEta3} of $r_{j,\varepsilon}^{(\eta_2)}$ we have that
$$\bar{u}_{j,\varepsilon}\left(\bar{r}_{j,\varepsilon}^{(\eta_2)} \right)\le (1-\eta_2)\gamma_{j,\varepsilon}(1+o(1)) $$ 
as $\varepsilon\to 0$. For all $l\in J_i$, let $w_{l,\varepsilon}$ be given by 
\begin{equation}\label{WjDef3}
\begin{cases}
&\Delta w_{l,\varepsilon}=-e^{2 \varphi_{l,\varepsilon}} {h_{l,\varepsilon}} u_{l,\varepsilon}\text{ in } B_{r_{i,\varepsilon}/(2\tau)}(0)\,,\\
&w_{l,\varepsilon}=0\text{ on }\partial B_{r_{i,\varepsilon}/(2\tau)}(0)\,.
\end{cases}
\end{equation} 
By observing that $\Delta (u_{l,\varepsilon}-w_{l,\varepsilon})\ge 0$ in $B_{r_{i,\varepsilon}/(2\tau)}(0)$ by \eqref{EqUIEps3}, the maximum principle yields that \emph{$u_{l,\varepsilon}-w_{l,\varepsilon}$ attains its infimum on $B_{r_{i,\varepsilon}/(2\tau)}(0)$ at some point in $\partial B_{r_{i,\varepsilon}/(2\tau)}(0)$}.
Moreover, for all given $p\in (1,+\infty)$, we get from Lemma \ref{CorPropWeakPwEst} and \eqref{ConVarphi3} that $$\|\Delta (w_{l,\varepsilon}(r_{i,\varepsilon}\cdot))\|_{L^p(B_{1/(2\tau)}(0))}=O\left(r_{i,\varepsilon}^{\frac{2(p-1)}{p}} \right)\,,$$ so, by elliptic theory, \eqref{DefRiEpsEta3}-\eqref{BarRITo03} and \eqref{ContradCons3}, we get
\begin{equation}\label{AnLinTerm3}
w_{l,\varepsilon}(r_{i,\varepsilon}\cdot)=O\left(r_{i,\varepsilon}^{\frac{2(p-1)}{p}}\right)=o\left(\gamma_{i,\varepsilon}^{1-p_\varepsilon} \right)
\end{equation}
uniformly in $B_{1/(2\tau)}(0)$ as $\varepsilon\to 0$. Summarizing this argument for $l=j$, we get
\begin{equation}\label{PlayingWEta3}
(1-\eta)\gamma_{i,\varepsilon}\le (1-\eta_2)\gamma_{j,\varepsilon}(1+o(1)) 
\end{equation}
as $\varepsilon\to 0$, using also \eqref{LocEucl3}-\eqref{LocBd3}. Indeed, by \eqref{LocEucl3}, observe that we may choose $\tau>0$ sufficiently small from the beginning to have 
\begin{equation}\label{Inclusion3}
\partial B_{r_{i,\varepsilon}/(2\tau)}(0)\subset \phi_{l,\varepsilon}^{-1}\circ \phi_{i,\varepsilon}\left(\Omega_{i,\varepsilon}\right)\,,
\end{equation}
so that we may estimate $u_{l,\varepsilon}$ on $\partial B_{r_{i,\varepsilon}/(2\tau)}(0)$ with \eqref{LocBd3}, for all $l\in J_i$ and all $\varepsilon\ll 1$. Since $\eta_2<1$ may be chosen arbitrarily close to $1$, \eqref{PlayingWEta3} gives \eqref{NoCluster3}.

\smallskip

\noindent \textbullet~ We prove now that, for all $j\in J_i$,
\begin{equation}\label{PropGammaSim3}
\gamma_{j,\varepsilon}=O\left(\gamma_{i,\varepsilon} \right)\,.
\end{equation}
By contradiction, if \eqref{PropGammaSim3} does not hold true, we choose $j\in J_i$ such that
\begin{equation}\label{HighContrad3}
\lim_{\varepsilon\to 0}\frac{\gamma_{i,\varepsilon}}{\gamma_{j,\varepsilon}}=0\,,
\end{equation} 
up to a subsequence. In particular, we have $j\neq i$. If $j>i$, we may write that
\begin{equation}\label{tjComput3}
\begin{split}
t_{j,\varepsilon}(r_{j,\varepsilon})~&~=\ln \frac{r_{j,\varepsilon}^2}{\mu_{j,\varepsilon}^2}+ o(1)\,,\\
&=\ln\frac{r_{j,\varepsilon}^2}{r_{i,\varepsilon}^2}+t_{i,\varepsilon}(r_{i,\varepsilon})+\ln \frac{\mu_{i,\varepsilon}^2}{\mu_{j,\varepsilon}^2}+o(1)\,,\\
&=O(1)+\eta \frac{p_\varepsilon}{2} \gamma_{i,\varepsilon}^{p_\varepsilon}+\gamma_{j,\varepsilon}^{p_\varepsilon}-\gamma_{i,\varepsilon}^{p_\varepsilon}+O\left(\ln \gamma_{i,\varepsilon}+\ln \gamma_{j,\varepsilon} \right)\,,\\
&=\gamma_{j,\varepsilon}^{p_\varepsilon}(1+o(1))
\end{split}
\end{equation}
as $\varepsilon\to 0$. The first two equalities use \eqref{RMuITo03}; the third one uses first our assumption $j>i$ with \eqref{OrderRI3} and \eqref{ObviousProp}, then the definition \eqref{DefRiEpsEta3} for $\eta$ as in \eqref{ContradCons3}, and at last \eqref{Mui3}; the last equality uses \eqref{HighContrad3}. Thus, given  any $\eta_2\in (0,1)$, we get in complement of \eqref{NoCluster3} and the paragraph below that \eqref{EncadrRjEta1} holds true also if $j>i$. As a first consequence, for all given $0<\eta'_2<\eta_2<1$, we get that 
\begin{equation}\label{Ro3}
\lim_{\varepsilon\to 0}\frac{r_{j,\varepsilon}^{(\eta'_2)}}{r_{i,\varepsilon}}=0\,,
\end{equation}
using \eqref{ObviousProp}. We get from \eqref{LowPOVBubble1} and \eqref{CorSect23}, for $l=j$ and parameter $\eta'_2$, that
\begin{equation}\label{UpperEstimBj3}
\bar{u}_{j,\varepsilon}\left(r_{j,\varepsilon}^{(\eta'_2)}\right)=-\left(\frac{2}{p_\varepsilon}-1 \right)\gamma_{j,\varepsilon}+\frac{2}{p_\varepsilon \gamma_{j,\varepsilon}^{p_\varepsilon-1}}\Bigg(\ln \frac{1}{\lambda_\varepsilon \gamma_{j,\varepsilon}^{2(p_\varepsilon-1)} \left(r_{j,\varepsilon}^{(\eta'_2)}\right)^2}+O(1)\Bigg)
\end{equation}
for all $\varepsilon\ll 1$. 

\smallskip

In order to have the desired contradiction with \eqref{HighContrad3}, fixing $\eta_2\in (0,1)$, we prove now the following estimate
\begin{equation}\label{LowerEstBj3}
\bar{u}_{j,\varepsilon}\left(r_{j,\varepsilon}^{(\eta_2)}\right)\ge \bar{u}_{i,\varepsilon}(\tau r_{i,\varepsilon})+\frac{2}{p_\varepsilon \gamma_{j,\varepsilon}^{p_\varepsilon-1}}\ln \frac{r_{i,\varepsilon}^2}{ \left(r_{j,\varepsilon}^{(\eta_2)}\right)^2}+O\left(\gamma_{i,\varepsilon}^{1-p_\varepsilon} \right)
\end{equation}
for all $\varepsilon\ll 1$. Let $\psi_\varepsilon$ be given by
\begin{equation*}
\begin{cases}
&\Delta \psi_\varepsilon=0\text{ in }B_{r_{i,\varepsilon}/(2\tau)}(0)\,,\\
&\psi_\varepsilon=u_{j,\varepsilon}\text{ on }\partial B_{r_{i,\varepsilon}/(2\tau)}(0)\,,
\end{cases}
\end{equation*}
for all $\varepsilon$. We get first
\begin{equation}\label{HarmControl3}
\psi_\varepsilon=\bar{u}_{i,\varepsilon}(\tau r_{i,\varepsilon}) +O\left(\gamma_{i,\varepsilon}^{1-p_\varepsilon} \right)
\end{equation}
for all $\varepsilon\ll 1$, by \eqref{LocBd3}, \eqref{Inclusion3} and the maximum principle for the harmonic function $\psi_\varepsilon$. Let $(z_\varepsilon)_\varepsilon$ be any sequence of points such that $|z_\varepsilon|=r_{j,\varepsilon}^{(\eta_2)}$ for all $\varepsilon$. Let $G_\varepsilon$ be the Green's function of $\Delta$ in $B_{r_{i,\varepsilon}/(2\tau)}(0)$ with zero Dirichlet boundary conditions. We know that $G_\varepsilon(x,y)>0$ by the maximum principle for all $x, y\in B_{r_{i,\varepsilon}/(2\tau)}(0)$, $x\neq y$ and for all $\varepsilon$. Let $\eta_1\in (0,\eta_2)$ be fixed. By  Green's representation formula and \eqref{EqUIEps3}, using the positivity of $\Delta u_{j,\varepsilon}+ e^{2\varphi_{j,\varepsilon}} {h_{j,\varepsilon}} u_{j,\varepsilon}$ and that of $G_{\varepsilon}(z_\varepsilon,\cdot)$, we have that
\begin{equation}\label{GreensLowerBd3}
(u_{j,\varepsilon}-\psi_\varepsilon-w_{j,\varepsilon})(z_\varepsilon)\ge \lambda_\varepsilon p_\varepsilon \int_{B_{r_{j,\varepsilon}^{(\eta_1)}}(0)} G_{\varepsilon}(z_\varepsilon,y)  e^{2\varphi_{j,\varepsilon}}u_{j,\varepsilon}(y)^{p_\varepsilon-1} e^{u_{j,\varepsilon}^{p_\varepsilon}(y)} dy
\end{equation}
for all $\varepsilon\ll 1$, with $w_{j,\varepsilon}$ given by  \eqref{WjDef3}. There exists $C>0$ (see \cite[Appendix B]{DruThiI}) such that
\begin{equation*}
\left|G_{\varepsilon}(z,y)-\frac{1}{2\pi} \ln \frac{r_{i,\varepsilon}}{|z-y|} \right|\le C
\end{equation*}
for all $y\in B_{r_{i,\varepsilon}/(2\tau)}(0)$, for all $z\in B_{5r_{i,\varepsilon}/(12\tau)}(0)$, $y\neq z$ and for all $\varepsilon\ll 1$. Observe also that \eqref{RiEtaCompar3} holds true for $l=j$. Then, since $|z_\varepsilon|=r_{j,\varepsilon}^{(\eta_2)}$, we  first get that
\begin{equation}\label{GreensPrec3}
G_\varepsilon(z_\varepsilon,\cdot)=\frac{1}{2\pi }\ln\frac{r_{i,\varepsilon}}{|z_\varepsilon|}+O(1)+O\left(\frac{|\cdot|}{|z_\varepsilon|} \right)=\frac{1}{2\pi }\ln\frac{r_{i,\varepsilon}}{r_{j,\varepsilon}^{(\eta_2)}}+O(1)
\end{equation}
uniformly in $B_{r_{j,\varepsilon}^{(\eta_1)}}(0)$ and for all $\varepsilon\ll 1$. Now, by \eqref{IneqViEps3} and \eqref{CorSect23} for $l=j$ with parameter $\eta=\eta_2$, computing as in Proposition \ref{PropRadAnalysis1} or in the argument involving \eqref{WLapl2}, we get that, for some given $\tilde{\eta}\in (0,1)$,
\begin{equation}\label{ExpanPrec3}
\lambda_\varepsilon p_\varepsilon e^{2 \varphi_{j,\varepsilon}} u_{j,\varepsilon}^{p_\varepsilon-1} e^{u_{j,\varepsilon}^{p_\varepsilon}}=\frac{8 e^{-2 t_{j,\varepsilon}}}{\mu_{j,\varepsilon}^2 \gamma_{j,\varepsilon}^{p_\varepsilon-1} p_\varepsilon} \left(1+O\left(e^{\tilde{\eta} t_{j,\varepsilon}}\left[\frac{|\cdot|}{r_{j,\varepsilon}^{(\eta_2)}}+|\cdot|+\frac{1}{\gamma_{j,\varepsilon}^{p_\varepsilon}} \right] \right) \right)
\end{equation}
in  $B_{r_{j,\varepsilon}^{(\eta_1)}}(0)$ and for all $\varepsilon\ll 1$. Resuming arguments in \eqref{tjComput3} and using \eqref{Ro3}, we have that
\begin{equation}\label{tjSuite3}
0<\ln \frac{r_{i,\varepsilon}}{r_{j,\varepsilon}^{(\eta_2)}}\le \ln \frac{r_{i,\varepsilon}}{\mu_{i,\varepsilon}}+\ln \frac{\mu_{i,\varepsilon}}{\mu_{j,\varepsilon}}=\gamma_{j,\varepsilon}^{p_\varepsilon}(1+o(1))
\end{equation} 
as $\varepsilon\to 0$, since \eqref{HighContrad3} is assumed to be true. By \eqref{GreensPrec3}, \eqref{ExpanPrec3} and \eqref{tjSuite3}, we get that
 \begin{equation*}
 \begin{split}
 &\lambda_\varepsilon p_\varepsilon \int_{B_{r_{j,\varepsilon}^{(\eta_1)}}(0)} G_{\varepsilon}(z_\varepsilon,y)  u_{j,\varepsilon}(y)^{p_\varepsilon-1} e^{u_{j,\varepsilon}^{p_\varepsilon}(y)} dy\\
 &= \left( \frac{1}{2\pi }\ln\frac{r_{i,\varepsilon}}{r_{j,\varepsilon}^{(\eta_2)}}+O(1)\right)\frac{8\pi}{\gamma_{j,\varepsilon}^{p_\varepsilon-1} p_\varepsilon}\Bigg(1+O\Bigg(\left[{\mu_{j,\varepsilon}}/{r_{j,\varepsilon}^{(\eta_1)}}\right]^2\Bigg)+O\Bigg(\frac{r_{j,\varepsilon}^{(\eta_1)}}{r_{j,\varepsilon}^{(\eta_2)}}+\frac{1}{\gamma_{j,\varepsilon}^{p_\varepsilon}}\Bigg) \Bigg)\\
 &=\frac{2}{p_\varepsilon \gamma_{j,\varepsilon}^{p_\varepsilon-1}}\ln \frac{r_{i,\varepsilon}^2}{ \left(r_{j,\varepsilon}^{(\eta_2)}\right)^2}+O\left(\gamma_{j,\varepsilon}^{1-p_\varepsilon} \right)
 \end{split}
 \end{equation*}
for all $\varepsilon\ll 1$, using the definition of $r_{j,\varepsilon}^{(\eta_1)}$, \eqref{RiEtaCompar3} and  \eqref{EncadrRjEta1} with \eqref{BarRITo03} for $l=j$. By plugging this last estimate with \eqref{AnLinTerm3}, \eqref{HighContrad3} and \eqref{HarmControl3} in \eqref{GreensLowerBd3}, since $(z_\varepsilon)_\varepsilon$ is arbitrary, this concludes the proof of \eqref{LowerEstBj3}. 

\smallskip

We now plug \eqref{EstBdryBri3} and \eqref{UpperEstimBj3} in \eqref{LowerEstBj3} and we get
$$\left(\frac{2}{p_\varepsilon}-1 \right)\gamma_{j,\varepsilon}(1+o(1))+\frac{2+o(1)}{p_\varepsilon \gamma_{i,\varepsilon}^{p_\varepsilon-1}}\left(\ln \frac{1}{r_{i,\varepsilon}^2}+\ln \frac{1}{\lambda_\varepsilon} \right)\le O\left(\gamma_{i,\varepsilon}^{1-p_\varepsilon} \ln \gamma_{i,\varepsilon} \right) $$
still using \eqref{HighContrad3}. However this estimate gives a contradiction for $\varepsilon\ll 1$, by \eqref{BdLambdaEps3} and \eqref{BarRITo03} for $l=i$ and \eqref{ContradCons3}: \eqref{PropGammaSim3} is proven.

\smallskip

\noindent \textbullet~ Then, using \eqref{OrderRI3} and \eqref{NoCluster3}, \eqref{PropGammaSim3} implies that for all $l\in J_i$
\begin{equation}\label{ComplObviousProp}
r_{i,\varepsilon}\le r_{l,\varepsilon}
\end{equation}
for all $\varepsilon\ll 1$. We now claim  that there exists $\eta_3\in (\eta,1)$ such that
\begin{equation}\label{ComplContradCons3}
\bar{r}_{j,\varepsilon}^{(\eta_3)}=r_{j,\varepsilon}
\end{equation}
for all $j\in J_i$ and all $\varepsilon\ll 1$. Coming back otherwise to \eqref{DefRiEpsEta3}-\eqref{BarRDef3}, up to a subsequence, we may assume  by contradiction that there exists $j\in J_i$ such that
$$\frac{2 t_{j,\varepsilon}(r_{j,\varepsilon})}{p_\varepsilon \gamma_{j,\varepsilon}^{p_\varepsilon}}\ge 1+o(1) $$
as $\varepsilon\to 0$. As a remark, we must have $j\neq i$ by \eqref{ContradCons3}. Then, for all given $\eta_2\in(0,1)$, \eqref{EncadrRjEta1} holds true and the argument between \eqref{EncadrRjEta1} and \eqref{PropGammaSim3} gives \eqref{HighContrad3}, which does not occur by \eqref{PropGammaSim3} and proves \eqref{ComplContradCons3}. For $j\in J_i$, since $$\phi_{i,\varepsilon}\circ\phi_{j,\varepsilon}^{-1}\left(\partial B_{r_{j,\varepsilon}/2}(0) \right)\subset \Omega_{i,\varepsilon}$$ by \eqref{DefRIEps3}, \eqref{LocEucl3}, \eqref{ComplObviousProp} and the definition of $\tau$, we get from the equality in \eqref{LocBd3}
$$\bar{u}_{j,\varepsilon}(r_{j,\varepsilon}/2)=\bar{u}_{i,\varepsilon}(\tau r_{i,\varepsilon})+O\left(\gamma_{i,\varepsilon}^{p_\varepsilon-1} \right)\,, $$
so that we eventually have
\begin{equation}\label{ComplPropGammaSim3}
\gamma_{i,\varepsilon}=O\left(\gamma_{j,\varepsilon} \right)\,,
\end{equation}
using the inequality in \eqref{LocBd3} and since $\bar{u}_{j,\varepsilon}(\bar{r}_{j,\varepsilon}^{(\eta_3)}/2)\le 2\gamma_{j,\varepsilon}$ by \eqref{IneqViEps3}, \eqref{CorSect23} and \eqref{ComplContradCons3}, for all $\varepsilon\ll 1$. 

\smallskip

\noindent \textbullet~We are now in position to conclude the proof of \eqref{FarEnoughEq3}. Setting $$\tilde{u}_{\varepsilon}:=\gamma_{i,\varepsilon}^{p_\varepsilon-1} \left(u_{i,\varepsilon}(r_{i,\varepsilon}\cdot)-\bar{u}_{i,\varepsilon}(r_{i,\varepsilon})\right)\,,$$
with an argument similar to the proof of \eqref{LocBd3} one deduces from \eqref{WeakGradEst} and \eqref{LocEucl3} that $(\tilde{u}_{\varepsilon})_\varepsilon$ is uniformly locally bounded in $\mathbb{R}^2\backslash \mathcal{S}_i$ for all $\varepsilon\ll 1$, where $\mathcal{S}_i$ is given below \eqref{LocEucl3}. Then, using \eqref{IneqViEps3} and \eqref{CorSect23} for $l=i$ with \eqref{ContradCons3}, we get from \eqref{ConVarphi3} and \eqref{EqUIEps3} that
$$\Delta \tilde{u}_\varepsilon=O\left(\gamma_{i,\varepsilon}^{p_\varepsilon} r_{i,\varepsilon}^2\right)+O\left(r_{i,\varepsilon}^2 \lambda_\varepsilon\left( \gamma_{i,\varepsilon}^{p_\varepsilon-1}  v_{i,\varepsilon}^{p_\varepsilon-1} e^{v_{i,\varepsilon}^{p_\varepsilon}}\right)(r_{i,\varepsilon\cdot})\right)=o(1) $$
uniformly locally in $\mathbb{R}^2\backslash \mathcal{S}_i$ for all $\varepsilon\ll 1$. To get the last estimate, we use \eqref{BarRITo03} for $l=i$ to control the first term, while we estimate the second one first by $O(\left(\mu_{i,\varepsilon}/r_{i,\varepsilon} \right)^{2(1-\tilde{\eta})} )$ (see Proposition \ref{PropRadAnalysis1}) and then we conclude with \eqref{RMuITo03}. Hence, there exists a harmonic function $\tilde{u}_0$ such that $\tilde{u}_\varepsilon\to \tilde{u}_0$ in $C^1_{loc}(\mathbb{R}^2\backslash \mathcal{S}_i)$ as $\varepsilon\to 0$. Now observe that \eqref{WeakGradEst} also gives the existence of $C>0$ such that
$$|\nabla \tilde{u}_0|\le C\sum_{x\in \mathcal{S}_i}\frac{1}{|x-\cdot|}\text{ in }\mathbb{R}^2\backslash \mathcal{S}_i\,, $$
using the local convergence of the $\tilde{u}_\varepsilon$'s in $\mathbb{R}^2\backslash \mathcal{S}_i$ and the lower estimate in \eqref{IneqViEps3} for $l=i$. Then, by harmonic function's theory, there exist real numbers $\alpha_x$ and $\Lambda$ such that
\begin{equation}\label{DefTildeU3}
\tilde{u}_0=\Lambda+\sum_{x\in \mathcal{S}_i} \alpha_x \ln \frac{1}{|x-\cdot|}\text{ in }\mathbb{R}^2\backslash \mathcal{S}_i\,. 
\end{equation}
However, by \eqref{ObviousProp} and \eqref{ComplObviousProp}, by \eqref{PropGammaSim3} and \eqref{ComplPropGammaSim3}, Proposition \ref{PropRadCompar2} gives that the $\alpha_x$ are positive and in particular \eqref{CondH2} gives that
$$\nabla\left(\tilde{u}_0-\alpha_x \ln \frac{1}{|x-\cdot|} \right)(x)=0  $$
for all $x\in \mathcal{S}_i$. Picking now $y$ an extreme point of the convex hull of $\mathcal{S}_i$, we get from \eqref{DefTildeU3} that this last property fails for $x=y$, since $\mathcal{S}_i$ possesses at least two points. This gives the expected contradiction to \eqref{ContradCons3} and concludes the proof of \eqref{FarEnoughEq3}.
\end{proof}
\noindent Step \ref{StFarEnough3} is proven.
\end{proof}

Up to a subsequence, we assume from now on that
\begin{equation}\label{ConvP04}
\lim_{\varepsilon\to 0} p_\varepsilon=p_0\,,
\end{equation}
for some $p_0\in [1,2]$. As a first consequence of Step \ref{StFarEnough3}, we improve \eqref{BdLambdaEps3} and  conclude the proof of \eqref{Quantization} and thus that of Theorem \ref{ThmBlowUpAnalysis} in the subcritical case. A key ingredient to get the sharp quantization \eqref{Quantization} (and not \eqref{BadQuantization} for $u_0\not \equiv 0$, for instance) is given by \eqref{MajorGlobal3} in Step \ref{StFarEnough3}: roughly speaking, the only way for the RHS of \eqref{MajorGlobal3} to be positive at some $r$ not too small is that \emph{$\lambda_\varepsilon$ is quite small} (see \eqref{PosIneq4} and \eqref{MajorGlobal3UsedHere} below).

\begin{Step}\label{StSubcriticalCase}
In any case, we have that 
\begin{equation}\label{Lim0Lambda4}
\lim_{\varepsilon\to 0}\lambda_\varepsilon=0\,.
\end{equation}
Moreover, assuming that $p_0\in [1,2)$, \eqref{Quantization} holds true for $k=N$ and $N$ given by Proposition \ref{PropWeakPwEst}.
\end{Step}

\begin{proof}[Proof of Step \ref{StSubcriticalCase}]
By evaluating \eqref{MajorGlobal3} at $r=\kappa \gamma_{i,\varepsilon}^{2(1-p_\varepsilon)/3}$, we get that
\begin{equation}\label{PosIneq4}
\left(1-\frac{p_\varepsilon}{2} \right)\gamma_{i,\varepsilon}^{p_\varepsilon}+\frac{2}{3}(p_\varepsilon-1)\ln \gamma_{i,\varepsilon}\le \ln \frac{1}{\lambda_\varepsilon}+O(1)
\end{equation}
for all $\varepsilon\ll 1$ and all $i\in \{1,...,N\}$, which clearly proves \eqref{Lim0Lambda4}. Now assume that $p_0<2$ in \eqref{ConvP04}. Up to renumbering, fix $i$ such that $\gamma_{i,\varepsilon}$ is the largest of the $\gamma_{j,\varepsilon}$'s for all $\varepsilon\ll 1$ and all $j$. Given any $\eta\in (0,1)$ to be chosen later, setting $r_{l,\varepsilon}^{(\eta)}$ as in \eqref{DefRiEpsEta3}, we know from \eqref{FarEnoughEq3} that $\bar{r}_{l,\varepsilon}^{(\eta)}={r}_{l,\varepsilon}^{(\eta)}$ for all $\varepsilon\ll 1$ and all $l$. Then, we get from \eqref{IneqViEps3} and \eqref{CorSect23} (see also Proposition \ref{PropRadAnalysis1}) that
\begin{equation}\label{CC4}
\begin{split}
&\int_{B_{{r}_{l,\varepsilon}^{(\eta)}}(0)}  \frac{\lambda_\varepsilon p_\varepsilon^2}{2} u_{l,\varepsilon}^{p_\varepsilon} e^{u_{l,\varepsilon}^{p_\varepsilon}}~ e^{2 \varphi_{l,\varepsilon}}dx= \left(4\pi+o(1)\right)\gamma_{l,\varepsilon}^{2-p_\varepsilon}\,,\text{ that}\\
&\int_{B_{{r}_{l,\varepsilon}^{(\eta)}}(0)}  \frac{\lambda_\varepsilon p_\varepsilon^2}{2} e^{u_{l,\varepsilon}^{p_\varepsilon}}~ e^{2 \varphi_{l,\varepsilon}}dx= \frac{4\pi+o(1)}{\gamma_{l,\varepsilon}^{2(p_\varepsilon-1)}}
\end{split}
\end{equation}
 and that 
\begin{equation}\label{ThisLastPoint} 
\begin{split}
 u_{l,\varepsilon} ~&~=(1-\eta) \gamma_{l,\varepsilon}+O\left(\gamma_{l,\varepsilon}^{1-p_\varepsilon} \right)\,,\\
 ~&~=-\left(\frac{2}{p_\varepsilon}-1 \right)\gamma_{l,\varepsilon}+\frac{2}{p_\varepsilon \gamma_{l,\varepsilon}^{p_\varepsilon-1}}\ln \frac{1}{\lambda_\varepsilon (r_{l,\varepsilon}^{(\eta)})^2}+O(1)
 \end{split}
\end{equation}
 uniformly in $\partial B_{{r}_{l,\varepsilon}^{(\eta)}}(0)$ for all $\varepsilon\ll 1$ and for all $l$. The second equality uses also \eqref{LowPOVBubble1} with $\gamma=\gamma_{l,\varepsilon}$ and $p_\gamma=p_\varepsilon$. Up to a subsequence, by comparing the two RHS of \eqref{ThisLastPoint}, by using $p_0<2$, $r_{l,\varepsilon}^{(\eta)}\le \kappa$ and that $\eta$ (moving only here) may be arbitrarily close to $1$, we may complement \eqref{PosIneq4} here and get that 
 \begin{equation}\label{EstiLambda}
 \left(1-\frac{p_0}{2} \right)\gamma_{l,\varepsilon}^{p_\varepsilon}(1+o(1))=\ln \frac{1}{\lambda_\varepsilon}
 \end{equation}
for all $l$ as $\varepsilon\to 0$, so that we have in particular
\begin{equation}\label{GapFixed}
\gamma_{l,\varepsilon}=(1+o(1)) \gamma_{i,\varepsilon}
\end{equation}
for all $l$. Given any $\check{\eta}\in(0,\eta)$, we claim that the first equality in \eqref{ThisLastPoint} implies that 
 \begin{equation}\label{Star4}
 u_\varepsilon\le \left(1-\check{\eta} \right) \gamma_{i,\varepsilon} \text{ in }\Omega_\varepsilon:=\Sigma\backslash \cup_{l=1}^{N} \phi_{l,\varepsilon}^{-1}\left(B_{{r}_{l,\varepsilon}^{(\eta)}}(0) \right)
 \end{equation}
 for all $\varepsilon\ll 1$. Otherwise, as when proving Proposition \ref{PropWeakPwEst}, if $x_\varepsilon\in \Omega_\varepsilon$ satisfies $u_\varepsilon(x_\varepsilon)=\max_{\Omega_\varepsilon} u_\varepsilon$, then $x_\varepsilon$ is a good candidate to be another concentration point for $u_\varepsilon$: we get that $\mu_{l,\varepsilon}=o\left(d_g(x_{l,\varepsilon}, x_\varepsilon) \right)$ for all $l$ by \eqref{Convloc3} and that $\min_{l\in\{1,...,N\}} u_\varepsilon^{p_\varepsilon-1}(x_\varepsilon) d_g(x_{l,\varepsilon},x_\varepsilon)^2|(\Delta_g u_\varepsilon)(x_\varepsilon)|\to +\infty$ as $\varepsilon\to 0$, which contradicts \eqref{WeakPointwEst} and establishes \eqref{Star4}. Independently, \eqref{EstiLambda} gives
$$\lambda_\varepsilon \int_{\Omega_\varepsilon}  \left(1+u_\varepsilon^{p_\varepsilon}\right) e^{u_\varepsilon^{p_\varepsilon}} dv_g=O\left(\exp\left(\left((1-\check{\eta})^{p_0}-\left(1-\frac{p_0}{2} \right) \right)\gamma_{i,\varepsilon}^{p_\varepsilon}+o(\gamma_{i,\varepsilon}^{p_\varepsilon}) \right) \right) $$
for all $\varepsilon\ll 1$. Choosing $0<\check{\eta}<\eta<1$ sufficiently close to $1$ from the beginning (depending on the smallness of $2-p_0>0$ here), we may plug  this estimate and \eqref{CC4} in \eqref{BetaEps} to conclude the proof of \eqref{Quantization}, using also \eqref{Conjugate} and \eqref{GapFixed}. 
\end{proof}

In contrast to the case $p_0=2$ handled below (see also \cite{DruetDuke}), it is interesting to note that, due to the global nature of both integrals in \eqref{BetaEps}, we need also \eqref{GapFixed} to get the quantization \eqref{Quantization}, at least for $k>1$ and $1<p_0<2$ in \eqref{ConvP04}. At that stage, we are left with the proof of \eqref{Quantization} in the more delicate borderline case $p_0=2$. \emph{We assume from now on that $p_0=2$ in \eqref{ConvP04}.}

\begin{proof}[Conclusion of the proof of Theorem \ref{ThmBlowUpAnalysis}] 
We still use the notation and observations of \eqref{DefRiEpsEta3}-\eqref{BarRDef3} and below. On the other hand, by \eqref{FarEnoughEq3} in Step \ref{StFarEnough3}, for all given $\eta\in (0,1)$, we have that
\begin{equation}\label{PrelRem4}
r_{l,\varepsilon}^{(\eta)}=o(r_{l,\varepsilon})\quad \implies \quad \bar{r}_{l,\varepsilon}^{(\eta)}=r_{l,\varepsilon}^{(\eta)}
\end{equation}
for all $\varepsilon\ll 1$ and all $l\in \{1,...,N\}$. Then, as a consequence of Propositions \ref{PropRadAnalysis1} and \ref{PropRadCompar2}, we get that \eqref{BarRITo03}-\eqref{GradCompl3} hold true. In particular, for all given $\eta'<\eta$ in $(0,1)$, we get from \eqref{GradCompl3} that
\begin{equation}\label{GradEstSlightImprov4}
|\nabla (u_{l,\varepsilon}-v_{l,\varepsilon})|=o\left(\frac{1}{\gamma_{l,\varepsilon}^{p_\varepsilon-1} r_{l,\varepsilon}^{(\eta')}} \right)
\end{equation}
uniformly in $B_{r_{l,\varepsilon}^{(\eta')}}(0)$ for all $\varepsilon\ll 1$ and all $l$. Then, for all given $\eta'\in (0,1)$, since we also have
$$0\le v_{l,\varepsilon}-v_{l,\varepsilon}\left(r_{l,\varepsilon}^{(\eta')} \right)\le \frac{2+o(1)}{\gamma_{l,\varepsilon}^{p_\varepsilon-1}} \ln \frac{r_{l,\varepsilon}^{(\eta')}}{|\cdot|}\,, $$
using the estimate in $w'_\gamma$ in Proposition \ref{PropRadAnalysis1}, we eventually get that
\begin{equation}\label{MajorGrad4}
\left|u_{l,\varepsilon}-\overline{u}_{l,\varepsilon}\left(r_{l,\varepsilon}^{(\eta')} \right)\right|\le \frac{2+o(1)}{\gamma_{l,\varepsilon}^{p_\varepsilon-1}} \ln \frac{2r_{l,\varepsilon}^{(\eta')}}{|\cdot|}
\end{equation}
uniformly in $B_{r_{l,\varepsilon}^{(\eta')}}(0)\backslash \{0\}$ for all $\varepsilon\ll 1$ and all $l$. During the whole proof below, we choose and fix $\eta_0\in (0,1)$ and set
\begin{equation}\label{DefNuJEps4}
\begin{split}
&\nu_{j,\varepsilon}=\\
&\sup\left\{ r\in \left(r_{j,\varepsilon}^{(\eta_0)},\kappa \right]\text{ s.t. }
\begin{cases}
& \left|u_{j,\varepsilon}-\bar{u}_{j,\varepsilon}(r) \right|\\
&<5\Big(\pi C_2\bar{u}_{j,\varepsilon}(r)^{1-p_\varepsilon}\\
&\quad \quad\quad +2 \sum_{l\in I_{j,\varepsilon}(r)} \gamma_{l,\varepsilon}^{1-p_\varepsilon}\ln \frac{ 6 r}{|\cdot-\phi_{j,\varepsilon}(x_{l,\varepsilon})|} \Big)\\
&\text{in }B_r(0)\backslash \cup_{l\in I_{j,\varepsilon}(r)} B_{r_{l,\varepsilon}^{(\eta_0)}}\left(\phi_{j,\varepsilon}(x_{l,\varepsilon}) \right)
\end{cases}
 \right\}
 \end{split}
\end{equation}
for all $j\in \{1,...,N\}$ and all $\varepsilon\ll 1$, where $C_2>0$ is as in \eqref{WeakGradEst} and where $I_{j,\varepsilon}(r)$ is given by
$$I_{j,\varepsilon}(r)=\left\{l\in \{1,...,N\}\text{ s.t. }\phi_{j,\varepsilon}(x_{l,\varepsilon})\in B_{\frac{3 r}{2} }(0) \right\}\,. $$
As a first remark, it follows from the very definition \eqref{DefNuJEps4} of $\nu_{j,\varepsilon}$ and from \eqref{MajorGrad4} that, for all given $\eta_2\in [\eta_0,1)$, we have 
\begin{equation}\label{IneqNu4}
\nu_{l,\varepsilon}\ge r_{l,\varepsilon}^{(\eta_2)}
\end{equation}
for all $\varepsilon\ll 1$ and all $l\in \{1,...,N\}$. Our main goal now is to show that
\begin{equation}\label{BdNuJ4}
\bar{u}_{j,\varepsilon}(\nu_{j,\varepsilon})=O(1)
\end{equation}
for all $\varepsilon\ll 1$ and all $j\in \{1,...,N\}$. For all $j$, we may assume up to a subsequence that either \eqref{BdNuJ4} or 
\begin{equation}\label{ContradBdNu4}
\lim_{\varepsilon\to 0}\bar{u}_{j,\varepsilon}(\nu_{j,\varepsilon})= +\infty
\end{equation}
hold true. Assume from now on by contradiction that \eqref{BdNuJ4} does not hold true for all $j$ so that we may choose and fix $i\in \{1,...,N\}$ such that
\begin{equation}\label{NuIEpsDef4}
\nu_{i,\varepsilon}=\min \left\{\nu_{j,\varepsilon}\text{ s.t. \eqref{ContradBdNu4} holds true} \right\}\,.
\end{equation}
Clearly, we then have
\begin{equation}\label{LimBarU4}
\lim_{\varepsilon\to 0} \bar{u}_{i,\varepsilon}(\nu_{i,\varepsilon})=+\infty\,.
\end{equation}
By \eqref{LocConv3}, we also have that
\begin{equation}\label{NuITo0}
\lim_{\varepsilon \to 0}\nu_{i,\varepsilon}=0\,,
\end{equation}
so that, using \eqref{ConVarphi3}, the following property currently used in the sequel holds true:
\begin{equation}\label{LocEucl4}
\tilde{g}_{l,\varepsilon}:=\left(\left(\phi_{l,\varepsilon}\right)_\star g \right)(\nu_{i,\varepsilon \cdot})\to \xi\text{ in }C^2_{loc}(\mathbb{R}^2)
\end{equation}
as $\varepsilon\to 0$, for all $l\in I$, where 
$$I:=\left\{l\in \{1,...,N\}\text{ s.t. }d_g(x_{i,\varepsilon}, x_{l,\varepsilon})=O\left(\nu_{i,\varepsilon} \right)\text{ for all }\varepsilon\ll 1 \right\}\,. $$
Up to a further subsequence, we may also assume that
$$\lim_{\varepsilon\to 0} \frac{\phi_{i,\varepsilon}(x_{l,\varepsilon})}{\nu_{i,\varepsilon}}=\tilde{x}_l\in \mathbb{R}^2 $$
for all $l\in I$. Set also $\mathcal{S}=\{\tilde{x}_l~|~ l\in I \}$ so that clearly $0\in \mathcal{S}$. Fix $\tau\in (0,1)$ and $R\ge 1$ to be chosen properly later on such that
$$3 \tau<
\begin{cases}
&1\text{ if }\mathcal{S}=\{0\}\,,\\
&\min_{\{(x,y)\in \mathcal{S}^2|x\neq y \}}|x-y|\text{ otherwise}\,,
\end{cases} $$
and such that $\mathcal{S}\subset B_{3 R}(0)$. Set $D_\varepsilon=B_{R \nu_{i,\varepsilon}}(0)\backslash \cup_{l\in I} B_{\tau \nu_{i,\varepsilon}/3}(\phi_{i,\varepsilon}(x_{l,\varepsilon}))$ for all $\varepsilon\ll 1$. Let now $\tilde{w}_\varepsilon$ be given by 
\begin{equation}\label{TildeW4}
\begin{cases}
&\Delta \tilde{w}_\varepsilon=-e^{2 \varphi_{i,\varepsilon}} {h_{i,\varepsilon}} u_{i,\varepsilon}\text{ in }B_{R\nu_{i,\varepsilon}}(0)\,,\\
&\tilde{w}_\varepsilon=0\text{ on }\partial B_{R\nu_{i,\varepsilon}}(0)\,,
\end{cases}
\end{equation}
for all $\varepsilon$. Observe first by \eqref{EqUIEps3}
that $\Delta(u_{i,\varepsilon}-\tilde{w}_\varepsilon)\ge 0$ in $B_{R\nu_{i,\varepsilon}}(0)$ so that $\overline{u_{i,\varepsilon}-\tilde{w}_\varepsilon}$ is radially nonincreasing in $[0,R\nu_{i,\varepsilon}]$. Moreover, the maximum principle gives that $u_{i,\varepsilon}-\tilde{w}_\varepsilon$ attains its infimum in $B_{R\nu_{i,\varepsilon}}(0)$ at some point on $\partial B_{R\nu_{i,\varepsilon}}(0)$. Independently, for all given $p>2$, by elliptic theory, we get from Lemma \ref{CorPropWeakPwEst} and \eqref{ConVarphi3} that
\begin{equation}\label{LpBound4}
\|\tilde{w}_\varepsilon(\nu_{i,\varepsilon}\cdot)\|_{L^\infty(B_R(0))}=O\left(\|\Delta \left(\tilde{w}_\varepsilon (\nu_{i,\varepsilon} \cdot) \right)\|_{L^p(B_R(0))} \right)=O\left(\nu_{i,\varepsilon}^{\frac{2(p-1)}{p}} \right) 
\end{equation}
for all $\varepsilon\ll 1$. Summarizing, by \eqref{NuITo0} and since $\tau<1$, this argument for $R=1$ (only there) gives that $\bar{u}_{i,\varepsilon}(\tau \nu_{i,\varepsilon})\ge \bar{u}_{i,\varepsilon}(\nu_{i,\varepsilon})+o(1)$, so that \eqref{LimBarU4} leads to
\begin{equation}\label{GammaDef4}
\Gamma_\varepsilon:=\bar{u}_{i,\varepsilon}(\tau \nu_{i,\varepsilon})\to +\infty
\end{equation}
 as $\varepsilon\to 0$. Then, as a consequence of \eqref{WeakGradEst} (and \eqref{LocEucl4} again), we get that
 \begin{equation}\label{UnifEquiv4}
 u_{i,\varepsilon}=\Gamma_\varepsilon+O\left(\Gamma_\varepsilon^{1-p_\varepsilon} \right)
 \end{equation}
uniformly in $D_\varepsilon$ and for all $\varepsilon\ll 1$, using once more the mean value property on $\partial B_{\tau \nu_{i,\varepsilon}}(0)$ and the definition of $\tau$. Then, by the maximum principle-based argument below \eqref{TildeW4}, with \eqref{NuITo0} and \eqref{LpBound4}, we get that
\begin{equation}\label{MinorGlob4}
\inf_{B_{R\nu_{i,\varepsilon}}(0)} u_{i,\varepsilon}\ge \min_{\partial B_{R \nu_{i,\varepsilon}}(0)} u_{i,\varepsilon}+o(1)=\Gamma_\varepsilon+o(1)
\end{equation}
as $\varepsilon\to 0$. 

\smallskip

We prove now that
\begin{equation}\label{RatioGamma4}
\Gamma_\varepsilon=o(\gamma_{j,\varepsilon})
\end{equation}
as $\varepsilon\to 0$, for all $j\in I$, up to a subsequence. Consider first the case $j=i$ in \eqref{RatioGamma4}. For all given $\eta_2\in [\eta_0,1)$, we have that
\begin{equation}\label{InterMP4}
\bar{u}_{j,\varepsilon}\left(r_{j,\varepsilon}^{(\eta_2)}\right)=(1-\eta_2) \gamma_{j,\varepsilon}(1+o(1))\ge \Gamma_\varepsilon(1+o(1)) 
\end{equation}
for all $\varepsilon\ll 1$. The first equality comes from the definition \eqref{DefRiEpsEta3} of $r_{i,\varepsilon}^{(\eta_2)}$, from \eqref{PrelRem4}, from the equality in \eqref{IneqViEps3} and from \eqref{CorSect23} for $l=i$, while the inequality comes from \eqref{IneqNu4}, \eqref{MinorGlob4} and the above largeness assumption $\mathcal{S}\subset B_{3R}(0)$ on $R\gg 1$. Observe that \eqref{IneqNu4} implies that \eqref{IntermObservation4} below holds true for $t=i$. Since $\eta_2$ may be arbitrarily close to $1$, \eqref{InterMP4} concludes the proof of \eqref{RatioGamma4}
 for $j=i$. If now $I\neq \{i\}$, we may pick $j\in I\backslash\{i\}$ and we get from the very definition of $I$ with \eqref{DefRIEps3} and \eqref{LocEucl4} again that $r_{j,\varepsilon}=O\left(\nu_{i,\varepsilon} \right)$ for all $\varepsilon$. Then, also in the last present case $j\neq i$, using now  \eqref{PrelRem4} for $l=j$, we get \begin{equation}\label{IntermObservation4}
  \lim_{\varepsilon\to 0}\frac{r_{t,\varepsilon}^{(\eta_0)}}{\nu_{i,\varepsilon}}=0\,, 
 \end{equation}
  for all $t\in I$, and then similarly \eqref{InterMP4}, to conclude the proof of \eqref{RatioGamma4}.
  
\smallskip  
  
   At that stage, we may improve the estimate in \eqref{LpBound4}. As a consequence of \eqref{GammaDef4}, \eqref{UnifEquiv4} and Lemma \ref{CorPropWeakPwEst}, writing merely that $\|u_{i,\varepsilon}\|_{L^p(D_\varepsilon)}=O(1)$, we get that $\nu_{i,\varepsilon}^2 \Gamma_\varepsilon^p=O(1)$ for all $\varepsilon$, so that \eqref{LpBound4} gives 
 \begin{equation}\label{TildeWInter4}
 |\tilde{w}_\varepsilon|=O\left(\Gamma_\varepsilon^{1-p} \right)=o(\Gamma_\varepsilon^{1-p_\varepsilon})
 \end{equation}
uniformly in $D_\varepsilon$, for all $\varepsilon\ll 1$, since $p$ is fixed greater than $2$ just above \eqref{LpBound4}. Let $\zeta_\varepsilon$ be given by
$$\begin{cases}
&\Delta \zeta_\varepsilon=0\text{ in }B_{R\nu_{i,\varepsilon}}(0)\,,\\
&\zeta_\varepsilon=u_{i,\varepsilon}\text{ on }\partial B_{R\nu_{i,\varepsilon}}(0)
\end{cases} $$
for all $\varepsilon$. By keeping track of the constant $C_2$ of \eqref{WeakGradEst} and choosing $R\gg 1$ large enough (depending only on $\mathcal{S}$) from the beginning, using a mean value theorem on $\partial B_{R\nu_{i,\varepsilon}}(0)$, \eqref{LocEucl4} and \eqref{GammaDef4}, we may get a slightly more precise version of \eqref{UnifEquiv4} on $\partial B_{R \nu_{i,\varepsilon}}(0)$, namely we have that
\begin{equation}\label{ContrHarmPart4}
\sup_{B_{R\nu_{i,\varepsilon}}(0)}|\zeta_\varepsilon-\bar{u}_{i,\varepsilon}(R\nu_{i,\varepsilon})|\le \sup_{\partial B_{R\nu_{i,\varepsilon}}(0)} |u_{i,\varepsilon}-\bar{u}_{i,\varepsilon}(R\nu_{i,\varepsilon})|\le \frac{2\pi C_2}{\Gamma_\varepsilon^{p_\varepsilon-1}}
\end{equation}
for all $\varepsilon\ll 1$, using also the maximum principle. Observe in particular that $u_{i,\varepsilon}=(1+o(1))\Gamma_\varepsilon$ uniformly in $D_\varepsilon$. Let $\tilde{G}_\varepsilon$ be the Green's function of $\Delta$ in $B_{R\nu_{i,\varepsilon}}(0)$ with zero Dirichlet boundary condition. Let $(z_\varepsilon)_\varepsilon$ be any sequence of points such that 
\begin{equation}\label{CondZ4}
z_\varepsilon\in \overline{B_{R\nu_{i,\varepsilon}}(0))\backslash \cup_{l\in I}B_{r_{l,\varepsilon}^{(\eta_0)}}\left(\phi_{i,\varepsilon}(x_{l,\varepsilon}) \right)}
\end{equation}
 for all $\varepsilon$. We have that
 $$0<\tilde{G}_\varepsilon(z_\varepsilon,\cdot)\le \frac{1}{2\pi} \ln\frac{2R \nu_{i,\varepsilon}}{|z_\varepsilon-\cdot|} $$
 in $B_{R{\nu_{i,\varepsilon}}}(0)\backslash \{z_\varepsilon\}$ for all $\varepsilon\ll 1$. Thus, the Green's reprentation formula gives that
 \begin{equation}\label{GF4}
 \begin{split}
 &0\le \left(u_{i,\varepsilon}-\tilde{w}_\varepsilon-\zeta_\varepsilon\right)(z_\varepsilon)\\
 &\quad \quad \quad \quad \le \frac{\lambda_\varepsilon p_\varepsilon}{2\pi} \int_{B_{R\nu_{i,\varepsilon}}(0)} \ln \frac{2 R\nu_{i,\varepsilon}}{|z_\varepsilon-y|} \left(e^{2\varphi_{i,\varepsilon}} u_{i,\varepsilon}^{p_\varepsilon-1} e^{u_{i,\varepsilon}^{p_\varepsilon}}\right)(y)~ dy
 \end{split}
 \end{equation}
for all $\varepsilon$, using \eqref{EqUIEps3}. Using Step \ref{StFarEnough3} as above to employ Proposition \ref{PropRadAnalysis1} and \eqref{GradCompl3}, we have that for all $l\in \{1,...,N\}$
$$|\nabla u_{l,\varepsilon}|=O\left(\frac{1}{r_{l,\varepsilon}^{(\eta_0)} \gamma_{l,\varepsilon}^{p_\varepsilon-1}}  \right)\text{ uniformly in }B_{3 r_{l,\varepsilon}^{(\eta_0)}}\left(0 \right)\backslash B_{\frac{r_{l,\varepsilon}^{(\eta_0)}}{3}}\left( 0 \right) $$
so that, for all $j\in I$, we get as a byproduct of \eqref{DefNuJEps4} and \eqref{NuIEpsDef4} with $\tau<1$ that
$$\left|\bar{u}_{j,\varepsilon}(\tau \nu_{i,\varepsilon})-u_{j,\varepsilon} \right|=O\left(\bar{u}_{j,\varepsilon}(\tau \nu_{i,\varepsilon})^{1-p_\varepsilon} \right)+O\left(\sum_{l\in I_{j,\varepsilon}(\tau \nu_{i,\varepsilon})} \frac{1}{\gamma_{l,\varepsilon}^{p_\varepsilon-1}}\ln \frac{4\tau \nu_{i,\varepsilon}}{|\cdot-\phi_{j,\varepsilon}(x_{l,\varepsilon})|}\right) $$
uniformly in $B_{\tau \nu_{i,\varepsilon}}(0)\backslash \cup_{l\in I_{j,\varepsilon}(\tau \nu_{i,\varepsilon})} B_{2 r_{l,\varepsilon}^{(\eta_0)}/5}\left(\phi_{j,\varepsilon}(x_{l,\varepsilon}) \right)$, and then we eventually obtain with \eqref{UnifEquiv4} and our definition of $\tau$ that
\begin{equation}\label{InterSuite4}
|\Gamma_\varepsilon-u_{i,\varepsilon}|=O\left(\Gamma_\varepsilon^{1-p_\varepsilon}\right)+O\left(\sum_{l\in I_{j,\varepsilon}(\tau \nu_{i,\varepsilon})} \frac{1}{\gamma_{l,\varepsilon}^{p_\varepsilon-1}}\ln \frac{4\tau \nu_{i,\varepsilon}}{|\cdot-\phi_{i,\varepsilon}(x_{l,\varepsilon})|} \right)
\end{equation}
uniformly in $D_{j,\varepsilon}:=B_{\tau \nu_{i,\varepsilon}/2}(\phi_{i,\varepsilon}(x_{j,\varepsilon}))\backslash \cup_{l\in I_{j,\varepsilon}(\tau \nu_{i,\varepsilon})} B_{ r_{l,\varepsilon}^{(\eta_0)}/2 }\left(\phi_{i,\varepsilon}(x_{l,\varepsilon})\right)$ for all $\varepsilon$, still using \eqref{LocEucl4}. Independently, using that $|z_\varepsilon-\phi_{i,\varepsilon}(x_{l,\varepsilon})|\ge r_{l,\varepsilon}^{(\eta_0)}$, we have 
\begin{equation}\label{SplittingGreen4}
\ln \frac{2 R \nu_{i,\varepsilon}}{|z_\varepsilon-\cdot|}=\ln \frac{2 R \nu_{i,\varepsilon}}{|z_\varepsilon-\phi_{i,\varepsilon}(x_{l,\varepsilon})|}+O\left(\frac{r_{l,\varepsilon}^{(\eta_0)}}{r_{l,\varepsilon}^{(\eta_0)}+|z_\varepsilon-\phi_{i,\varepsilon}(x_{l,\varepsilon})|}\right)
\end{equation}
uniformly in $B_{r_{l,\varepsilon}^{(\eta_0)}/2 }\left(\phi_{i,\varepsilon}(x_{l,\varepsilon})\right)$ and for all $\varepsilon\ll 1$. By \eqref{GradEstSlightImprov4} for some given $\eta'\in (\eta_0,1)$ and since $u_{l,\varepsilon}(0)=v_{l,\varepsilon}(0)$, we get  $u_{l,\varepsilon}-v_{l,\varepsilon}=o(\gamma_{l,\varepsilon}^{1-p_\varepsilon})$, so  we eventually get for all given $\tilde{\eta}\in (\eta_0,1)$ that
\begin{equation}\label{UnifNonlin4}
\lambda_\varepsilon p_\varepsilon u_{l,\varepsilon}^{p_\varepsilon-1} e^{u_{l,\varepsilon}^{p_\varepsilon}}=\frac{8 e^{-2 t_{l,\varepsilon}}(1+o(e^{\tilde{\eta} t_{l,\varepsilon}}))}{\mu_{l,\varepsilon}^2 \gamma_{l,\varepsilon}^{p_\varepsilon-1} p_\varepsilon}
\end{equation} 
uniformly in $B_{r_{l,\varepsilon}^{(\eta_0)}}(0)$ and for all $\varepsilon\ll 1$, still  applying Proposition \ref{PropRadAnalysis1}. Then, using also \eqref{ConVarphi3} and \eqref{LocEucl4}, we get from \eqref{SplittingGreen4} and \eqref{UnifNonlin4} that
\begin{equation}\label{SingSource4}
\begin{split}
&\frac{\lambda_\varepsilon p_\varepsilon}{2\pi} \int_{B_{r_{l,\varepsilon}^{(\eta_0)}/2}\left(\phi_{i,\varepsilon}(x_{l,\varepsilon}) \right)} \ln \frac{2 R\nu_{i,\varepsilon}}{|z_\varepsilon-y|} \left(e^{2\varphi_{i,\varepsilon}} u_{i,\varepsilon}^{p_\varepsilon-1} e^{u_{i,\varepsilon}^{p_\varepsilon}}\right)(y)~ dy\\
&= \frac{(4+o(1))}{p_\varepsilon \gamma_{l,\varepsilon}^{p_\varepsilon-1}}\ln \frac{ \nu_{i,\varepsilon}}{|z_\varepsilon-\phi_{i,\varepsilon}(x_{l,\varepsilon})|}+O\left(\frac{1}{\gamma_{l,\varepsilon}^{p_\varepsilon-1}} \right)\,, 
\end{split}
\end{equation}
 as $\varepsilon\to 0$ and for all $l\in I$. Using the basic inequalities 
 $$|(1+t)^p-1|\le C\left(|t|+|t|^p \right) $$
 for all $t>-1$, and  
 $$\left(\sum_{t=1}^N a_t\right)^p\le C\sum_{t=1}^N a_t^p$$
for all $a_t\ge 0$ and all $p\in [1,2]$, we get first from \eqref{InterSuite4} that
\begin{equation}\label{ComputPower4}
\begin{split}
&u_{i,\varepsilon}^{p_\varepsilon}=\Gamma_\varepsilon^{p_\varepsilon}+O(1)+O\left(\sum_{l\in I_{j,\varepsilon}(\tau \nu_{i,\varepsilon})} \left(\frac{1}{\gamma_{l,\varepsilon}^{p_\varepsilon-1}}\ln \frac{4\tau \nu_{i,\varepsilon}}{|\cdot-\phi_{i,\varepsilon}(x_{l,\varepsilon})|}\right)^{p_\varepsilon}  \right)\\
&\quad \quad +O\left(\sum_{l\in I_{j,\varepsilon}(\tau \nu_{i,\varepsilon})} \left(\frac{\Gamma_\varepsilon}{\gamma_{l,\varepsilon}}\right)^{p_\varepsilon-1}\ln \frac{4\tau \nu_{i,\varepsilon}}{|\cdot-\phi_{i,\varepsilon}(x_{l,\varepsilon})|}  \right)
\end{split}
\end{equation}
uniformly in $D_{j,\varepsilon}$ and for all $\varepsilon$. Independently, we get from \eqref{BdLambdaEps3}, \eqref{Mui3}, \eqref{DefRiEpsEta3} and \eqref{PrelRem4} that
\begin{equation}\label{BasicEst4}
\begin{split}
\ln \frac{1}{\left(r_{l,\varepsilon}^{(\eta_0)}\right)^2} &=-t_{l,\varepsilon}(r_{l,\varepsilon}^{(\eta_0)})+o(1)+\ln \frac{1}{\mu_{l,\varepsilon}^2}\,,\\
&\le \left(-\frac{p_0 \eta_0}{2}+1+o(1)\right)\gamma_{l,\varepsilon}^{p_\varepsilon}\,,
\end{split}
\end{equation}
as $\varepsilon\to 0$ and for all $l$. Recall that we are now assuming that $p_0=2$ in \eqref{ConvP04}. Then, we may get from \eqref{NuITo0}, \eqref{IntermObservation4} and \eqref{BasicEst4} that
\begin{equation}\label{BasicEst24}
\begin{split}
&\left( \frac{1}{\gamma_{l,\varepsilon}^{p_\varepsilon-1}}\ln \frac{4\tau \nu_{i,\varepsilon}}{|\cdot-\phi_{i,\varepsilon}(x_{l,\varepsilon})|^2}\right)^{p_\varepsilon}\\
&= \left(\frac{1}{\gamma_{l,\varepsilon}^{p_\varepsilon}}\ln \frac{4\tau \nu_{i,\varepsilon}}{|\cdot-\phi_{i,\varepsilon}(x_{l,\varepsilon})|^2} \right)^{p_\varepsilon-1} \ln \frac{4\tau \nu_{i,\varepsilon}}{|\cdot-\phi_{i,\varepsilon}(x_{l,\varepsilon})|^2}\,,\\
&\le C(1-\eta_0+o(1))\ln \frac{4\tau \nu_{i,\varepsilon}}{|\cdot-\phi_{i,\varepsilon}(x_{l,\varepsilon})|^2} 
\end{split}
\end{equation}
uniformly in $D_{j,\varepsilon}$ as $\varepsilon\to 0$ and for all $l\in I$. Choose now $j_1,...,j_{|\mathcal{S}|}$ in $I$ such that $\{\tilde{x}_{j_1},...,\tilde{x}_{j_{|\mathcal{S}|}}\}=\mathcal{S}$. We compute and then get from \eqref{ComputPower4}-\eqref{BasicEst24} and from \eqref{RatioGamma4} that
\begin{equation}\label{ComplSingSource4}
\begin{split}
&\frac{\lambda_\varepsilon p_\varepsilon}{2\pi} \int_{D_{j_t,\varepsilon}} \ln \frac{2 R\nu_{i,\varepsilon}}{|z_\varepsilon-y|} \left(e^{2\varphi_{i,\varepsilon}} u_{i,\varepsilon}^{p_\varepsilon-1} e^{u_{i,\varepsilon}^{p_\varepsilon}}\right)(y)~ dy\\
&=O\Bigg( \lambda_\varepsilon \Gamma_\varepsilon^{p_\varepsilon-1}\exp\left(\Gamma_\varepsilon^{p_\varepsilon} \right)\times\\
&\quad \sum_{l\in I_{j_t,\varepsilon}(\tau \nu_{i,\varepsilon})}\int_{B_{\frac{\tau \nu_{i,\varepsilon}}{2}}\left(\phi_{i,\varepsilon}(x_{j_t,\varepsilon})\right)} \ln \frac{2 R\nu_{i,\varepsilon}}{|z_\varepsilon-y|} \left(\frac{4\tau \nu_{i,\varepsilon}}{|y-\phi_{i,\varepsilon}(x_{l,\varepsilon})|^2}\right)^{1-\frac{\eta_0}{2}} dy\Bigg)\,,\\
&=O\left(\lambda_\varepsilon \Gamma_\varepsilon^{p_\varepsilon-1} \exp\left(\Gamma_\varepsilon^{p_\varepsilon} \right) \nu_{i,\varepsilon}^2 \right)\,,
\end{split}
\end{equation}
for all $t\in \{1,...,|\mathcal{S}|\}$ and all $\varepsilon\ll 1$, using that $\eta_0>0$ to get the last estimate. At last, it readily follows from \eqref{UnifEquiv4} that
\begin{equation}\label{LastComplSingSource4}
\begin{split}
&\frac{\lambda_\varepsilon p_\varepsilon}{2\pi} \int_{D_{0,\varepsilon}} \ln \frac{2 R\nu_{i,\varepsilon}}{|z_\varepsilon-y|} \left(e^{2\varphi_{i,\varepsilon}} u_{i,\varepsilon}^{p_\varepsilon-1} e^{u_{i,\varepsilon}^{p_\varepsilon}}\right)(y)~ dy\\
&\quad \quad\quad\quad\quad\quad\quad\quad\quad\quad\quad\quad =O\left(\lambda_\varepsilon \Gamma_\varepsilon^{p_\varepsilon-1} \exp\left(\Gamma_\varepsilon^{p_\varepsilon} \right) \nu_{i,\varepsilon}^2 \right)
\end{split}
\end{equation}
for all $\varepsilon\ll 1$, where $$D_{0,\varepsilon}=B_{R\nu_{i,\varepsilon}}(0)\backslash \cup_{t=1}^{|\mathcal{S}|} B_{{\tau \nu_{i,\varepsilon}}/{2}}\left(\phi_{i,\varepsilon}(x_{j_t,\varepsilon})\right)\,.$$
Summarizing, by plugging \eqref{TildeWInter4}, \eqref{ContrHarmPart4}, \eqref{SingSource4}, \eqref{ComplSingSource4} and \eqref{LastComplSingSource4} in \eqref{GF4}, we get that
\begin{equation}\label{BeforeLastStep4}
\begin{split}
&|u_{i,\varepsilon}(z_\varepsilon)-\bar{u}_{i,\varepsilon}(R\nu_{i,\varepsilon})|\\
&\quad\le 2\pi C_2 \Gamma_\varepsilon^{1-p_\varepsilon}+\sum_{l\in I} \frac{2+o(1)}{p_\varepsilon \gamma_{l,\varepsilon}^{p_\varepsilon-1}} \left(2\ln \frac{ 4 \tau\nu_{i,\varepsilon}}{|z_\varepsilon-\phi_{i,\varepsilon}(x_{l,\varepsilon})|} +O(1)\right)\\
&\quad\quad \quad+O\left(\lambda_\varepsilon \Gamma_\varepsilon^{p_\varepsilon-1} \exp\left(\Gamma_\varepsilon^{p_\varepsilon} \right) \nu_{i,\varepsilon}^2 \right)
\end{split}
\end{equation}
for all $\varepsilon$, given $(z_\varepsilon)_\varepsilon$ as in \eqref{CondZ4}. By the estimate $\nu_{i,\varepsilon}^2 \Gamma_\varepsilon^p=O(1)$ just above \eqref{TildeWInter4} for $p>4/3$, we get that $\nu_{i,\varepsilon}^{3/2}=o(\Gamma_\varepsilon^{1-p_\varepsilon})$. Then, evaluating \eqref{MajorGlobal3} at $\tau \nu_{i,\varepsilon}$ and by \eqref{GammaDef4}, we get that
\begin{equation}\label{MajorGlobal3UsedHere}
\Gamma_\varepsilon\le\frac{2}{p_\varepsilon \gamma_{i,\varepsilon}^{p_\varepsilon-1}}\left(\ln \frac{1}{\lambda_\varepsilon \gamma_{i,\varepsilon}^{2(p_\varepsilon-1)} \nu_{i,\varepsilon}^2}+O(1) \right)+o\left(\Gamma_\varepsilon^{1-p_\varepsilon} \right)\,, 
\end{equation}
then with \eqref{RatioGamma4} that
$$\exp\left(\Gamma_\varepsilon^{p_\varepsilon}\right)\le\exp\left( \frac{2 \Gamma_\varepsilon^{p_\varepsilon-1}}{p_\varepsilon\gamma_{i,\varepsilon}^{p_\varepsilon-1}} \ln \frac{1}{\lambda_\varepsilon \gamma_{i,\varepsilon}^{2(p_\varepsilon-1}) \nu_{i,\varepsilon}^2}+o(1) \right) \,, $$
that
$${\lambda_\varepsilon \gamma_{i,\varepsilon}^{2(p_\varepsilon-1)} \nu_{i,\varepsilon}^2}\le \exp\left(-\frac{p_\varepsilon}{2}\Gamma_\varepsilon(1+o(1)) \gamma_{i,\varepsilon}^{p_\varepsilon-1} \right) $$
and eventually that
\begin{equation}\label{PartialCC4}
\lambda_\varepsilon \Gamma_\varepsilon^{p_\varepsilon-1} \nu_{i,\varepsilon}^2 \exp\left(\Gamma_\varepsilon^{p_\varepsilon} \right)=o(\Gamma_\varepsilon^{1-p_\varepsilon}) 
\end{equation}
for all $\varepsilon\ll 1$. By \eqref{MajorGrad4} and \eqref{BeforeLastStep4} with \eqref{PartialCC4}, we get that
$$|u_{i,\varepsilon}-\bar{u}_{i,\varepsilon}(R\nu_{i,\varepsilon})|\le (2\pi C_2+o(1)) \Gamma_{\varepsilon}^{1-p_\varepsilon}+O\left(\sum_{l\in I}\frac{1}{\gamma_{l,\varepsilon}} \ln \frac{3R \nu_{i,\varepsilon}}{|\cdot-\phi_{i,\varepsilon}(x_{l,\varepsilon})|}\right)$$
uniformly in $B_{R\nu_{i,\varepsilon}}(0)\backslash\{\phi_{i,\varepsilon}(x_{j_1,\varepsilon}),...,\phi_{i,\varepsilon}(x_{j_{|\mathcal{S}|},\varepsilon})\}$. In particular, using \eqref{RatioGamma4} again, we get  
\begin{equation}\label{Average4}
|\bar{u}_{i,\varepsilon}(\nu_{i,\varepsilon})-\bar{u}_{i,\varepsilon}(R\nu_{i,\varepsilon})|\le (2\pi C_2+o(1)) \Gamma_{\varepsilon}^{1-p_\varepsilon} 
\end{equation}
as $\varepsilon\to 0$. Then, $p_0=2$, \eqref{RatioGamma4}, \eqref{BeforeLastStep4}, \eqref{PartialCC4} and \eqref{Average4} give that
\begin{equation}\label{ContradFinal4}
\begin{split}
&|u_{i,\varepsilon}-\bar{u}_{i,\varepsilon}(\nu_{i,\varepsilon})|\\
&\quad \le \frac{9}{2}\pi C_2 \Gamma_\varepsilon^{1-p_\varepsilon}+\sum_{l\in I_{i,\varepsilon}(\nu_{i,\varepsilon})} \frac{2+o(1)}{\gamma_{l,\varepsilon}^{p_\varepsilon-1}} \ln \frac{4\tau \nu_{i,\varepsilon}}{|z_\varepsilon-\phi_{i,\varepsilon}(x_{l,\varepsilon})|} 
\end{split} 
\end{equation}
uniformly in $B_{\nu_{i,\varepsilon}}(0)\backslash\cup_{l\in I_{i,\varepsilon}(\nu_{i,\varepsilon})} B_{r_{l,\varepsilon}^{(\eta_0)}}\left(\phi_{i,\varepsilon}(x_{l,\varepsilon}) \right)$ and for all $\varepsilon$. But by \eqref{IneqNu4} for $l=i$, our assumption \eqref{ContradBdNu4} and by \eqref{LocConv3}, the inequality in \eqref{DefNuJEps4} for $j=i$ and $r=\nu_{i,\varepsilon}$ should be an equality somewhere on $\partial B_{\nu_{i,\varepsilon}}(0)$ of this set for all $\varepsilon\ll 1$, which gives a contradiction to \eqref{ContradFinal4} and concludes the proof of \eqref{BdNuJ4}.

\smallskip

 Then, picking now a sequence $(\tilde{\Gamma}_\varepsilon)_\varepsilon$ such that $\lim_{\varepsilon\to 0}\tilde{\Gamma}_\varepsilon= +\infty$ and $\tilde{\Gamma}_\varepsilon=o(\gamma_{j,\varepsilon})$, and setting 
$$\tilde{\nu}_{j,\varepsilon}=\inf\left\{r>0 \text{ s.t. }\bar{u}_{j,\varepsilon}\ge \tilde{\Gamma}_\varepsilon \text{ in }[0,r] \right\}\,, $$
we get from \eqref{BdNuJ4} that 
\begin{equation}\label{SmallProperty4}
\tilde{\nu}_{j,\varepsilon}\le \nu_{j,\varepsilon}
\end{equation}
 for all $j\in \{1,...,N\}$ and all $\varepsilon\ll 1$. By \eqref{Convloc3}, $\tilde{\nu}_{j,\varepsilon}=o(1)$. As in \eqref{UnifEquiv4}, we get from \eqref{WeakGradEst} and \eqref{LocConv3} that we can fix $0<R<1$ such that $u_\varepsilon=\tilde{\Gamma}_\varepsilon(1+o(1))$ uniformly in $\partial \phi_{j,\varepsilon}^{-1}(B_{R\tilde{\nu}_{j,\varepsilon}}(0))$ for all $\varepsilon\ll 1$ and all $j$.  Arguing now as below \eqref{Star4}, we get from \eqref{WeakPointwEst} that 
\begin{equation}\label{MaxPrin4}
\sup_{\Sigma\backslash \cup_{j} \phi_{j,\varepsilon}^{-1}\left(B_{R\tilde{\nu}_{j,\varepsilon}}(0)\right)} u_\varepsilon\le 2\tilde{\Gamma}_\varepsilon
\end{equation}
for all $\varepsilon\ll 1$. Then choose and fix $(\tilde{\Gamma}_\varepsilon)_\varepsilon$  growing slowly to $+\infty$ and more precisely such that
\begin{equation}\label{ConTildeGamma}
\begin{split}
&\lambda_\varepsilon \tilde{\Gamma}_\varepsilon^{p_\varepsilon} \exp\left((2 \tilde{\Gamma}_\varepsilon)^{p_\varepsilon} \right)=o\left(\gamma_{j,\varepsilon}^{2-p_\varepsilon}\right)\text{ and }\\
&\quad\quad\quad(2-p_\varepsilon)\ln \left(1+\lambda_\varepsilon \gamma_{j,\varepsilon}^{2(p_\varepsilon-1)} \exp((2\tilde{\Gamma}_\varepsilon)^{p_\varepsilon}) \right)=o(1) 
\end{split}
\end{equation}
for all $j$ as $\varepsilon\to 0$. The first condition is clearly possible by \eqref{Lim0Lambda4}. The second one is also possible since $\lambda_{\varepsilon} \gamma_{j,\varepsilon}^{2(p_\varepsilon-1)}=O(1)$ by \eqref{PosIneq4} and since now $p_0=2$ in \eqref{ConvP04}. We may now compute and use either \eqref{MaxPrin4} in $\Sigma\backslash \cup_{j} \phi_{j,\varepsilon}^{-1}(B_{R\tilde{\nu}_{j,\varepsilon}}(0))$, or the controls given by the inequality in \eqref{DefNuJEps4} for $r=\tilde{\nu}_{j,\varepsilon}$ thanks to  \eqref{SmallProperty4}, allowing to estimate the nonlinearity as in \eqref{ComputPower4}-\eqref{ComplSingSource4}. This leads to the following integral estimates:
\begin{equation}\label{Quantif14}
\begin{split}
&\frac{\lambda_\varepsilon p_\varepsilon^2}{2}\int_{\Sigma\backslash \cup_{j} \phi_{j,\varepsilon}^{-1}\big(B_{r_{j,\varepsilon}^{(\eta_0)} }\left(0\right)\big)} u_\varepsilon^{p_\varepsilon} e^{u_\varepsilon^{p_\varepsilon}} dx=o\left(\gamma_{j,\varepsilon}^{2-p_\varepsilon}\right)\,,\\
&\frac{\lambda_\varepsilon p_\varepsilon^2}{2}\int_{\Sigma\backslash \cup_{j} \phi_{j,\varepsilon}^{-1}\big(B_{r_{j,\varepsilon}^{(\eta_0)}}\left(0 \right)\big)} \( e^{u_\varepsilon^{p_\varepsilon}}-1\) dx=O\left(\lambda_\varepsilon\exp\left((2\tilde{\Gamma}_\varepsilon)^{p_\varepsilon} \right)\right)\,,
\end{split}
\end{equation}
while, computing as in \eqref{SingSource4}, we get that
\begin{equation}\label{Quantif24}
\begin{split}
&\frac{\lambda_\varepsilon p_\varepsilon^2}{2}\int_{\phi_{j,\varepsilon}^{-1}\big(B_{r_{j,\varepsilon}^{(\eta_0)}}\left(0 \right)\big)} u_\varepsilon^{p_\varepsilon} e^{u_\varepsilon^{p_\varepsilon}} dx=\left(4\pi+o(1) \right)\gamma_{j,\varepsilon}^{2-p_\varepsilon}\,,\\
&\frac{\lambda_\varepsilon p_\varepsilon^2}{2}\int_{\phi_{j,\varepsilon}^{-1}\big(B_{r_{j,\varepsilon}^{(\eta_0)}}\left(0 \right)\big)}\(  e^{u_\varepsilon^{p_\varepsilon}} -1\)dx=\frac{4\pi+o(1)}{\gamma_{j,\varepsilon}^{2(p_\varepsilon-1)}}
\end{split}
\end{equation}
as $\varepsilon\to 0$. {Thus, by plugging \eqref{Quantif14}-\eqref{Quantif24} in \eqref{BetaEps} and by using our conditions \eqref{ConTildeGamma} on $(\tilde{\Gamma}_\varepsilon)_\varepsilon$, we get that
\begin{equation*}
\begin{split}
\beta_\varepsilon~&~=\left(\sum_{j=1}^N \frac{4\pi+o(1)}{\gamma_{j,\varepsilon}^{2(p_\varepsilon-1)}} \right)^{\frac{2-p_\varepsilon}{p_\varepsilon}}\left((4\pi +o(1))\sum_{j=1}^N \gamma_{j,\varepsilon}^{2-p_\varepsilon} \right)^{\frac{2(p_\varepsilon-1)}{p_\varepsilon}}\,,\\
&~=4\pi (1+o(1)) \Bigg(\underset{(\star)}{\underbrace{1+\sum_{j\neq j_0} \left(\frac{\gamma_{j,\varepsilon}}{\gamma_{j_0,\varepsilon}} \right)^{2-p_\varepsilon}}} \Bigg)^{\frac{2(p_\varepsilon-1)}{p_\varepsilon}}\,,
\end{split} 
\end{equation*}
using that
$$\left(1+\sum_{j\neq j_0} \left(\frac{\gamma_{j_0,\varepsilon}}{\gamma_{j,\varepsilon}} \right)^{2(p_\varepsilon-1)} \right)^{\frac{2-p_\varepsilon}{p_\varepsilon}}=1+o(1) $$
since $p_\varepsilon\to 2$, where we choose $j_0\in\{ 1,...,N\}$ such that $\gamma_{j_0,\varepsilon}=\min_{j\in \{1,...,N\}} \gamma_{j,\varepsilon}$ for all $\varepsilon$, up to a subsequence. Then, in order to conclude the proof of \eqref{Quantization} for $k=N$, it is then sufficient to get that the term $(\star)$ converges to $N$, namely that
$$\forall j\in\{1,...,N\}\,,\quad \lim_{\varepsilon\to 0}(2-p_\varepsilon) \ln \frac{\gamma_{j,\varepsilon}}{\gamma_{j_0,\varepsilon}}=0\,. $$
To get this, we use \eqref{Lim0Lambda4}, \eqref{PosIneq4} and argue as below \eqref{ThisLastPoint} for $\eta=1/2$ to write
$$\left(2-p_\varepsilon \right)\gamma_{j,\varepsilon}^{p_\varepsilon}\le (1+o(1))\ln \frac{1}{\lambda_\varepsilon^2}\le  \frac{1+o(1)}{2} \gamma_{j,\varepsilon}^{p_\varepsilon} $$
for all $j$, so that $1\le \left(\gamma_{j,\varepsilon}/\gamma_{j_0,\varepsilon}\right)^{p_\varepsilon}=O\left(1/(2-p_\varepsilon) \right)\le +\infty$. Theorem \ref{ThmBlowUpAnalysis} is proven.}
\end{proof}

\section{Compactness at the critical levels $\beta\in 4\pi \mathbb{N}^\star$ for $p\in (1,2]$}
\noindent Our main goal in this section is to prove the following result:
\begin{thm}\label{ThmCritLevels}
Let $(\lambda_\varepsilon)_\varepsilon$ be any sequence of positive real numbers. Let $p\in (1,2]$ be given and set $p_\varepsilon=p$ for all $\varepsilon$. Let $(u_\varepsilon)_\varepsilon$ be a sequence of smooth functions solving \eqref{MainEqEps}. Let $(\beta_\varepsilon)_\varepsilon$ be given by \eqref{BetaEps}. Assume that \eqref{BlowUp3} holds true, so that \eqref{EnergyBound3} holds true for some $\beta\in 4\pi \mathbb{N}^\star$ (see Theorem \ref{ThmBlowUpAnalysis}). Then we have that
\begin{equation}\label{EqThmCritLevels}
\beta_\varepsilon>\beta
\end{equation}
for all $\varepsilon\ll 1$.
\end{thm}
More precisely, if $(\gamma_{1,\varepsilon})_\varepsilon\,,...\,,(\gamma_{k,\varepsilon})_\varepsilon$ are the sequences of positive real numbers diverging to $+\infty$ given by Proposition \ref{PropWeakPwEst}, we show in the proof below that
\begin{equation}\label{ExpBetaEps}
\beta_\varepsilon\ge 4\pi \left(k+\frac{4(p-1)(1+o(1))}{p^2}\sum_{i=1}^k \gamma_{i,\varepsilon}^{-2p}\right)
\end{equation}
as $\varepsilon\to 0$. As a remark, according to the proof of Theorem \ref{ThmBlowUpAnalysis}, $N$ in Proposition \ref{PropWeakPwEst} equals $k$ in \eqref{Quantization}. Interestingly enough, the cancellation of  terms of order $\gamma_{i,\varepsilon}^{-p}$  still occurs here on a surface for all $p\in (1,2]$ and for arbitrary energies, 
as pointed out in \cite{MalchMartJEMS} concerning the unit disk for $p=2$ and in the minimal energy case $\beta=4\pi$. 

\subsection{Further estimates in the radially symmetric case}\label{SubsectRadAnalysis5}
Let $p\in (1,2]$ be given, let $(\mu_\gamma)_\gamma$ be a family of positive real numbers, and let $(\lambda_\gamma)_\gamma$ be such that \eqref{BigLambdaDef1} holds true, where $p_\gamma=p$ for all $\gamma$, let $t_\gamma, \bar{t}_\gamma$ be given by \eqref{TGammaDef1} and let $(B_\gamma)_\gamma$ be given by \eqref{EqBubble1}. Let also $(\bar{r}_\gamma)_\gamma$ be a family of positive real numbers such that \eqref{RMuTo01} holds true, and such that 
\begin{align}
&\quad t_\gamma(\bar{r}_\gamma)\le \sqrt{\gamma}\,,\label{BarRNotTooLarge5}\\
&  \gamma^{4 p}\bar{r}_\gamma^2=O(1) \label{LpBd5}
\end{align}
for all $\gamma\gg 1$. In this section we aim to get more precise estimates on the $B_\gamma$'s than in Section \ref{SectRad}, but at smaller scales around $0$, in order to be technically as simple as possible: namely, \eqref{BarRNotTooLarge5}-\eqref{LpBd5} imply \eqref{BarRNotTooLarge1}-\eqref{LpBd1}. We also restrict here to the specific case where $p$ is fixed. As already mentioned in the introduction, some issues may arise when studying compactness at the critical levels $\beta\in 4\pi \mathbb{N}^\star$ in the case $p=1$. Following \cite{MalchMartJEMS,MartMan} and still abusing the radial notation $r=|x|$, we let $w_0$ be given by
$$w_0(r)=-T_0(r)+\frac{2 r^2}{1+r^2}-\frac{1}{2} T_0(r)^2+\frac{1-r^2}{1+r^2}\int_1^{1+r^2} \frac{\ln t}{1-t} dt $$
for $T_0$ as in \eqref{Liouville1}, so that, by ODE theory, $w_0$ is the unique solution of 
\begin{equation}\label{W0Eq5}
\begin{cases}
&\Delta w_0=4e^{-2 T_0}\left(2 w_0+T_0^2-T_0 \right)\text{ in }\mathbb{R}^2\,,\\
& w_0(0)=0\,,\\
&w_0\text{ is radially symmetric around }0\in \mathbb{R}^2\,.
\end{cases}
\end{equation}
We further set
\begin{equation}\label{FRHSEq5}
\begin{split}
F=~&~2(p-1)w_0+(p-2)T_0^2-8(p-1)T_0 w_0-\frac{8p-10}{3}T_0^3\\
&~+4\left(p-1 \right)w_0^2+4(p-1)T_0^2 w_0+(p-1)T_0^4\,,
\end{split}
\end{equation}
and let $w_1$ be the unique solution of
\begin{equation}\label{W1Eq5}
\begin{cases}
& \Delta w_1=4 e^{-2T_0}\left(2w_1+\frac{4(p-1)}{p^3}F \right)\text{ in }\mathbb{R}^2\,,\\
&w_1(0)=0\,,\\
& w_1\text{ is radially symmetric around }0\in \mathbb{R}^2\,.
\end{cases}
\end{equation}
Resuming the strategy and the explicit computations in \cite[Section 3]{MartMan}, even if we do not have $w_1$ in closed form, we know that
\begin{equation}\label{Integ5}
\begin{split}
\int_{\mathbb{R}^2} \Delta w_1 dx=\frac{16(p-1)}{p^3}\bigg[ & (p-1)\left(\frac{\pi^3}{3}+\frac{33 \pi }{2} \right)\\
& +\frac{3\pi}{2}\left(p-2\right)-\frac{7(4p-5)\pi }{2} \bigg]\,.
\end{split}
\end{equation}
We also have that
\begin{equation}\label{InftyW15}
\begin{split}
& w_0(r)=-T_0(r)+O(1)\,,\\
& w_1(r)=-\frac{T_0(r)}{4\pi} \int_{\mathbb{R}^2} \Delta w_1 dx+O(1)
\end{split}
\end{equation}
as $r\to +\infty$. Note that the convention on the sign of the Laplace operator here is not the same as that in \cite{MartMan}. In complement of the computations already done in \cite{MartMan}, we compute also
$$\int_{\mathbb{R}^2} \frac{|x|^2-1}{(1+|x|^2)^3} T_0(x)^2 dx=\frac{3 \pi}{2} $$
to get \eqref{Integ5}. Let $w_{0,\gamma}$, $w_{1,\gamma}$ be given by $w_{0,\gamma}=w_0(\cdot/\mu_\gamma)$ and $w_{1,\gamma}=w_1(\cdot/\mu_\gamma)$, and let $w_\gamma$ be given by
\begin{equation}\label{WGamma5}
B_\gamma=\gamma-\frac{2 t_\gamma}{p \gamma^{p-1}}+\frac{4(p-1) w_{0,\gamma}}{p^2 \gamma^{2p-1}}+\frac{w_{1,\gamma}+w_\gamma}{\gamma^{3p-1}}\,.
\end{equation}
Proposition \ref{PropRadAnalysis1} already gives $B_\gamma\le \gamma$ and some estimates on $w_\gamma$ given by \eqref{WGamma5} in $[0,\bar{r}_\gamma]$ for all $\gamma\gg 1$. Much more precisely here, we get that $w_\gamma$ is actually a small remainder term in the following sense:

\begin{prop}\label{PropRadAnalysis5}
We have 
\begin{equation*}
w_\gamma =O(\gamma^{-p} t_\gamma)\,,\quad w'_\gamma =O(\gamma^{-p} t'_\gamma)\,,
\end{equation*}
and
\begin{equation*}
\lambda_\gamma p  B_\gamma^{p-1}e^{B_\gamma^p}=-\frac{2}{p\gamma^{p-1}}\Delta t_\gamma\left(1+O\left(\frac{e^{t_\gamma/2}}{\gamma^{3p}} \right) \right)+\frac{4(p-1)}{p^2 \gamma^{2p-1}} \Delta w_{0,\gamma}+\frac{\Delta w_{1,\gamma}}{\gamma^{3p-1}}\,,
\end{equation*}
uniformly in $[0,\bar{r}_\gamma]$ and for all $\gamma\gg 1$ large, where $w_\gamma$ is as in \eqref{WGamma5}. 
\end{prop}

The proof of Proposition \ref{PropRadAnalysis5} follows the strategy of the proof of Proposition \ref{PropRadAnalysis1}, but the stronger assumption \eqref{BarRNotTooLarge5} basically reduces now the computations to  Taylor expansions. 

\begin{proof}[Proof of Proposition \ref{PropRadAnalysis5}] Let $r_\gamma$ be given by 
$$r_\gamma=\sup\left\{r\in [0,\bar{r}_\gamma]\text{ s.t. }|w_\gamma|\le t_\gamma \right\} $$ 
for all $\gamma$. Taking advantage of the control on $w_\gamma$ in $[0,r_\gamma]$ given by this definition, we may perform the following computations uniformly in $[0,r_\gamma]$ as $\gamma\to +\infty$. We  first get 
\begin{equation*}
\begin{split}
B_\gamma^p=~&~\gamma^p-2t_\gamma+\frac{2(p-1)}{p \gamma^p}\left(2w_{0,\gamma}+t_\gamma^2\right)+\frac{p(w_{1,\gamma}+w_\gamma)}{\gamma^{2p}}\\
&\quad-\frac{8(p-1)^2 t_\gamma w_{0,\gamma}}{p^2 \gamma^{2p}}-\frac{8(p-1)(p-2)t_\gamma^3}{6p^2 \gamma^{2p}}+O\left(\frac{\bar{t}_\gamma^4}{\gamma^{3p}} \right)\,.
\end{split} 
\end{equation*}
and 
$$ B_\gamma^{p-1}=\gamma^{p-1}\left(1-\frac{2(p-1) t_\gamma}{p \gamma^p}+\frac{4(p-1)^2 w_{0,\gamma}}{p^2 \gamma^{2p}}+\frac{2(p-1)(p-2)t_\gamma^2}{p^2 \gamma^{2p}}+O\left(\frac{\bar{t}_\gamma^3}{\gamma^{3p}} \right) \right)\,. $$
We use for this \eqref{BarRNotTooLarge5} and \eqref{InftyW15}  and the expansions of $(1+\varepsilon)^q$ as $\varepsilon\to 0$. Then, using \eqref{BigLambdaDef1}, we similarly compute and get
\begin{equation*}
\begin{split}
&\lambda_\gamma p  e^{B_\gamma^p}\\
&=\frac{8 e^{-2t_\gamma}}{p \gamma^{2(p-1)}\mu_\gamma^2}\bigg[1+\frac{2(p-1)}{p \gamma^p}\left(2w_{0,\gamma}+t_\gamma^2\right)+\\
&\quad\frac{1}{p^2 \gamma^{2p}}\bigg(p^3(w_{1,\gamma}+w_\gamma)-8(p-1)^2 t_\gamma w_{0,\gamma}-\frac{4}{3}(p-1)(p-2)t_\gamma^3\\
&\quad +8(p-1)^2 t_\gamma^2 w_{0,\gamma}+2(p-1)^2 t_\gamma^4+8(p-1)^2 w_{0,\gamma}^2 \bigg)+O\bigg(\frac{1}{\gamma^{3p}}e^{C \bar{t}_\gamma\times \left(\frac{\bar{t}_\gamma^3}{\gamma^{3 p}} \right)} \bigg) \bigg]\,,
\end{split}
\end{equation*}
so that we eventually have
\begin{equation}\label{NonlinExpan5}
\begin{split}
&\lambda_\gamma p  B_\gamma^{p-1}e^{B_\gamma^p}\\
&=\frac{8 e^{-2t_\gamma}}{p \gamma^{p-1}\mu_\gamma^2}
\bigg[1+\frac{2(p-1)}{p \gamma^p}\left(2w_{0,\gamma}+t_\gamma^2-t_\gamma\right)+O\bigg(\frac{e^{t_\gamma/2}}{\gamma^{3p}} \bigg)\bigg]+\\
&\quad\quad\quad\frac{4 e^{-2t_\gamma}}{\gamma^{3p-1}\mu_\gamma^2}\bigg[2(w_{1,\gamma}+w_\gamma)+\frac{4(p-1)}{p^3}F\left(\frac{\cdot}{\mu_\gamma} \right)  \bigg]\,,
\end{split}
\end{equation}
for $F$ as in \eqref{FRHSEq5}, using again \eqref{BarRNotTooLarge5} to write ${\bar{t}_\gamma^3}/{\gamma^{3 p}}=o(1)$. Then, setting $\tilde{w}_\gamma=w_\gamma(\cdot/\mu_\gamma)$, using now not only \eqref{Liouville1}, but also \eqref{W0Eq5} and \eqref{W1Eq5}, we get from \eqref{EqBubble1} that 
\begin{equation}\label{EqTildeW5}
\Delta \tilde{w}_\gamma=8e^{-2 T_0} \tilde{w}_\gamma+ O\left(\mu_\gamma^2\gamma^{3 p} \right)+O\left(\frac{e^{-3T_0/2}}{\gamma^{p}} \right)\,, \end{equation}
uniformly in $[0,r_\gamma/\mu_\gamma]$ as $\gamma\to +\infty$, applying $\Delta$ to \eqref{WGamma5}. The second-last term in \eqref{EqTildeW1} is obtained when controlling $B_\gamma$ in the LHS of \eqref{EqBubble1}, since our definition of $r_\gamma$ implies $B_\gamma\le \gamma$ in $[0,r_\gamma]$ for all $\gamma\gg 1$. Then, \eqref{TildeW1Estim1} may be obtained from \eqref{EqTildeW1} by using also \eqref{LpBd5}. At that stage, we may conclude the proof of Proposition \ref{PropRadAnalysis5} by following closely the lines below \eqref{TildeW1Estim1} and showing mainly that \eqref{BarREqR1} holds true for all $\gamma\gg 1$.
\end{proof}
\noindent As a direct corollary of Proposition \ref{PropRadAnalysis5}, we get the following estimates:
\begin{cor}\label{CorRadLastPart5}
Assume that \eqref{BarRNotTooLarge5} is an equality, namely that
\begin{equation}\label{BarREquality5}
t_\gamma(\bar{r}_\gamma)=\sqrt{\gamma}
\end{equation}
for all $\gamma\gg 1$, then we have that
\begin{equation*}
\begin{split}
&\frac{\lambda_\gamma p^2}{2} \int_{B_{\bar{r}_\gamma}(0)} B_\gamma^p e^{B_\gamma^p} dx\\
&=4\pi \gamma^{2-p}\bigg[1+\frac{2(p-2)}{p\gamma^p}+o\left(\frac{1}{\gamma^{2p}} \right)\\
&
+\frac{p-1}{p^2 \gamma^{2p}}\bigg(-8-\frac{2\pi^2}{3}+2(p-1)\left(\frac{\pi^2}{3}+\frac{33}{2} \right)+3(p-2)-7\left(4p-5 \right) \bigg)\bigg]\,,\\
&\frac{\lambda_\gamma p^2}{2} \int_{B_{\bar{r}_\gamma}(0)} e^{B_\gamma^p} dx\\
&=\frac{4\pi}{\gamma^{2(p-1)}}\bigg[1+\frac{4(p-1)}{p\gamma^p}+o\left(\frac{1}{\gamma^{2p}} \right)\\
&\quad\quad\quad+\frac{1}{\gamma^{2p}}\bigg(\frac{2(p-1)}{p^2}\left((p-1)\left(\frac{\pi^2}{3}+\frac{33}{2} \right)+\frac{3}{2}(p-2)-\frac{7(4p-5)}{2} \right)\\
&\quad\quad\quad\quad\quad+\frac{4(p-1)}{p}+\frac{(p-1)^2}{p^2}\left(8+\frac{2\pi^2}{3} \right) \bigg) \bigg]\,,
\end{split}
\end{equation*}
and then that
\begin{equation}\label{EnergyComputations5}
\begin{split}
& \left(\frac{\lambda_\gamma p^2}{2} \int_{B_{\bar{r}_\gamma}(0)} e^{B_\gamma^p} dx\right)^{\frac{2-p}{p}} \left(\frac{\lambda_\gamma p^2}{2} \int_{B_{\bar{r}_\gamma}(0)} B_\gamma^p e^{B_\gamma^p} dx \right)^{\frac{2(p-1)}{p}}\\
&\quad \quad=4\pi \left(1+ \frac{4(p-1)}{p^2 \gamma^{2p}}+o\left(\frac{1}{\gamma^{2p}} \right) \right)
\end{split}
\end{equation}
as $\gamma\to +\infty$.
\end{cor}
Since the computations to get Corollary \ref{CorRadLastPart5} from Proposition \ref{PropRadAnalysis5} basically resume those in \cite{MartMan}, we leave them to the reader. In particular, proving the  first two estimates in Corollary \ref{CorRadLastPart5} uses \eqref{Integ5} and the following computations 
\begin{equation*}
\begin{split}
&\int_{\mathbb{R}^2} \Delta w_0 dx =- \int_{\mathbb{R}^2} \Delta T_0 dx=-\int_{\mathbb{R}^2} T_0 \Delta T_0 dx=-\frac{1}{2}\int_{\mathbb{R}^2} T_0^2 \Delta T_0 dx=4\pi\,,\\
& \int_{\mathbb{R}^2} \left(w_0 (\Delta T_0)+T_0\Delta w_0 \right) dx=8\pi+\frac{2\pi^3}{3}\,.
\end{split}
\end{equation*}
Once the first two estimates of Corollary \ref{CorRadLastPart5} are obtained, proving \eqref{EnergyComputations5} is quite elementary: in particular, we observe in \eqref{EnergyComputations5} the aforementioned cancellation of the term $\gamma^{-p}$. Besides, the term $\gamma^{-2p}$ vanishes as well for $p=1$. That is the technical reason why the approach of this section does not work for $p=1$ and why we assume $p>1$ in Theorem \ref{ThmCritLevels} (see also the paragraph {above Remark \ref{RemCompactPathPEq1To2}}).

\subsection{Conclusion of the proof of Theorem \ref{ThmCritLevels}}
Let $(\lambda_\varepsilon)_\varepsilon$ be any sequence of positive real numbers. Let $p\in (1,2]$ be given and set $p_\varepsilon=p$ for all $\varepsilon$. Let $(u_\varepsilon)_\varepsilon$ be a sequence of smooth functions solving \eqref{MainEqEps}. Let $(\beta_\varepsilon)_\varepsilon$ be given by \eqref{BetaEps}. Assume that \eqref{BlowUp3} holds true, so that \eqref{EnergyBound3} holds true for some $\beta\in 4\pi \mathbb{N}^\star$ by Theorem \ref{ThmBlowUpAnalysis}. We may also apply Proposition \ref{PropWeakPwEst}, getting in particular sequences $(\mu_{i,\varepsilon})_\varepsilon$, $(x_{i,\varepsilon})_\varepsilon$, $(\gamma_{i,\varepsilon})_\varepsilon$ and $(\varphi_{i,\varepsilon})_\varepsilon$, and we resume the notation $r_{i,\varepsilon}$, $t_{i,\varepsilon}$ and $v_{i,\varepsilon}$ in \eqref{DefRIEps3}-\eqref{DefVI3}; let also $\bar{r}_{i,\varepsilon}$ be given by
\begin{equation}\label{BarRDef5}
t_{i,\varepsilon}(\bar{r}_{i,\varepsilon})=\sqrt{\gamma_{i,\varepsilon}}
\end{equation} 
for all $i\in \{1,...,k\}$ and all $\varepsilon$. By \eqref{FarEnoughEq3} in Step \ref{StFarEnough3}, we know that $\bar{r}_{l,\varepsilon}^{(1/2)}$ given by \eqref{BarRDef3} equals $r_{l,\varepsilon}^{(1/2)}$ in \eqref{DefRiEpsEta3} for all $l\in \{1,...,k\}$ and all $\varepsilon\ll 1$. Moreover, since $r_{i,\varepsilon}=O(1)$ according to \eqref{DefRIEps3}, we get that
\begin{equation}\label{BigOOfOne5}
r_{i,\varepsilon}^{(1/2)}=o(r_{i,\varepsilon})=o(1)
\end{equation}
for all $\varepsilon\ll 1$ and all $i$. By \eqref{DefRiEpsEta3} and \eqref{BarRDef5}, we deduce that
$$\ln\frac{\bar{r}_{i,\varepsilon}^2}{\left(r_{i,\varepsilon}^{(1/2)}\right)^2}=t_{i,\varepsilon}(\bar{r}_{i,\varepsilon})-t_{i,\varepsilon}\left(r_{i,\varepsilon}^{(1/2)}\right)+o(1)\le -3\gamma_{i,\varepsilon} $$
for all $i$ and all $\varepsilon\ll 1$. Then, we find from \eqref{BigOOfOne5} that
\begin{equation}\label{BarRSmall5}
\bar{r}_{i,\varepsilon}=O\left(e^{-\gamma_{i,\varepsilon}} \right)
\end{equation}
for all $\varepsilon\ll 1$ and all $i$. Proposition \ref{PropRadAnalysis1} may be applied as below \eqref{BarRDef3}.  We get that
\begin{equation*}
|u_{i,\varepsilon}-v_{i,\varepsilon}|=O\left(\frac{\bar{r}_{i,\varepsilon}}{r_{i,\varepsilon}^{(1/2)} \gamma_{i,\varepsilon}^{p-1}} \right)=O\left(e^{-\gamma_{i,\varepsilon}} \right)  
\end{equation*}
uniformly in $B_{\bar{r}_{i,\varepsilon}}(0)$ for all $\varepsilon\ll 1$ and all $i$, using also \eqref{CorSect23}. Then, using similarly Proposition \ref{PropRadAnalysis1} to get that $v_{i,\varepsilon}=\gamma_{i,\varepsilon}(1+o(1))$, we obtain that 
\begin{equation}\label{ComparInterRadFin5}
u_{i,\varepsilon}^p=v_{i,\varepsilon}^p\left(1+O\left(\frac{e^{-\gamma_{i,\varepsilon}}}{\gamma_{i,\varepsilon}} \right) \right)\,, 
\end{equation}
so that we have
\begin{equation}\label{ComparRadFin5}
e^{u_{i,\varepsilon}^p}=e^{v_{i,\varepsilon}^p}\left(1+o\left(\frac{1}{\gamma_{i,\varepsilon}^{2p}} \right) \right)
\end{equation}
uniformly in $B_{\bar{r}_{i,\varepsilon}}(0)$, for all $\varepsilon\ll 1$ and all $i$. An easy consequence of \eqref{ConVarphi3}, \eqref{DefRIEps3} and \eqref{BigOOfOne5} is that the domains $\phi_{i,\varepsilon}^{-1}(B_{\bar{r}_{i,\varepsilon}}(0))$ are two by two disjoint for all $\varepsilon\ll 1$. Then we may write that
\begin{equation}\label{UsPosToMinor5}
\begin{split}
&\frac{\lambda_\varepsilon p^2}{2} \int_\Sigma u_\varepsilon^{p} e^{u_\varepsilon^p} dv_g\ge  \sum_{i=1}^k \underset{:=a_{i,\varepsilon}}{\underbrace{\frac{\lambda_\varepsilon p^2}{2} \int_{B_{\bar{r}_{i,\varepsilon}}(0)} u_{i,\varepsilon}^{p} e^{u_{i,\varepsilon}^p} e^{2\varphi_{i,\varepsilon}} dx}}\,,\\
&\frac{\lambda_\varepsilon p^2}{2} \int_\Sigma \(e^{u_\varepsilon^p} -1\)dv_g\ge  \sum_{i=1}^k \underset{:=b_{i,\varepsilon}}{\underbrace{\frac{\lambda_\varepsilon p^2}{2} \int_{B_{\bar{r}_{i,\varepsilon}}(0)} \(e^{u_{i,\varepsilon}^p}-1\) e^{2\varphi_{i,\varepsilon}} dx}}\,.
\end{split}
\end{equation}
Using \eqref{ConVarphi3}, \eqref{BarRSmall5}, \eqref{ComparInterRadFin5} and \eqref{ComparRadFin5}, we write $e^{2\varphi_{i,\varepsilon}}=1+O\left(\bar{r}_{i,\varepsilon} \right)$ and get
\begin{equation}\label{Suirte5}
\int_{B_{\bar{r}_{i,\varepsilon}}(0)} u_{i,\varepsilon}^{p} e^{u_{i,\varepsilon}^p} e^{2\varphi_{i,\varepsilon}} dx=\left(\int_{B_{\bar{r}_{i,\varepsilon}}(0)} v_{i,\varepsilon}^p e^{v_{i,\varepsilon}^p} dx\right) \left(1+o\left(\gamma_{i,\varepsilon}^{-2p} \right) \right)\,, 
\end{equation}
for all $\varepsilon\ll 1$ and all $i$. Similar arguments give that
\begin{equation}\label{SuirteBis5}
\int_{B_{\bar{r}_{i,\varepsilon}}(0)} \(e^{u_{i,\varepsilon}^p}-1\) e^{2\varphi_{i,\varepsilon}} dx=\left(\int_{B_{\bar{r}_{i,\varepsilon}}(0)}  e^{v_{i,\varepsilon}^p} dx\right) \left(1+o\left(\gamma_{i,\varepsilon}^{-2p} \right) \right)\,, 
\end{equation}
for all $\varepsilon\ll 1$ and all $i$. By plugging \eqref{Suirte5}-\eqref{SuirteBis5} in \eqref{UsPosToMinor5} and coming back to the definition \eqref{BetaEps}, we obtain 
$$\beta_\varepsilon\ge \left(\sum_{i=1}^k b_{i,\varepsilon}  \right)^{\frac{2-p}{p}}\left(\sum_{i=1}^k a_{i,\varepsilon}\right)^{\frac{2(p-1)}{p}}\ge \sum_{i=1}^k b_{i,\varepsilon}^{\frac{2-p}{p}} a_{i,\varepsilon}^{\frac{2(p-1)}{p}}\,, $$
by H\"older's inequality for vectors in $\mathbb{R}^k$. In order to compute the RHS, since we have \eqref{BarRSmall5}, so that (see \eqref{LpBd5}) we may apply also Proposition \ref{PropRadAnalysis5} to $v_{i,\varepsilon}$ in $B_{\bar{r}_{i,\varepsilon}}(0)$ and thus use \eqref{EnergyComputations5}. This proves \eqref{ExpBetaEps} and concludes the proof of Theorem \ref{EqThmCritLevels}. 

\begin{rem}\label{RkExtremals}
The minimization of $I_\beta$ in \eqref{EnergyZeroAv} for $\beta=4\pi$ attracted some attention (see for instance  \cite{DJLWA,NolascoTarantello}): in this case we basically have $p=1$. Then, turn now to the case $p\in (1,2]$ of this section. First if $p=2$, we may get by following the strategy in \cite{MartMan} that the convergence of $(\beta_\varepsilon)_\varepsilon$ to $4\pi$ from above in \eqref{ExpBetaEps} for $k=1$ gives back the existence of a maximizer for \eqref{MTSurf} if $\beta=4\pi$ (see also \cite{CarlesonChang,StruweCrit,Flucher}). Now, if $p\in (1,2)$, we already pointed out in the introduction that 
$$-\infty<{\Theta_{p,\varepsilon}}:=\inf_{u\in H^1} J_{p,4\pi (1-\varepsilon)}(u) $$   
for all $\varepsilon\in [0,1)$, where $J_{p,\beta}$ is as in \eqref{Energy}. Moreover, the existence of a minimizer $u_\varepsilon$ for ${J_{p,4\pi (1-\varepsilon)}}$ follows from a standard minimization argument for all given $\varepsilon\in (0,1)$. Here again, the convergence of $(\beta_\varepsilon)_\varepsilon$ to $4\pi$ from above in \eqref{ExpBetaEps} for $k=1$ gives the {attainment of $\Theta_{p,0}$}, since the present $u_\varepsilon$'s then have to converge strongly in $C^2$ as $\varepsilon\to 0$. 

\medskip

We conclude this remark by a curiosity. If $G>0$ is the Green's function of $\Delta_g+{h}$ in $\Sigma$, we may write $G(x,y)=\frac{1}{4\pi}\left(\ln \frac{1}{|x-y|^2}+\mathcal{H}(x,y) \right)$ for all $x\neq y$. We know that $\mathcal{H}\in C^0(\Sigma\times \Sigma)$ and we set $M=\max_{x\in \Sigma} \mathcal{H}(x,x)$. As a byproduct of the analysis in the present paper, it can be also checked that
$$\ln \frac{1}{\lambda_\varepsilon}=\left(1-\frac{p}{2} \right) \gamma_\varepsilon^{p}+\ln \frac{p^2 \gamma_\varepsilon^{2(p-1)}}{8}+H_x(x)+(p-1)+o(1)\,,$$
as $\varepsilon\to 0$, if the $u_\varepsilon$'s blow-up at some $x\in \Sigma$ for $k=1$ in \eqref{Quantization} and solve \eqref{MainEqEps}, with $\lambda_\varepsilon$ given by \eqref{BetaEps}, for $\beta_\varepsilon=4\pi(1-\varepsilon)$, $p_\varepsilon=p$ and $\gamma_\varepsilon=\max_{\Sigma} u_\varepsilon$ for all $\varepsilon$. We may also get that
\begin{equation}\label{Curiosity}
{\Theta_{p,0}}=\inf_{u\in H^1} J_{p,4\pi}(u)<-\left(\ln \pi +M+\left(p-1 \right)+\frac{(2-p)(p-1)}{p} \right)\,.
\end{equation}
The large inequality in \eqref{Curiosity} is a byproduct of a by now rather standard test function computations (see for instance \cite[Step 3.1]{WhenExtremals}). The strict inequality is more subtle and can be seen as a consequence of the convergence of the $\beta_\varepsilon$'s from above, picking the \textbf{refined test functions provided by the blow-up analysis}, in the spirit of \cite[Section 4]{WhenExtremals}. At last, observe that the exponential of the opposite of the RHS of \eqref{Curiosity} converges to $\pi\exp\left(1+M \right)$ as $p\to 2$, which turns out to be consistent with the original works \cite{CarlesonChang,Flucher}.
\end{rem}

\section*{Conclusion of the proofs of Theorems \ref{MainThm} and \ref{MainThm2}} Let $\beta>0$ be given. Assume first that $p$ is given in $(1,2)$. By Theorem \ref{ThmVariationalPart}, there exist a sequence $(\beta_\varepsilon)_\varepsilon$ increasing to $\beta^-$ as $\varepsilon\to 0$, and $u_\varepsilon$ such that \eqref{MainEqEps} is satisfied for $p_\varepsilon=p$ and $\lambda_\varepsilon$ given by \eqref{BetaEps} for all $\varepsilon$. Now, we claim that the $u_\varepsilon$'s are uniformly bounded: this is a direct consequence of \eqref{Quantization} in Theorem \ref{ThmBlowUpAnalysis} if $\beta\not \in 4\pi \mathbb{N}^\star$ and follows from Theorem \ref{ThmCritLevels} if $\beta \in 4\pi \mathbb{N}^\star$, since the present sequence $(\beta_\varepsilon)_\varepsilon$ is assumed to increase. By elliptic theory in \eqref{MainEqEps} and \eqref{BdLambdaEps3}, we easily then get that, up to a subsequence, the $\lambda_\varepsilon$'s converge to some $\lambda$ and the $u_\varepsilon$'s converge in $C^2$ to some $u$ solving the equation in \eqref{MainEquation} and \eqref{FormulaLambda}. Observe in particular that since $\beta>0$, \eqref{FormulaLambda} gives that $u\ge 0$ is not identically zero, so that $u>0$ in $\Sigma$ by Lemma \ref{l:pos}. Then ${\mathcal{C}_{p,\beta}}\ni u$ is not empty in Theorem \ref{MainThm2}. The compactness of ${\mathcal{C}_{p,\beta}}$ also clearly follows from Theorems \ref{ThmBlowUpAnalysis} and \ref{ThmCritLevels}. For $p=1$, and $\beta\not\in 4\pi \mathbb{N}^*$, we take a sequences $(p_\ve),~ p_\ve\downarrow 1$ and $u_\ve\in \mathcal{C}_{p_\ve,\beta}$. As before, by Theorem \ref{ThmBlowUpAnalysis}, up to a subsequence $(u_\ve)$ converges to a positive function $u\in \mathcal{C}_{1,\beta}$, and  Theorem \ref{MainThm2} is proven. Assume now that $p=2$. By Theorem \ref{ThmVariationalPart} again, there exist a sequence $(\beta_\varepsilon)_\varepsilon$ increasing to $\beta^-$, a sequence $(p_\varepsilon)_\varepsilon$ increasing to $2^-$ as $\varepsilon\to 0$, and $u_\varepsilon$ such that \eqref{MainEqEps} is satisfied for $\lambda_\varepsilon$ given by \eqref{BetaEps} for all $\varepsilon$. First, if we have in addition $\beta\not \in 4\pi \mathbb{N}^\star$, we get similarly from Theorem \ref{ThmBlowUpAnalysis} that, up to a subsequence, the $\lambda_\varepsilon$'s converge to some $\lambda$ and the $u_\varepsilon$'s converge in $C^2$ to some $u$ solving the equation in \eqref{ELMainEq} and {\eqref{FormulaLambda}}. Then, we use again that $\beta$ is positive to get from {\eqref{FormulaLambda}} that $u$ is actually positive in $\Sigma$ and then that $u\in {\mathcal{C}_{2,\beta}}$. Thus, if we have now $\beta \in 4\pi \mathbb{N}^\star$, setting $\beta_\varepsilon=\beta-\varepsilon$ and $p_\varepsilon=2$, there exists $u_\varepsilon$ such that \eqref{MainEqEps} is satisfied for $\lambda_\varepsilon$ given by \eqref{BetaEps} for all $0<\varepsilon\ll 1$. By Theorem \ref{ThmCritLevels}, we similarly get that the $u_\varepsilon$'s converge in $C^2$ to some $u\in {\mathcal{C}_{2,\beta}}$ solving {\eqref{FormulaLambda}}-\eqref{ELMainEq}, up to a subsequence. The compactness of ${\mathcal{C}_{2,\beta}}$ follows from Theorems \ref{ThmBlowUpAnalysis} and \ref{ThmCritLevels} again, which concludes the proof of Theorem \ref{MainThm} in any case.

\nocite{DelPBeyond}\nocite{MasNakSan}

\end{document}